\documentclass[11pt]{article}
\usepackage[psamsfonts]{amsfonts}
\usepackage{amsmath,amssymb,amsthm,dsfont,enumerate,mathrsfs}
\usepackage{graphicx}               %voor het in de tekst opnemen van plaatjes
\usepackage{xcolor}
\usepackage{dsfont}
\usepackage{subfig}
\usepackage{float}
\usepackage{hyperref}

\usepackage{titleps}
\newpagestyle{classica}{%
	\sethead{}{}{\thepage}
}
\newtheorem{theorem}{Theorem}[section]

\newtheorem{definition}[theorem]{Definition}
\newtheorem{lemma}[theorem]{Lemma}
\newtheorem{corollary}[theorem]{Corollary}

\newtheorem{conjecture}[theorem]{Conjecture}

\newtheorem*{remark}{Remark}

\newcommand{\be}{\begin{equation}}
\newcommand{\ee}{\end{equation}}

\newcommand{\pl}{\eta_i^+}
\newcommand{\mi}{\eta_i^-}
\newcommand{\dd}{\mathrm{d}}
\newcommand{\al}{\alpha}
\newcommand{\Lai}{{\Lambda\backslash i}}

\newcommand{\atanh}{{\rm atanh}}

\newcommand{\arccotanh}{{\rm arccotanh}}

\newcommand{\plO}{\omega_{i_0}^{+,i}}
\newcommand{\miO}{\omega_{i_0}^{-,i}}
\numberwithin{equation}{section}

\DeclareMathOperator{\lab}{lab}
\DeclareMathOperator{\inter}{int}
\DeclareMathOperator{\type}{type}
\DeclareMathOperator{\cotanh}{cotanh}
\DeclareMathOperator{\ex}{ex}
\newcommand{\interior}[1]{%
	{\kern0pt#1}^{\mathrm{o}}%
}

\def\a{\alpha}

\def\Z{\mathbb{Z}}
\def\R{\mathbb{R}}
\def\N{\mathbb N}

\title{Dynamical Gibbs-non-Gibbs transitions in lattice 
	Widom-Rowlinson models with hard-core and soft-core interactions}

\author
{
 Sascha Kissel\footnote{
   Ruhr-Universit\"at Bochum, Fakult\"at f\"ur Mathematik, Universit\"atsstra\ss e 150, 44780 Bochum, Germany.
   E-mail: {\tt sascha.kissel@ruhr-uni-bochum.de},  {\tt christof.kuelske@ruhr-uni-bochum.de}
 }
 \and Christof K\"ulske\footnotemark[\value{footnote}] 
 }

\date{\today}

\begin{document}
	\maketitle 
	
	\begin{abstract} We consider the Widom-Rowlinson model 
		on the lattice $\Z^d$ in two versions, comparing the cases of a hard-core repulsion and of a 
		soft-core repulsion between particles 
		carrying opposite signs. 
		For both versions we investigate their 
		dynamical Gibbs-non-Gibbs transitions under an independent 
		stochastic symmetric spin-flip dynamics.  
		While both models have a similar phase transition in the high-intensity regime in equilibrium, 
		we show that they behave differently under time-evolution: 
		The time-evolved soft-core model is Gibbs for small times and loses the Gibbs property 
		for large enough times. 
		By contrast, the time-evolved hard-core model loses the Gibbs property immediately, and   
		for asymmetric intensities, shows a transition back to the Gibbsian regime at a sharp transition time. 
		%
		%
		%
		%
		%
		%
		%which arises by weakening the Hard-Core constraint of the discrete Widom-Rowlinson model. This model is called SOft-Core Widom-Rowlinson model where $+$ and $-$ particles are allowed to be connected but will be punished by some constant $\beta$.\\
		% We are investigating the behaviour of pressure of this model and it is qualitatively the same as in the Curie-Weiss model. The symmetric model has a second order phase transition and we give a sharp value of the temperature where the transition occurs. This will be done by a large deviation analysis. \\
		% In the second part of this work we consider an independent stochastic spin-flip of the soft-core model. We establish that the time evolved measure is Gibbs for small times or hight temperature.

	\end{abstract}
	\vspace{0.3cm} 
	\newpage
	\section{Introduction}
		Dynamical Gibbs-non{-}Gibbs transitions have attracted much attention over the last 
		years. This started from an investigation of the Ising model under a high-temperature 
		Glauber time-evolution on the lattice in 
		\cite{enter-fernandez-hollander-redig02}. It was found that, in zero external magnetic field, 
		the Gibbs property is lost at a finite transition time, after which the measure continues
		to be non-Gibbsian. The loss of the Gibbs property is indicated by a very 
		long-range (discontinuous) dependence 
		of finite-volume conditional probabilities. When such discontinuities occur they 
		are related to a hidden phase transition of an internal system which provides a mechanism 
		to carry the influence of variations of boundary conditions over very long distances. As there 
		are model-dependent different mechanisms of such phase-transitions, also a  variety of types of associated 
		Gibbs-non{-}Gibbs transitions 
		may occur. 
		For more related work on dynamical Gibbs-non{-}Gibbs transitions in a Glauber-evolved 
		Ising model, and beyond, see \cite{fernandez-pfister97},\cite{enter-fernandez-sokal93},\cite{kuelske-le07},\cite{kraaij-redig-zuijlen17}.
		
		The present paper is an essential piece in a series of investigations 
		in which we study Gibbs-non-Gibbs transitions of the {\em Widom-Rowlinson model} 
		under stochastic spin-flip-dynamics in {\em various geometries}. 	
		The Widom-Rowlinson model is{,} in its original form \cite{widom-rowlinson70}{,} a model for 
		point particles in Euclidean space which carry a plus-sign or a minus-sign, and 
		which interact via a hardcore repulsion which forbids particles of opposite 
		sign to become closer than a fixed radius. 
		It is one of the simplest continuum models for which  a phase transition has been proved \cite{chayes-chayes-kotecky95}, and analyzed. 
		An investigation of the Euclidean hard-core Widom-Rowlinson model under 
		a stochastic spin-flip dynamics was given in \cite{jahnel-kuelske17c, jahnel-kuelske17a}. 
		In this work a strong form of non-Gibbsian 
		behavior{,} which appears to be more severe than for instance 
		in the case of the Ising model{,} was found, 
		including full measure discontinuities 
		of the time-evolved conditional probabilities, and an immediate loss of the Gibbs property. 
		The latter is quite unusual for a lattice model, see however the examples in mean-field  
		\cite{hollander-redig-zuijlen15},  on a  tree \cite{enter-ermolaev-Iacobelli-kuelske12}, and for a transformed measure not coming from a time-evolution in \cite{kuelske-le-redig04}
		
		Motivated by  the strong anomalies which occur for the Widom-Rowlinson model in continuum, one becomes interested in the behavior of the model 
		in other geometries:  as a  mean-field model, 
		on the lattice, on a tree, on more general graphs, 
		or in a long-range Kac-version.  
		For a recent overview, see \cite{kuelske19}. 
		
		In the present paper we focus on the Widom-Rowlinson model on the integer lattice, where we 
		treat and compare two versions. 
		The hard-core version comes with a {\em hard-core constraint} which forbids 
		particles of opposite sign to occupy neighboring lattice sites (see also \cite{gallavotti-lebowitz71,  higuchi-takei04}), the soft-core version comes with a soft  
		constraint where such pairs of opposite signs are not strictly forbidden, only energetically 
		disfavored with a repulsion constant $\beta$.
		% {The hard-core model was first investigated by \cite{gallavotti-lebowitz71} where also a phase transition in the %absence of a external magnetic field has been observed. While for the Ising model \emph{with an external field} on the %lattice there is a unique Gibbs measure, which can be proven by the Lee-Yang theorem \cite{lee-yang52}, it is still %unclear if there is a phase transition for the hard-core Widom-Rowlinson model with external field. In \cite{higuchi-%takei04} results according to this problem can be found.}
		The soft-core model has a mean-field analogue which was analyzed in \cite{kissel-kuelske18}, where 
		the loss of the sequential Gibbs property under a stochastic independent spin-flip dynamics 
		was found, at a finite transition time.  
		A closed solution for the equilibrium model was given and it was shown that the 
		sets of bad empirical measures (discontinuity points of a limiting 
		specification kernel) consist of finitely many curves which evolve with time. 
		For the Widom-Rowlinson model on a Cayley tree so far there are detailed equilibrium 
		results (see \cite{rozikov-book} for the hard-core model, and \cite{kissel-kuelske-rozikov19} 
		for the soft-core model), but no dynamical results yet. 
		
		The remainder of the paper is organized as follows. In {Section} \ref{Sec: 2} we introduce equilibrium models, time-evolution and state 
		our results. {In Section 3 and Section 4} the proofs are found. 
		Theorem \ref{thm : Sc Phase} of {Subsection} \ref{Sec: 2.1} ensures that both models have a phase transition in equilibrium,  
		at sufficiently large symmetric particle intensities. The proof relies on a Peierls argument 
		which treats both models in a unified way. In {Subsection} \ref{sec: dob} we discuss the Dobrushin uniqueness theory in relation 
		to our model, and present regions in the parameter space of a priori measures and repulsion strength 
		for which Dobrushin uniqueness holds, see Theorem \ref{thm: dob hc},\ref{Thm: beta B},\ref{thm: neigh} and Figure \ref{fig: Dob sc}. This is first described in an equilibrium setup, but will later be used for the dynamical 
		model. 
		In {Subsection} \ref{sec: time} our results on dynamical Gibbs-non-Gibbs transitions are presented, starting with 
		the hard-core model. 
		Theorem \ref{thm: non-gibbs hc} gives a sharp result for the hard-core model in the percolation regime, 
		on the immediate loss of the Gibbs property with
		full-measure discontinuities.  
		The proof relies on a cluster representation of 
		single-site conditional probabilities. Theorem \ref{thm: non-gibbs hc low} describes the weaker singularities in the non-percolation 
		regime. {In both cases the Gibbs property for the asymmetric model is recovered after a sharp time which is stated in Theorem \ref{thm: hc is gibbs}}.  In view of these two theorems the 
		dynamical lattice hard-core model behaves as the corresponding Euclidean hard-core model, but 
		different to the lattice soft-core model, as the following results show.  
		Indeed, Theorem \ref{thm: short time gibbs} asserts that for the lattice soft-core model 
		there is a short-time Gibbsian behavior,  with a proof based on Dobrushin uniqueness. Theorems \ref{thm: gibbs all time} and \ref{thm: gibbs high asm}  give more sufficient criteria for the Gibbs property.  
		Theorem \ref{thm: sc non gibbs}  on the opposite ensures large-time non-Gibbsianness, by an argument which reduces 
		the question to the {corresponding} statement 
		for the dynamical Ising model for which it is known to be true.  
	
	\section{Setup and main-results}\label{Sec: 2}
	\subsection{The hard-core and soft-core Widom-Rowlinson model and phase transition}\label{Sec: 2.1}
	We consider the single-state space $E:=\{-1,0,1\}$ and the site space $\Z^d$. The configuration space $\Omega := E^{\Z^d}$ is equipped with the product $\sigma$-Field $\mathcal{F}$ given by the discrete topology on $E$. For a finite set $\Lambda$ of $\Z^d$ we write $\Lambda\Subset \Z^d$.  By $\Omega_\Lambda$ and $\mathcal{F}_\Lambda$ we denote the restriction to some set $\Lambda\subset\Z^d$. For neighboring sites $i,j\in \Z^d$, i.e. $\Vert i-j\Vert_1 =1$, we write $i\sim j$. By $\mathcal{E}_\Lambda^b:=\{\{i,j\}\subset \Z^d\,:\, \Lambda\cap \{i,j\}\neq \emptyset \,,i\sim j \}$ we denote the set of bonds in $\Lambda\cup \partial \Lambda$, where $\partial \Lambda:=\{j\in \Lambda\,:\, i \sim j \text{ for some }i\in \Lambda\}$ is the outer boundary of $\Lambda$.
	
	If a function $f\,:\,\Omega \rightarrow \R$ is $\mathcal{F}_\Lambda$-measurable for some $\Lambda \Subset Z^d$ then $f$ is called local function. A function $f$ is quasilocal on $\Omega$ if there exists a sequence of local functions $(f_n)_{n\in \N}$ with $\lim_{n\rightarrow \infty} \Vert f-f_n\Vert_\infty = 0$. 
	
	A specification $\gamma=(\gamma_\Lambda)_{\Lambda \Subset \Z^d}$ is a family of probability kernel $\gamma_\Lambda$ from $\mathcal{F}_{\Lambda^c}$ to $\mathcal{F}$ which satisfy the properness condition $\gamma_\Lambda(B\vert \cdot) = \mathds{1}_B(\cdot)$ and the consistency condition $\gamma_\Lambda(\gamma_\Delta(A\vert\cdot )\vert \omega) = \gamma_\Lambda(A\vert \omega)$, for all $\Delta \subset \Lambda\Subset \Z^d, \omega\in {\Lambda^c}, A\in \mathcal{F}$ and $B\in\mathcal{F}_{\Lambda^c}$. A specification is called quasilocal if for every $\Lambda \Subset \Z^d$ and every quasilocal function $f\,:\, \Omega\rightarrow \R$ the function
	\begin{align*}
			\gamma_\Lambda(f\vert \cdot ) := \int_\Omega \gamma_\Lambda(d\omega\vert \cdot ) f(\omega) 
	\end{align*} is quasilocal.
	We say a measure $\mu$ on $(\Omega,\mathcal{F})$ is admitted by a specification $\gamma$ if the DLR-equation 
	\begin{align*}
	\mu(A\vert \mathcal{F}_{\Lambda^c} )(\cdot) = \gamma_\Lambda(A\vert \cdot) \,\, \mu\text{-a.s.}
	\end{align*} holds for every $\Lambda \Subset \Z^d$ and $A\in \mathcal{F}$. If $\mu$ is admitted by a quasilocal specification we call $\mu$ a Gibbs measure. We define the set of all Gibbs measures for a quasilocal specification $\gamma$ by $\mathcal{G}(\gamma)$. We say a phase transition occurs if there are multiple  Gibbs measures for a specification. 
	
The interpretation of the spin state is as follows.	If $\omega_i=0$ we say that there is no particle at site $i$, if $|\omega_i|=1$ we say that a particle is present at $i$, where we interpret the value $-1$ as particle with a negative spin, and $+1$ as a particle with positive spin. We are interested in a model with hard-core repulsion in the sense that $+$ and $-$-particle are not allowed to be nearest neighbors, and also a related model with a soft-core repulsion where particles with different sign of the spin value can be nearest neighbors but it will be punished by a parameter $\beta>0$.
	
	\begin{definition}\label{defi: models}
		Let $h\in\R$ and $\beta,\lambda>0$.
		\begin{itemize}
			\item The specification $\gamma_{\lambda,h}^{hc}$  for the discrete hard-core Widom-Rowlinson model with parameters $h$ and $\lambda$ is defined via
			\begin{align*}
			\gamma^{hc}_{\Lambda,\lambda,h}(\omega\vert \eta ) =\frac{1 }{Z^{hc}_\eta} I_{\Lambda}^{hc}(\omega_\Lambda\eta_{\Lambda^c})\,e^{\sum_{i\in\Lambda}\log(\lambda)\omega_i^2+h\omega_i}
			\end{align*}
			where $I_{\Lambda}^{hc}(\omega):=\prod_{\{i,j\}\in\mathcal{E}^b_\Lambda}\mathds{1}(\omega_i\omega_j\neq-1)$, $\Lambda \Subset V$ is the hard-core restriction, $\omega\in \Omega_\Lambda$ and $\eta \in \Omega_{\Lambda^c}$.
			\item The specification $\gamma^{sc}_{\beta,\lambda,h}$ for the discrete soft-core Widom-Rowlinson model with parameters $\beta,\lambda$ and $h$ is defined via 
			\begin{align*}
			\gamma^{sc}_{\Lambda,\beta,\lambda,h}(\omega\vert \eta ) = \frac{1}{Z^{sc}_\eta} e^{-\mathcal{H}_\Lambda(\omega_{\Lambda}\eta_{\Lambda^c})}
			\end{align*}
			where $\mathcal{H}_\Lambda(\omega):= \sum_{\{i,j\}\in \mathcal{E}_{\Lambda}^b}\beta\mathds{1}(\omega_i\omega_j=-1)-\sum_{i\in\Lambda} \log(\lambda)\omega_i^2-h\omega_i$ is the finite-volume Hamiltonian, $\Lambda \Subset V$, $\omega\in \Omega_\Lambda$ and $\eta \in \Omega_{\Lambda^c}$. 
		\end{itemize}
		$Z^{hc}_\eta$ and $Z^{sc}_\eta$ are called partition functions and are chosen such that  $\gamma^{hc}_{\Lambda,\lambda,h}(\cdot\vert \eta )$ and $\gamma^{sc}_{\Lambda,\beta,\lambda,h}(\cdot\vert \eta )$ are probability measures on $(\Omega_\Lambda,\mathcal{F}_\Lambda)$
	\end{definition}
	The parameters of our models can be understood as external magnetic field $h$, particle intensity $\lambda$ and repulsion strength $\beta$. Another useful description is to work with an a priori measure $\al\in\mathcal{M}_1(E)$ where all information about the single-site behavior is contained. The relation between the descriptions is given by $h= \frac{1}{2}\log(\frac{\al(1)}{\al(-1)})$ and $\lambda = \frac{\sqrt{\al(1)\al(-1)}}{\al(0)} $. The particles interact only if they are connected with a bond hence both specifications are local and consequently quasilocal. 
	
	\begin{remark}
		In literature the hard-core Widom-Rowlinson model is usually called discrete Widom-Rowlinson model. We introduced the prefix hard-core just to distinguish between our two models. The name of the second model is justified by the fact that $\lim_{\beta\rightarrow \infty} \gamma^{sc}_{\Lambda,\beta,\lambda,h}(\cdot\vert\eta)=\gamma^{hc}_{\Lambda,\lambda,h}(\cdot\vert \eta)$.
	\end{remark}
	
	 For our models we have the following theorem concerning phase transition.

		\begin{theorem}\label{thm : Sc Phase}
			Let $d\geq2$ and $h=0$. There exist $\beta_c,\lambda_c>0$ such that for all $\beta\geq\beta_c$ and $\lambda\geq\lambda_c$ the soft-core Widom-Rowlinson model has a phase transition, i.e.
			\begin{align*}
			\vert \mathcal{G}(\gamma^{sc}_{\beta,\lambda,0})\vert >1.
			\end{align*}
		\end{theorem}
		
		We will prove this theorem by a Peierls argument. Since the Peierls constant turns out to be of the form $\rho_{\beta,\lambda}= \frac{\min\{\beta,\log(\lambda)\}}{2d+1}$ we get a phase transition result for the hard-core model from the estimate for the soft-core model.
		
		\begin{corollary}\label{thm : Hc Phase}
			Let $d\geq2$ and $h=0$. There exists $\lambda_c>0$ such that for all $\lambda\geq\lambda_c$ the hard-core Widom-Rowlinson model has a phase transition, i.e.
			\begin{align*}
			\vert \mathcal{G}(\gamma^{hc}_{\lambda,0})\vert >1.
			\end{align*}
		\end{corollary}
		That a phase transition occurs for the two dimensional hard-core model was already proven in \cite{higuchi83} with percolation methods.
		
		\subsection{Dobrushin condition}\label{sec: dob}
		
		A crucial part in proving the short-time Gibbs property of the time-evolved model plays Dobrushin's uniqueness theorem. It gives a condition for absence of a phase transition and can be handled by discrete computations and works for strong asymmetry (i.e. high external magnetic field) or weak interacting.
		
		We will formulate this theory for connected locally finite graphs with infinite vertex set. Later results for models on the graph $\Z^d\backslash \{0\}$ are needed. So let $G= (V,K)$ be a locally finite graph with vertex set $V$ and edge set $K$. The construction of the DLR-formalism can be adapted to this setup. By $B_i$ we denote the degree of the vertex $i$, i.e. the number of edges which are connected to this vertex, and we define the maximal degree $B=\sup_{i\in V} B_i$. 
		
		For the Dobrushin theorem we need the single-site kernels 
		$$\gamma_i^0(\omega_i\vert \eta):= \gamma_{\{i\}}(\omega_i\vert \eta_{V\backslash \{i\}}) \,,\, \eta\in \Omega_V,\omega_i\in E, i\in V,$$
		of a specification $\gamma$ where $\gamma_i^0(\cdot\vert \eta)$ is a measure only on the single-site space $(E,\mathcal{F}_0)$. Via these kernels one can define Dobrushin's interdependence matrix
		\begin{align*}
		C(\gamma) &:= (C_{ij}(\gamma))_{i,j\in V} \text{ with }\\
		C_{ij}(\gamma)&:= \sup_{\eta,\zeta\in \Omega :\, \eta_{V\backslash{j} }=\zeta_{V\backslash {j}}} \dd_{TV}(\gamma_i^0(\cdot\vert \eta),\gamma_i^0(\cdot\vert \zeta)).
		\end{align*}
		where $d_{TV}$ is the total variational distance on the space $\mathcal{M}_1(E)$. The entries $C_{ij}$ measure how much the single-site kernels depend on the boundary condition, if we change one site in it. If all $C_{ij}$ are small then the model depends only weakly on the boundary condition. 
		\begin{definition}
			Let $\gamma$ be a specification. If the Dobrushin constant $c(\gamma):=\sup_{i\in V} \sum_{j\in V} C_{ij} <1$  and $\gamma$ is quasilocal we say that $\gamma$ satisfies Dobrushin's condition.   
		\end{definition}
		  Let $\ell^\infty$ the space of bounded sequences equipped with the uniform norm. Then one can see $C(\gamma)$ as a linear operator from $\ell^\infty\rightarrow \ell^\infty$. The Dobrushin condition can be rephrased with the operator norm $c(\gamma)= \Vert C(\gamma) \Vert_{op} <1$ and hence one can see this as an contradiction argument.
		\begin{theorem}
		Suppose  $(E,\mathcal{F}_0)$ is a standard Borel space.	If a specification $\gamma$ satisfies the Dobrushin condition then $\vert \mathcal{G}(\gamma)\vert=1$.
		\end{theorem}
		\begin{proof}
			See \cite[Theorem 8.7]{georgii-book}
		\end{proof}
		Of course the space $E=\{-1,0,1\}$ equipped with the discrete topology is standard Borel. In the following it is easier to state the results for the $\al$ description. For the hard-core model we can give an explicit regime for Dobrushin uniqueness. 
		\begin{theorem}\label{thm: dob hc}
		The hard-core Widom-Rowlinson specification satisfies  Dobrushin's condition iff \begin{itemize}
				\item $B=1$ and $\al\notin \{\delta_1,\delta_{-1}\}$,
				\item $2 \leq B<\infty$ and $\max\{\al(-1),\al(1)\} < \frac{\al(0)}{ B-1}$ or
				\item $B=\infty$ and $\al(0)=1$
			\end{itemize}
		\end{theorem}
			In fig. \ref{fig: B4}  one can see the areas of Dobrushin uniqueness (blue) on the simplex of probability measures on $E$. Since the boundary condition has more influence on the single-site behavior if $B$ is large the regions get smaller with increasing $B$.
%			\begin{figure}[!htb]
%				\hfill %
%				\subfloat[$B=2$]{\includegraphics[width=0.32\textwidth]{HC_2_new.png}}
%				\hfill % alternativ auch \hspace{1cm} für genaue Angaben
%				\subfloat[$B=4$]{\includegraphics[width=0.32\textwidth]{HC_4_new.png}}
%				\hfill %
%				\subfloat[$B=10$]{\includegraphics[width=0.32\textwidth]{HC_10_new.png}}
%				\hfill %
%				\subfloat[$B=4, \beta=0.4$\label{fig: beta 0,4}]{\includegraphics[width=0.32\textwidth]{SC_0,4_4_new.png}}
%				\hfill % alternativ auch \hspace{1cm} für genaue Angaben
%				\subfloat[$B=4, \beta=1$\label{fig: beta 1}]{\includegraphics[width=0.32\textwidth]{SC_1_4_new.png}}
%				\hfill %
%				\subfloat[$B=4,\beta =1.1$\label{fig: beta 1,1}]{\includegraphics[width=0.32\textwidth]{SC_1,1_4_new.png}}
%				\hfill %
%				\caption{Regions of Dobrushin uniqueness (blue) for the hard-core model}\label{fig: Dob hc}
%			\end{figure}
			\begin{figure}[!htb]
				\subfloat[$B=4, \beta=0.4$\label{fig: beta 0,4}]{\includegraphics[width=0.24\textwidth]{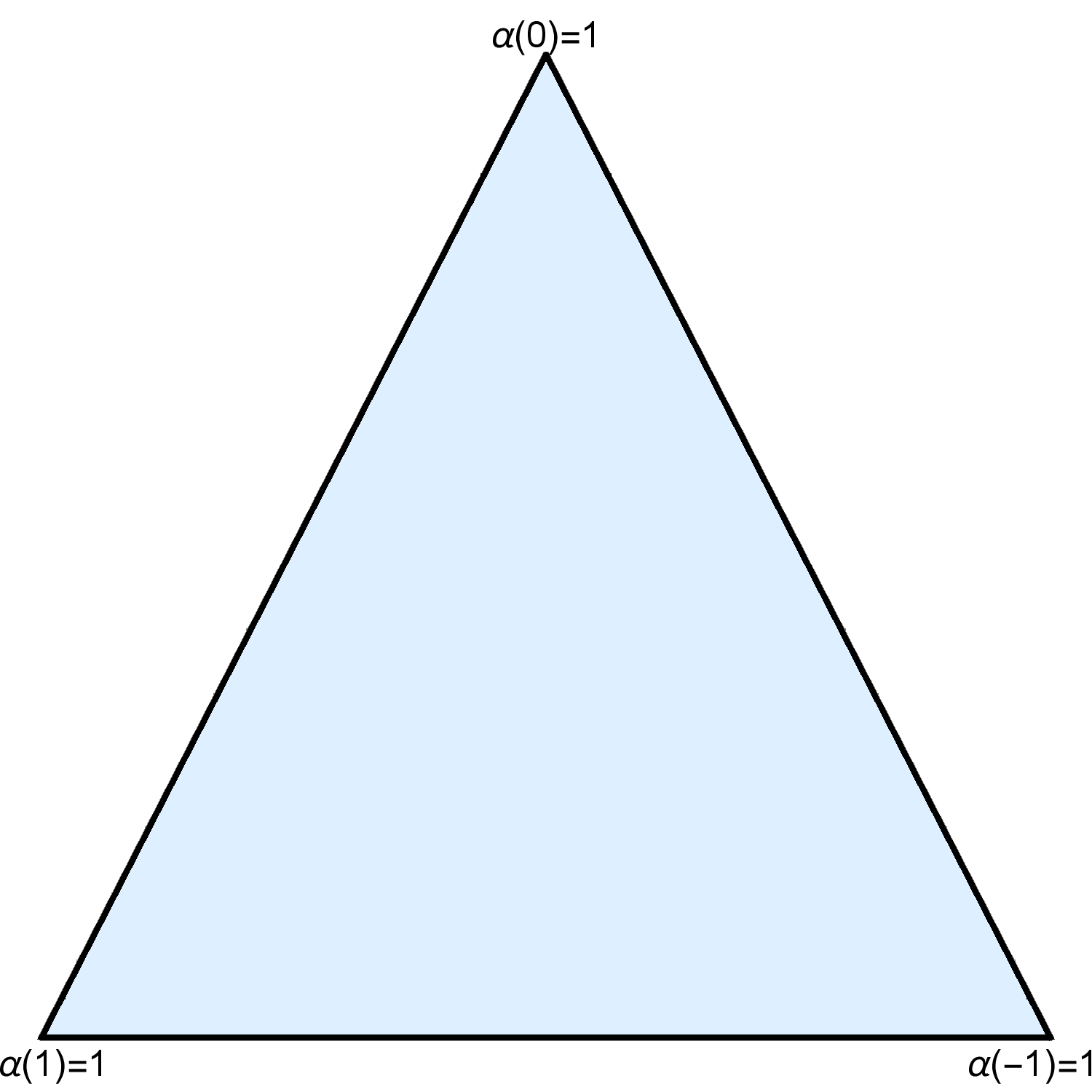}}
				\hfill
				\subfloat[$B=4, \beta=1$\label{fig: beta 1}]{\includegraphics[width=0.24\textwidth]{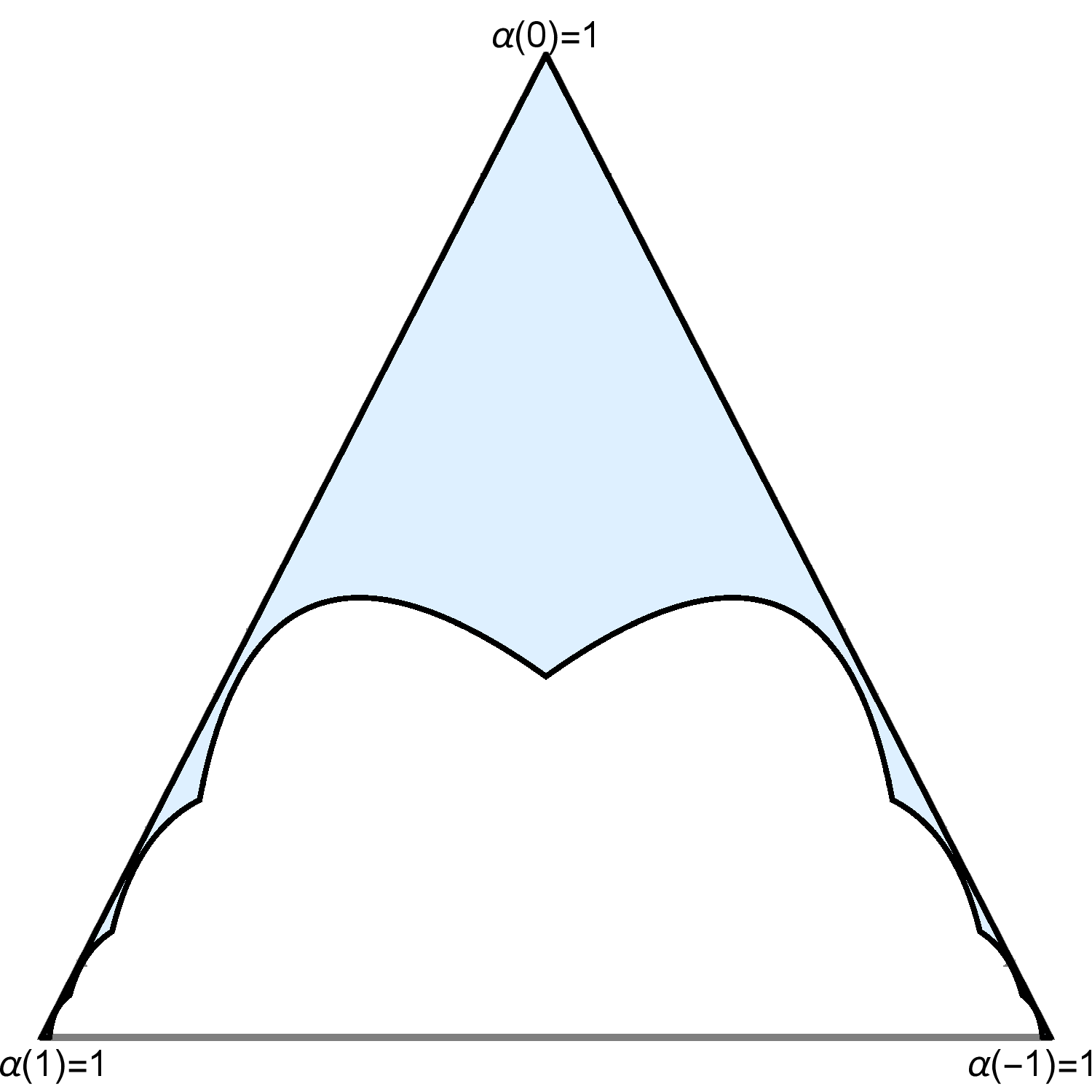}}
				\hfill
				\subfloat[$B=4,\beta =1.1$\label{fig: beta 1,1}]{\includegraphics[width=0.24\textwidth]{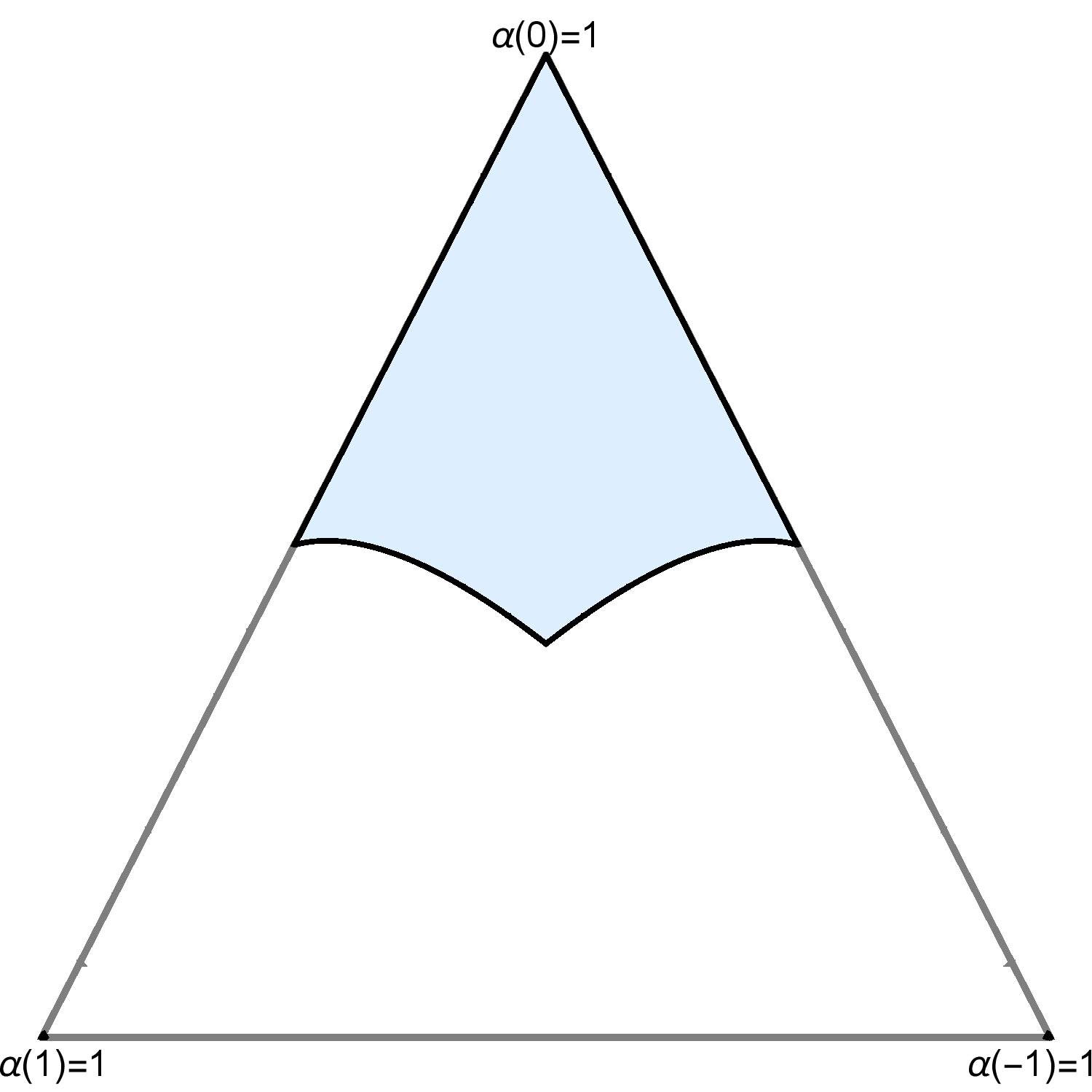}}
				\hfill
				\subfloat[$B=4,\beta=\infty$\label{fig: B4}]{\includegraphics[width=0.24\textwidth]{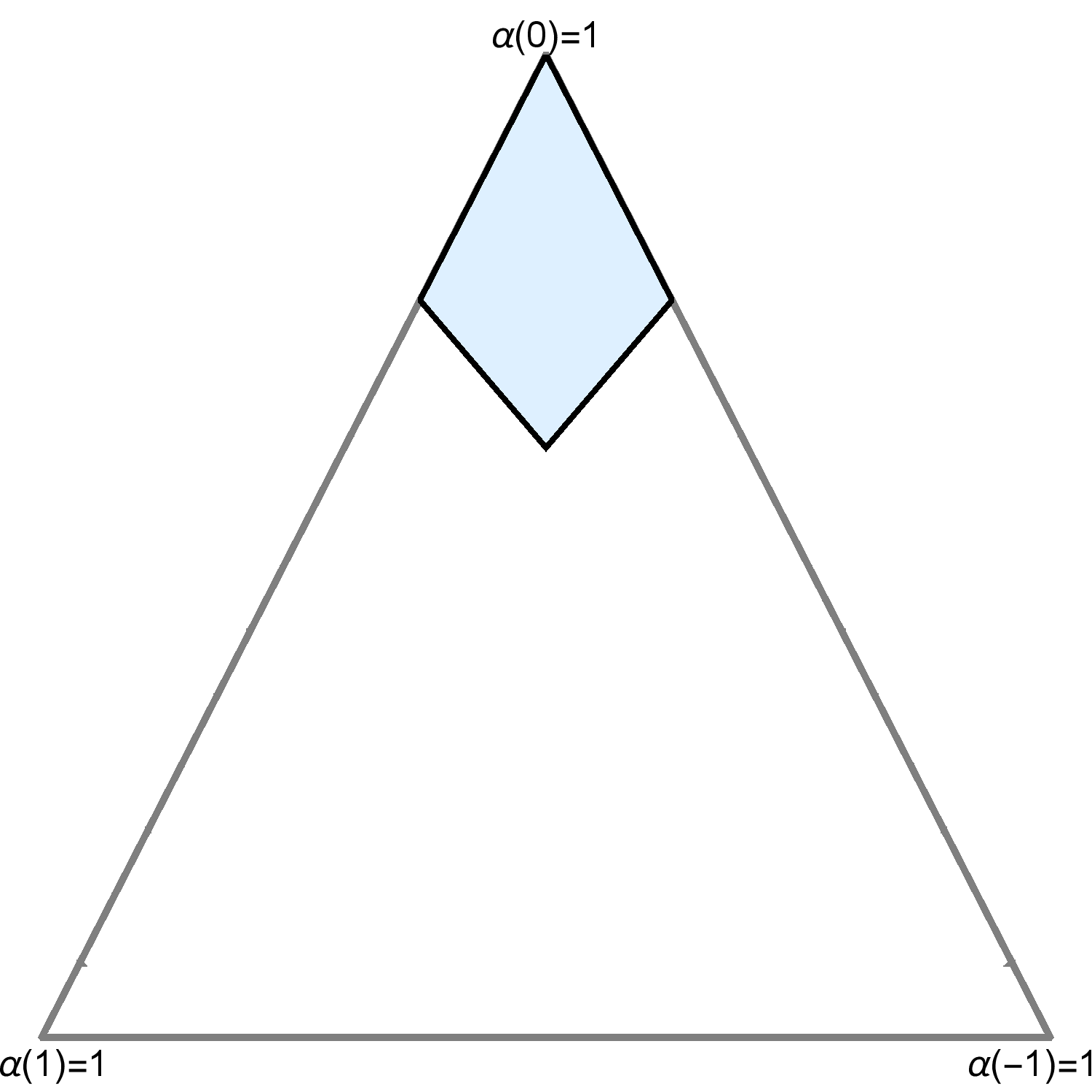}}
				\hfill
				\caption{Regions of Dobrushin uniqueness (blue) for the  soft-core model (first three) and hard-core model (last one).}\label{fig: Dob sc}
			\end{figure}
			
			For the soft-core model we can give a formula for the entries of $C_{ij}$ where we have to maximize over a finite set (see Lemma \ref{lem: C_ij_sc}). {It turns out that the entries are given as fractions of quadratic polynomials $\frac{Q_1(x,y)}{Q_2(x,y)}$ in two variables and we can reformulate the condition by requiring that all $B$-dependent quadratic polynomials $Q_B(x,y):=BQ_1(x,y)-Q_2(x,y)$ have to be smaller than $0$ for Dobrushin uniqueness. Since there only finitely many such polynomials the boundary of the Dobrushin uniqueness region on the simplex is given by the boundary of finitely many level sets of the polynomials.} If the interaction between particles is small, i.e. $\beta$ is small,  the specification satisfies Dobrushin's condition for every a priori measure $\al$ (see fig. \ref{fig: beta 0,4}). 
			
				\begin{theorem}\label{Thm: beta B}
					If $\beta B<2$ then the specification of the soft-core model satisfies Dobrushin's condition for every choice of $\al\in\mathcal{M}_1(E)$.
				\end{theorem}
			\begin{proof}
				This follows by Proposition 8.8 in \cite{georgii-book}.  
			\end{proof}
			In fig. \ref{fig: beta 1} and \ref{fig: beta 1,1} we see that around the measures with $\al(i)=1$, $i\in \{-1,1\}$, there are small areas of Dobrushin uniqueness. The existence of these small areas is one of the main ingredient to prove short-time Gibbs for the time-evolved soft-core model. For $\epsilon>0$ and $i\in E$ we write $U^i_\epsilon:= \{\al\in \mathcal{M}_1(E)\,:\, d_{TV}(\al,\delta_i)<\epsilon\}$ for an $\epsilon$-neighborhood of $\al = \delta_i$.
				
			\begin{theorem}\label{thm: neigh}
				Assume  $B<\infty$. Then for every $\beta>0$ there exist neighborhoods $U^1_\epsilon, U^0_\epsilon $ and $U_\epsilon^{-1}$ such that for every $\al\in \bigcup_{i\in E} U^i_\epsilon$ the specification for the soft-core model satisfies Dobrushin's condition.
			\end{theorem}
			\subsection{Time-evolution}\label{sec: time}
			For the time-evolved model we consider a stochastic kernel which exchanges $+$ and $-$ spins with the same rate independently at each site $i\in \Z^d$ and there is no creation or erasing of a particle. Since the transition is independent at each site it is enough to define the transition kernel for a single-site
			\begin{align}\label{trans}
			p_t(a,b)= \frac{1}{2}(1+e^{-2t})\mathds{1}_{a=b\neq0}+\frac{1}{2}(1-e^{-2t})\mathds{1}_{a b =-1} 
			+\mathds{1}_{a=b=0}
			\end{align}
			where $a,b\in E$ and $t>0$. We write $\omega\in\Omega$ for a configuration at time $0$ and $\eta\in\Omega$ for one at time $t$. 
			Let $\mu$ be a Gibbs measure for the hard-core or soft-core model then the time-evolved measure at time $t>0$ is defined via $\mu_t(f) = \int_\Omega\int_{\Omega} f(\eta) p_t(\omega,d\eta)\mu(d\omega)$.  
			
			Whether the time-evolved measure is a Gibbs measure or not, depends on the existence of a quasilocal specification for $\mu_t$. For asymmetric $\al\in\mathcal{M}_1(E)$, i.e. $\al(1)\neq \al(-1)$, we have the Gibbs property for the time-evolved hard-core model, for large enough $t$, as we will describe now. By $\mu^+$ we denote the limiting Gibbs measure $\mu^+:= \lim_{\Lambda \uparrow} \gamma^{hc}_{\Lambda, \al}(\cdot\vert \omega^+)$ coming from the all-plus boundary condition. 
			
			\begin{theorem}\label{thm: hc is gibbs}
				Let $\al\in\mathcal{M}_1(E)$ with $\al(1)>\al(-1)$ and $\mu^+\in\mathcal{G}(\gamma^{hc}_\al)$. Then for all $t>t_G:= \frac{1}{2}\log\left(\frac{\al(1)+\al(-1)}{\al(1)-\al(-1)}\right)$ the time-evolved measure $\mu^+_t$ is Gibbs. 
			\end{theorem}
			It is conjectured that in the asymmetric model at time zero there is no phase transition and then all time-evolved measures would be Gibbs for $t>t_G$. 
			 Since the DLR-equation is formulated almost surely one has to prove for non-Gibbsianness that all specifications for $\mu_t$ are non-quasilocal.

			\begin{definition}
				Let $\gamma$ be a specification on $\Z^d$ with single-site spin state $(E,\mathcal{F}_0)$. A configuration $\eta\in \Omega$ is called bad for $\gamma$ if there exist $\Delta\Subset \Z^d$, a local function $f$ and $\zeta^1,\zeta^2\in \Omega$ such that 
				\begin{align*}
					\lim_{\Lambda \uparrow \Z^d} \vert \gamma_\Delta(f\vert \eta_{\Lambda\backslash\Delta} \zeta^1_{\Lambda^c})-\gamma_\Delta(f\vert \eta_{\Lambda\backslash\Delta} \zeta^2_{\Lambda^c})\vert>0.
				\end{align*}
			\end{definition}
			By \cite{georgii-book} the existence of a bad configuration for a specification $\gamma$ implies the non-quasilocality of $\gamma$. 
			
			For the time-evolved hard-core model we will prove that bad configurations exist by using a cluster representation of the model.

				\begin{definition}\label{Defi: Cluster}
						Let $\zeta\in \{0,1\}^{\Z^d}$.  Then $C\subset\Z^d$ is called a cluster (or connected component) 
						if it is connected, that is, 
						if for all $i,j\in C$ there exists a finite sequence  $i=i_1,\ldots,i_k=j\in \Z^d$ with $i_{m+1}\sim i_m$ and $\zeta_m=1$, and $C$ is maximal with this property. 
						The set of all clusters for $\zeta$ is denoted by 
						$\mathcal{C}(\zeta)$. \\
						Further define for a finite volume $\Lambda\Subset \Z^d$,  
						$\mathcal{C}_\Lambda(\zeta_\Lambda\zeta_{\Lambda^c})$ to be the set 
						of clusters for $\zeta$ with $C\cap (\Lambda\cup \partial\Lambda)\neq \emptyset$. 
						Denote by $\mathcal{C}_{\Lambda^c}(\zeta)$ the complement of 
						$\mathcal{C}_\Lambda(\zeta_\Lambda\zeta_{\Lambda^c})$ in 
						$\mathcal{C}(\zeta)$.  
			\end{definition}
			This decomposition of $\mathcal{C}$ has the advantage that for fixed $\zeta_{\Lambda^c}$ the set $\mathcal{C}_{\Lambda^c}(\zeta)= \mathcal{C}_{\Lambda^c}(\zeta_\Lambda\zeta_{\Lambda^c})$ does not depend on $\zeta_{\Lambda}$.
			 Since all connected components of the time zero configuration have the same sign, a connected component will be a cluster. We say the model is in a high intensity regime if for some $\Lambda\Subset \Z^d$ the event $\{$there exists an infinite cluster with $\Lambda \cap C \neq \emptyset\}=:\{\Lambda\leftrightarrow \infty\}$ has positive probability under $\mu\in\mathcal{G}(\gamma^{hc}_\al)$. 
			
			\begin{theorem}\label{thm: non-gibbs hc}
				Consider the asymmetric model $\al(-1)<\al(1)$, in the high-intensity regime.  
				Then the time-evolved hard-core measure $\mu^+_t$ is non-Gibbs if  $0<t<t_G$. 
				
				Consider the symmetric model $\al(-1)=\al(1)$, in the high-intensity regime. Then, for any translation-invariant Gibbs measure as a starting 
				measure,  the time-evolved hard-core measure 
				$\mu_t$ is non-Gibbs for all $t>0$.  
				
				In both cases the sets of bad configurations have full measure with respect to the time-evolved measure.  
			\end{theorem}
			The last statement means that the set of bad configurations for any specification of the time-evolved measure has probability one for the time-evolved measure. 
			In the low intensity regime the time-evolved model is also non-Gibbs but the bad configurations form a null set. 
			\begin{theorem}\label{thm: non-gibbs hc low}
				Consider the asymmetric model $\al(-1)<\al(1)$, in the low-intensity regime.  
				Then the time-evolved hard-core measure $\mu^+_t$ is non-Gibbs if  $0<t<t_G$. 
				
				Consider the symmetric model $\al(-1)=\al(1)$, in the low-intensity regime. Then, for any translation-invariant Gibbs measure as a starting 
				measure,  the time-evolved hard-core measure 
				$\mu_t$ is non-Gibbs for all $t>0$.  
				
				In both cases the sets of bad configurations have zero measure with respect to the time-evolved measure. 
			\end{theorem}
			In this case there exists an almost-surely quasilocal specification for the time-evolved measure and we say $\mu_t$ is almost surely Gibbs.
			The time zero measure $\mu$ is Gibbs and immediately after starting the time evolution it loses the Gibbs property. In the asymmetric model it recovers the Gibbs property after some time. For the soft-core model the case is different. Here the model is short-time Gibbs and in a low interaction regime it is Gibbs for all times $t>0$.
			
			\begin{theorem}\label{thm: short time gibbs}
			Let $\mu\in\mathcal{G}(\gamma^{sc}_{\beta,\al})$. For every $\beta$ and every $\al\in\mathcal{M}_1(E)$ there exists a time $t_0(\beta,\al)$ such that $\mu_t$ is a Gibbs measure for all times $t<t_0(\beta,\al)$.
			\end{theorem}
			\begin{theorem}\label{thm: gibbs all time}
			Let $\mu\in\mathcal{G}(\gamma^{sc}_{\beta,\al})$. If $\beta<\log(\frac{2d+1}{2d-1})$ then the time-evolved measure $\mu_t$ is Gibbs for all $t>0$.  
			\end{theorem}
			
			For highly asymmetric $\al$ the model is Gibbs for large times.
			\begin{theorem}\label{thm: gibbs high asm}
				Let $\mu\in \mathcal{G}(\gamma^{sc}_{\beta,\al})$, $U^{1}_{\epsilon},U^{-1}_{\epsilon}$ the neighborhoods given by Theorem \ref{thm: neigh}, and $\al\in\mathcal{M}_1(E)$ such that the probability measure $\bar{\al}$ with $\bar{\al}(\pm1)= \frac{\al(\pm1)}{\al(1)+\al(-1)}$ is an element of $U^1_{\epsilon} \cup U^{-1}_{\epsilon}$. Then there exists a time $t_1(\beta,\al)$ such that for all $t>t_1(\beta,\al)$ the time-evolved measure is Gibbs.
			\end{theorem}
			
			But for symmetric $\al$ the checkerboard configuration $\eta^{cb}$  is bad for the time-evolved measure and large times $t$. Its  defined via  
			\begin{align}\label{eq: check board}
			\eta^{cb}_i =\left\{  \begin{array}{ccc}1 &$if$&  \sum_{k=1}^d \vert i_k \vert$ is even$\\  -1 &\text{if}&  \sum_{k=1}^d \vert i_k\vert$ is odd$\end{array}\right..
			\end{align}
			
			\begin{theorem}\label{thm: sc non gibbs}
				Let $\mu\in \mathcal{G}(\gamma^{sc}_{\beta,\al})$ and $\al\in\mathcal{M}_1(E)$ symmetric. Then for large enough $\beta$ and $\lambda$ there exists a time $t_2(\beta,\al)$ such that $\eta^{cb}$ is bad for the time-evolved measure for all times $t>t_2(\beta,\al)$.
			\end{theorem} 
			 
			\section{Proofs for the static models}\label{Sec: 3}
			\subsection{Phase transition and Peierls argument}
		In this part we are only interested in models with no external magnetic field therefore we will not mention the parameter $h$. 
			The existence of a Gibbs measure for the soft-core model is given by the monotonicity property of the single-site kernels of the specification $\gamma^{sc}_{\beta,\lambda}$ and the FKG-inequality. Even more one can prove that there exist two special Gibbs measures which are translation invariant and are given by $\lim_{\Lambda\uparrow \Z^d}  \gamma^{sc}_{\Lambda,\beta,\lambda}(\cdot\vert \eta^{\pm})= \mu^{\pm}(\cdot)$ where $\eta^{\pm}$ are the all-plus and all-minus
		configurations, respectively. For more information about FKG-inequality see \cite{georgii-haggstrom-maes01}.
		
		A Hamiltonian can also be defined via a potential $\phi$. For the symmetric soft-core model it is given by 
			\begin{align*}
			\phi_\Delta(\omega):=\left\{\begin{array}{ccc}
			\beta \mathds{1}(\omega_i\omega_j=-1) & if & \Delta=\{i,j\} \text{ with } i\sim j\\
			-\log(\lambda)\omega_i^2 &if & \Delta=\{i\}\\
			0& else &\end{array}\right.
			\end{align*}	
			for $\Delta \Subset \Z^d$ and the Hamiltonian can be written as $\mathcal{H}_\Lambda(\omega)= \sum_{\Delta\Subset \Z^d, \Lambda\cap \Delta \neq \emptyset}\phi_\Delta(\omega)$. For the Peierls argument we need the definition of a ground state. 
		\begin{definition}
			Two configurations $\omega,\eta\in \Omega$  are equal up to a finite set, if there exists a finite set $\Lambda \subset \Z^d$ with $\eta_{\Lambda^c} = \omega_{\Lambda^c}$. This is denoted by  $\omega \stackrel{\infty}{=} \eta$. \\
				For those pair of configurations the relative Hamiltonian is defined by
			$	\mathcal{H}_\phi(\omega\vert \eta) = \sum_{\Delta \Subset V} (\phi_\Delta(\omega)-\phi_\Delta(\eta)) .$
				If $\mathcal{H}_\phi(\omega\vert \eta)\geq 0$ for all $\omega \stackrel{\infty}{=} \eta$ then $\eta$ is called ground state. 
		\end{definition}\noindent
		A ground state admits the minimal energy for a Hamiltonian and every finite change of the configuration increases the energy. The all-plus and all-minus configurations $\eta^{\pm}$ are the only periodic ground states for the symmetric soft-core model which can be proven by \cite[Lemma 7.4]{friedli_velenik_2017}. To specify the location of sites which not coincide with the spin of a ground state we define the following set $\mathcal{K}$.
		
		\begin{definition}
			A site $i \in \Z^d$ is said to be correct if there exists a ground state $\eta^\#$ with $\#\in \{+,-\}$ such that $\omega_j=\eta^\#_i$ for all $j\in \{k\,\vert\, k \sim i\}\cup \{i\}$. Then the set of incorrect sites is defined by
			\begin{align*}
			\mathcal{K}(\omega) := \{i \in \Z^d \,:\, \text{ i is  incorrect for any ground state}\}.
			\end{align*}
		\end{definition} 
		
		\hspace{-3pt}With $\mathcal{K}$ one can give a lower bound for the relative Hamiltonian of a ground state and a configuration which differs only on finitely many sites.
			\begin{lemma}\label{lem : relative Ham}
				Let $\omega \in \Omega$ be a configuration with $\omega \stackrel{\infty}{=} \eta^+$ or $\omega \stackrel{\infty}{=} \eta^-$ then 
				\begin{align*}
				\mathcal{H}_\phi(\omega\vert \eta^{\pm})\geq \vert \mathcal{K}(\omega) \vert \frac{\min\{\beta, \log(\lambda)\}}{(2d+1)} .
				\end{align*}
			\end{lemma}
			\begin{proof}
				We only prove it for $\eta^+$.  The key idea is to show that if $i$ is incorrect then there exists a 
				\begin{align*}
				\Delta\in\tilde{B}(i):=\{\{i\}\}\cup\{\{j\}\,\vert \, j\sim i\} \cup \{\{i,j\} \,\vert \, j\sim i\}
				\end{align*} 
				such that $\phi_\Delta(\omega) > \tilde{\phi}_\Delta:= \min_{\omega\in \Omega}\phi_\Delta(\omega) $. Since $\eta^+$ is a ground state it is easy to see that $\tilde{\phi}_\Delta= \phi_\Delta(\eta^+) = -\log(\lambda)\mathds{1}_{\vert \Delta \vert =1}$ for all $\Delta\subset \Z^d$.\\
				For $\omega_i=0$ and $\Delta=\{i\}$ the potential $\phi_\Delta(\omega)=0>-\log(\lambda)$. If $\omega_i=1$ two cases are possible. Either there exists a $j\sim i$ with $\omega_j=0$ or $\omega_j=-1$. For the first case set $\Delta=\{j\}$ then $\Delta\in \tilde{B}(i)$ and $\phi_\Delta(\omega)=0>-\log(\lambda)$. For the second case $\Delta=\{i,j\}$ can be used because $\phi_\Delta(\omega)=\beta > 0$. If the configuration at site $i$ is equal to minus we process the same as for the case where $\omega_i=1$.   It follows for every set $\Delta$ which is not in  $\cup_{i \in \mathcal{K}(\omega)}\tilde{B}(i)$ that $\phi_\Delta(\omega)-\tilde{\phi}_\Delta=0$. By this the relative Hamiltonian has the form
				\begin{align*}
				\mathcal{H}_\phi(\omega\vert \eta^+) &= \sum_{\Delta\in \cup_{i\in \mathcal{K}(\omega)}\tilde{B}(i)} (\phi_\Delta(\omega)-\tilde{\phi}_\Delta).
				\end{align*}
				We know that for every $i\in\mathcal{K}(\omega)$ there exists an $\Delta_i\in \tilde{B}(i)$ with $\phi_{\Delta_i}(\omega)-\tilde{\phi}_{\Delta_i}>0$ and so we can say that $i$ contributes $\frac{1}{2d+1}$  of the difference  $\phi_{\Delta_i}(\omega)-\tilde{\phi}_{\Delta_i}$. With this idea it follows that 
				\begin{align*}
				\mathcal{H}_\phi(\omega\vert \eta^+) \geq \sum_{i\in\mathcal{K}(\omega)}\frac{1}{2d+1} (\phi_{\Delta_i}(\omega)-\tilde{\phi}_{\Delta_i}) \geq \vert\mathcal{K}(\omega)\vert \frac{\epsilon}{2d+1},
				\end{align*}
				where $\epsilon=\min\{\phi_\Delta(\omega)-\tilde{\phi}_{\Delta}\,:\, \phi_\Delta(\omega)> \tilde{\phi}_{\Delta} , \Delta\in\cup_{i\in \mathcal{K}(\omega)}\tilde{B}(i)\} = \min\{\beta,\log(\lambda)\}$ and the lower bound has been proven.
			\end{proof}
		The constant $\rho:=\rho_{\beta,\lambda,d}:=\frac{\min\{\beta, \log(\lambda)\}}{(2d+1)}$ is called Peierls constant. For configurations $\omega\stackrel{\infty}{=} \eta^{\pm}$ one can write for every $\Lambda\Subset \Z^d$ that 
		\begin{align*}
		\gamma^{sc}_{\Lambda,\beta,\lambda}(\omega\vert \eta^+)= \frac{e^{\vert\Lambda\vert\log(\lambda)}e^{-\mathcal{H}_\phi(\omega\vert \eta^+)}}{\sum_{\omega\in\Omega_\Lambda}e^{\vert\Lambda\vert\log(\lambda)}e^{-\mathcal{H}_\phi(\omega\vert \eta^+)}}= \frac{e^{-\mathcal{H}_\phi(\omega\vert \eta^+)}}{\sum_{\omega\in\Omega_\Lambda}e^{-\mathcal{H}_\phi(\omega\vert \eta^+)}}.
		\end{align*}
		To prove phase transition we want to show that $\gamma^{sc}_{\Lambda,\beta,\lambda}(\{\omega\in \Omega\,:\, \omega_0=\{-1,0\}\}\vert \eta^+)<a(\beta,\lambda)$ with $\lim_{\beta,\lambda\rightarrow\infty}a(\beta,\lambda)=0$. For this we split $\mathcal{K}(\omega)$ into several parts.
		
		We say a set $\Lambda \subset \Z^d$ is connected if for all $i,j\in \Lambda$ there exists a sequence $i_1=i,\ldots, i_n=j$ such that $i_k\sim i_{k+1}$ and $i_k\in \Lambda$ for all $k\in \{1,\ldots,n-1\}$. Let $W\subset \Z^d$. A connected set $\Lambda\subset W$ is maximal if any set $\Delta$ with $W\supset\Delta\supsetneq\Lambda $ is disconnected.  This implies that for every configuration $\omega\stackrel{\infty}{=}\eta^+$ the set $\mathcal{K}(\omega)$ can be disassembled into maximal finite connected components $\bar{\kappa}_1,\ldots,\bar{\kappa}_k$ for some finite $k$. Furthermore every $\bar{\kappa}_j$ splits $\Z^d$ again into a finite set of maximal connected components $A_0,A_1,\ldots,A_k$ with $ \bar{\kappa}_j^c:=\cup_{i=0}^kA_i$. There exists exactly one of the $A_k$ which is unbounded and without loss of generality we say that $A_0$ is this set.  The pair $\kappa = (\bar{\kappa},\omega_{\bar{\kappa}})$ is called a contour of $\omega$.
		\begin{lemma}
			 For every $A_k$, which is defined by the decomposition given by some contour $\kappa$, we have $\omega_i=\eta_i^+$ for all $i\in\partial A_k$ or $\omega_i=\eta_i^-$ for all $i\in\partial A_k$. 
		\end{lemma}
		\begin{proof}
			Define the dual set $\partial^{in} A := \{i\in A\,:\, \exists j\in A^c \text{ with } i\sim j\}$. Then for every $i\in \partial^{in} A_k$ there exists a $j\in \partial A_k$ with $i\sim j$. The site $i$ has to be correct for $+$ or $-$ otherwise it would be an element of $\mathcal{K}(\omega)$ and is connected to $\bar{\kappa}$. This implies that $\omega_i=\omega_j\in\{-1,1\}$. By this it is enough show that if there exists a $i\in \partial^{in} A_k$ with $\omega_i=1$ then all site in $\partial^{in} A_k$ are occupied with positive spin value. It follows by the correctness of $i$ that the configuration of every site $m\in \partial^{in} A_k$ which is connected to $i$ has to be  positive. The set $A_k$ is not connected but for the maximal connected components of $A_k$, labeled by $A_{k_1},\ldots,A_{k_s}$, it follows by \cite[Appendix B.15]{friedli_velenik_2017} that there exists for every two sets $A_{k_l},A_{k_{l'}}$ a path $i_1=i\in A_{k_l},i_2,\ldots, i_n \in A_{k_{l'}}$ where $\omega_{i_j}=\omega_{i}$ for every $j\in\{1,\ldots,n\}$. This concludes the proof.
		\end{proof}
		With $\lab(A_k)$ we define the label of a set $A_k$ and say the label is positive (resp. negative) if all $i\in \partial A_k$ are occupied by plus (resp. minus) spin values. The label of the unbounded set $A_0$ of a decomposition given by some  $\kappa$ is called the type of the contour. \\
		The next lemma is one of the core idea of the proof. It combines the Peierls constant with the idea of splitting the incorrect set into disjoint sets.
			\begin{lemma}
				Let $\Lambda\Subset\Z^d$, $\rho=\frac{\min\{\beta, \log(\lambda)\}}{(2d+1)}$ the Peierls constant and $\kappa^*$ be some contour. Then 
				\begin{align*}
				\gamma^{sc}_{\Lambda,\beta,\lambda}(\{\omega\in \Omega\,:\; \bar{\kappa}^*\in \mathcal{K}(\omega)\}\vert \eta^+ )\leq e^{-\vert \bar{\kappa}^*\vert \rho}.
				\end{align*}
			\end{lemma}
			\begin{proof}
				First note that the relative Hamiltonian for some $\omega\stackrel{\infty}{=}\eta^+$ can decomposed into 
				$\mathcal{H}_{\phi}(\omega\vert \eta) = \sum_{\bar{\kappa}\in\mathcal{K}(\omega)}\mathcal{H}_{\phi}(\omega_{\bar{\kappa}}\omega^+_{\bar{\kappa}^c}\vert \eta^+)$.
				 Since we are only interested in configuration where $\bar{\kappa}^*$ is an element of $\mathcal{K}(\omega)$ we can write
				\begin{align*}
				\gamma^{sc}_{\Lambda,\beta,\lambda}(\{\omega\,:\; \bar{\kappa}^*\in \mathcal{K}(\omega)\} \vert \eta^+) &= e^{-\mathcal{H}_{\Lambda,\beta,\lambda}(\omega_{\bar{\kappa}^*}\vert \eta^+)} \frac{\sum_{\omega\,:\; \bar{\kappa}^*\in \mathcal{K}(\omega)}\prod_{\bar{\kappa}\in \mathcal{K}(\omega)\backslash\{\bar{\kappa}^*\}}e^{-\mathcal{H}_{\Lambda,\beta,\lambda}(\omega_{\bar{\kappa}}\vert \eta^+)}}{\sum_{\omega}\prod_{\bar{\kappa}\in \mathcal{K}(\omega)}e^{-\mathcal{H}_{\Lambda,\beta,\lambda}(\omega_{\bar{\kappa}}\vert \eta^+)}}\\
				&\leq e^{-\vert \bar{\kappa}^*\vert \rho}  \frac{\sum_{\omega\,:\; \bar{\kappa}^*\in \mathcal{K}(\omega)}\prod_{\bar{\kappa}\in \mathcal{K}(\omega)\backslash\{\bar{\kappa}^*\}}e^{-\mathcal{H}_{\Lambda,\beta,\lambda}(\omega_{\bar{\kappa}}\vert \eta^+)}}{\sum_{\omega}\prod_{\bar{\kappa}\in \mathcal{K}(\omega)}e^{-\mathcal{H}_{\Lambda,\beta,\lambda}(\omega_{\bar{\kappa}}\vert \eta^+)}}.
				\end{align*}
				It remains to show that $ \frac{\sum_{\omega\,:\; \bar{\kappa}^*\in \mathcal{K}(\omega)}\prod_{\bar{\kappa}\in \mathcal{K}(\omega)\backslash\{\bar{\kappa}^*\}}e^{-\mathcal{H}_{\Lambda,\beta,\lambda}(\omega_{\bar{\kappa}}\vert \eta^+)}}{\sum_{\omega}\prod_{\bar{\kappa}\in \mathcal{K}(\omega)}e^{-\mathcal{H}_{\Lambda,\beta,\lambda}(\omega_{\bar{\kappa}}\vert \eta^+)}}\leq1$. For this define the site-wise flip-function by
				\begin{align*}
				{F}^{{\kappa}^*}_i(\omega) = \left\{\begin{array}{lcc}
				\omega_i & \text{if }& i \in A_0\\
				\type({\kappa}^*)& \text{if }& i\in \bar{\kappa}^*\\
				\omega_i& \text{if }& i\in A_j \text{ for some } j \text{ and } \lab(A_j) = \type({\kappa}^*)\\
				-\omega_i& \text{if } &i\in A_j \text{ for some } j \text{ and } \lab(A_j) \neq \type({\kappa}^*)\end{array}
				\right. ,
				\end{align*}
				where $\{A_0,A_1,\ldots, A_k\}$ are given by the decomposition of $\mathcal{K}$.\\
				For a configuration $\omega$ with $\bar{\kappa}^* \in\mathcal{K}(\omega)$ the function ${F}^{{\kappa}^*}$ erases the contour $\kappa^*$ but leaves every other contour unchanged beside a possible spin flip. Write $\mathcal{F}(\kappa^*)$ for the set of configurations where the contour $\kappa^*$ has been removed. Since the relative Hamiltonian of the soft-core Widom-Rowlinson model is invariant under spin flip we get
				\begin{align*}
				e^{-\vert \bar{\kappa}^*\vert \rho}  \frac{\sum_{\omega\,:\; \bar{\kappa}^*\in \mathcal{K}(\omega)}\prod_{\bar{\kappa}\in \mathcal{K}(\omega)\backslash\{\bar{\kappa}^*\}}e^{-\mathcal{H}_{\Lambda,\beta,\lambda}(\omega_{\bar{\kappa}}\vert \eta^+)}}{\sum_{\omega}\prod_{\bar{\kappa}\in \mathcal{K}(\omega)}e^{-\mathcal{H}_{\Lambda,\beta,\lambda}(\omega_{\bar{\kappa}}\vert \eta^+)}}=e^{-\vert \bar{\kappa}^*\vert \rho}  \frac{\sum_{\omega\in\mathcal{F}(\kappa^*)}\prod_{\bar{\kappa}\in \mathcal{K}(\omega)}e^{-\mathcal{H}_{\Lambda,\beta,\lambda}(\omega_{\bar{\kappa}}\vert \eta^+)}}{\sum_{\omega}\prod_{\bar{\kappa}\in \mathcal{K}(\omega)}e^{-\mathcal{H}_{\Lambda,\beta,\lambda}(\omega_{\bar{\kappa}}\vert \eta^+)}}.
				\end{align*}
				The summation over $\omega\in\mathcal{F}(\kappa^*)$ is a restriction with respect to sum over all configuration and the fraction can be bounded by $1$.
			\end{proof}
			For configurations $\omega \stackrel{\infty}{=} \eta^+$ and $\omega_0 = 0$ or $\omega_0=-1$ there exists necessarily a contour $\kappa^*$ which is around the site $0$. By this we can prove the next lemma.
				\begin{lemma}\label{lem: bound for kernel phase}
					There exists a function $a(\beta,\lambda)$ such that $\lim_{\beta,\lambda\rightarrow\infty}a(\beta,\lambda)=0$ and
					\begin{align*}
					\gamma^{sc}_{\Lambda,\beta,\lambda}(\{\omega\in\Omega\,:\,\omega_0\in\{-1,0\}\}\vert \eta^+) \leq a(\beta,\lambda).
					\end{align*}
				\end{lemma}
				\begin{proof}
					For a configuration $\omega$ with $\omega_0=-1$ there are two cases. Either the site $0$ is inside the interior of a contour $\kappa^*$ or is an element of $\kappa^*$. If $\omega_0=0$ then the site $0$ is an element of $\kappa^*$. By this we can bound the measure by 
					\begin{align}
					&\gamma^{sc}_{\Lambda,\beta,\lambda}(\{\omega\,:\,\omega_0\in\{-1,0\}\}\vert \eta^+)\notag\\
					&\leq \sum_{\bar{\kappa}^*\,:\, 0\in\inter(\bar{\kappa}^*)} \gamma^{sc}_{\Lambda,\beta,\lambda}(\{\omega\,:\; \bar{\kappa}^*\in \mathcal{K}(\omega)\} \vert \eta^+) + \sum_{\bar{\kappa}^*\,:\, 0\in(\bar{\kappa}^*)} \kappa^{sc}_{\Lambda,\beta,\lambda}(\{\omega\,:\; \bar{\kappa}^*\in \mathcal{K}(\omega)\} \vert \eta^+) \notag\\
					&\leq \sum_{\bar{\kappa}^*\,:\, 0\in\inter(\bar{\kappa}^*)} e^{-\vert \bar{\kappa}^*\vert \rho} + \sum_{\bar{\kappa}^*\,:\, 0\in(\bar{\kappa}^*)} e^{-\vert \bar{\kappa}^*\vert \rho}\notag\\
					&\leq\sum_{k=2d+1}^\infty( \sum_{\stackrel{\bar{\kappa}^*\,:\, 0\in\inter(\bar{\kappa}^*)}{\vert \bar{\kappa}^*\vert=k}} e^{-\vert \bar{\kappa}^*\vert \rho} + \sum_{\stackrel{\bar{\kappa}^*\,:\, 0\in\bar{\kappa}^*}{\vert \bar{\kappa}^*\vert=k}} e^{-\vert \bar{\kappa}^*\vert \rho})\notag\\
					&=\sum_{k=2d+1}^\infty e^{-k \rho}( \#\{\bar{\kappa}^*\,:\, 0\in\inter(\bar{\kappa}^*)\,,\,\vert \bar{\kappa}^*\vert=k \} + \#\{\bar{\kappa}^*\,:\, 0\in\bar{\kappa}^*\,,\, \vert\bar{\kappa}^*\vert=k \} )\notag\\
					&\leq \sum_{k=2d+1}^\infty (e^{-k \rho} (k(2d)^{2k}+(2d)^{2k})). \notag
					\end{align}
					The last inequality follows by \cite[Lemma 3.38]{friedli_velenik_2017}. As long as $e^{-\rho}(2d)^2$ is smaller than $1$ the sum is finite and it follows that
					\begin{align*}
					\gamma^{sc}_{\Lambda,\beta,\lambda}(\{\omega\,:\,\omega_0=\{0,-1\}\}\vert \eta^+) \leq \frac{e^{-\rho}(2d)^2}{(1-e^{-\rho}(2d)^2)^2}+ \frac{1}{1-e^{-\rho}(2d)^2}-1.
					\end{align*}
					Since $\rho = \frac{\min\{\beta,\log{\lambda}\}}{2d+1}$ goes to infinity for $\beta,\lambda\rightarrow \infty$ the right hand side of the inequality goes to $0$ and we can define $a(\beta,\lambda):= \frac{e^{-\rho}(2d)^2}{(1-e^{-\rho}(2d)^2)^2}+ \frac{1}{1-e^{-\rho}(2d)^2}-1$.
					\end{proof}
					
					We are now able to prove the phase transition for the hard-core and soft-core model.
					\begin{proof}[Proof of Theorem \ref{thm : Sc Phase}]
						Due to the $\pm$-spin-flip symmetry of the soft-core model the non-existence of a phase transition would imply that 
						$\mu^+(\mathds{1}(\omega_0=\cdot)) = 0$ since 
						\begin{align*}
							\mu^+(\mathds{1}(\omega_0=\cdot)) =\lim_{\Lambda\uparrow\Z^d } \gamma^{sc}_{\Lambda,\beta,\lambda}(\mathds{1}(\omega_0=\cdot)\vert \eta^+)= -\lim_{\Lambda\uparrow\Z^d }\gamma^{sc}_{\Lambda,\beta,\lambda}(\mathds{1}(\omega_0=\cdot)\vert \eta^-) = -\mu^-(\mathds{1}(\omega_0=\cdot)).
						\end{align*} 
						Hence it is enough for the existence of a phase transition that $\mu^+(\mathds{1}(\omega_0=\cdot)) =\lim_{\Lambda\uparrow\Z^d }\gamma^{sc}_{\Lambda,\beta,\lambda}(\mathds{1}(\omega_0=\cdot)\vert \eta^+)>0$.
							A short calculation gives
							\begin{align*}
							\gamma^{sc}_{\Lambda,\beta,\lambda}(\mathds{1}(\omega_0=\cdot)\vert \eta^+)=1-\gamma^{sc}_{\Lambda,\beta,\lambda}(\omega_0\in\{0,-1\}\vert\eta^+)-\gamma^{sc}_{\Lambda,\beta,\lambda}(\omega_0=-1\vert \eta^+)
							\end{align*}
							and $\gamma^{sc}_{\Lambda,\beta,\lambda}(\omega_0=-1\vert \eta^+)$ can be bounded by $\gamma^{sc}_{\Lambda,\beta,\lambda}(\omega_0\in\{0,-1\}\vert \eta^+)$. This implies $\gamma^{sc}_{\Lambda,\beta,\lambda}(\omega_0=1\vert \eta^+)>1-2a(\beta,\lambda)$ and since $\lim_{\beta,\lambda\rightarrow \infty}a(\beta,\lambda)=0$ there exists $\beta_c$ and $\lambda_c$ such that $a(\beta,\lambda)<\frac{1}{2}$ for all $\beta\geq\beta_c$ and $\lambda\geq\lambda_c$.
					
					\end{proof}

		\begin{proof}[Proof of Corollary \ref{thm : Hc Phase}]
			By Lemma \ref{lem: bound for kernel phase} we have for  $\lambda\mapsto a(\lambda)= \lim_{\beta\rightarrow \infty} a(\beta,\lambda)$ that $$\gamma^{hc}_{\Lambda,\beta,\lambda}(\{\omega\in\Omega\,:\,\omega_0\in\{-1,0\}\}\vert \eta^+) =\lim_{\beta\rightarrow \infty}\gamma^{sc}_{\Lambda,\beta,\lambda}(\{\omega\in\Omega\,:\,\omega_0\in\{-1,0\}\}\vert \eta^+) \leq a(\lambda)$$ since the Peierls constant is given in terms of the minimum of $\log(\lambda)$ and $\beta$. By the arguments as in proof of Theorem \ref{thm : Sc Phase} the phase transition follows.
		\end{proof}
		\subsection{Regions of Dobrushin uniqueness}\label{sec: dob uniq}
		We start with the hard-core model.
		\begin{proof}[Proof of Theorem \ref{thm: dob hc}]
			The single-site probability measures reduce to 
			\begin{align*}
			\gamma_i^0 (\cdot \vert \eta) = \frac{\mathds{1}(\omega_i\eta_j\neq-1 \,:\, \forall j \sim i) \al(\cdot)}{\sum_{\tilde{\omega}\in\{-1,0,1\}} \mathds{1}(\tilde{\omega}_i\eta_j\neq-1 \,:\, \forall j \sim i) \al(\tilde{\omega_i})}.
			\end{align*} 
			Because of the hard-core restriction, there are only 4 different probability measures. %, depending on the existence of negative or positive particle which are connected to the vertex $i$
			The indicator $\mathds{1}(\omega_i\eta_j\neq-1 \,:\, \forall j \sim i)$ is equal to $0$ if there exists one vertex $j$ with $\omega_i\eta_j=-1$ and it does not matter if there are one or more vertices connected with $i$ which have this property. In the following for shorter notation $1\in \eta$ means that there exists a vertex $j$ with $i\sim j$ and $\eta_j=1$, and similar for the other cases. The 4 measures are:
	\begin{align*}
		\gamma_i^0(A\vert \eta)=	\left\{\begin{array}{cl} \delta_0(A) ,&  \text{ if }  1  \text{ and }  -1 \in \eta\\
		\frac{\sum_{k\in A}\al(k) }{\sum_{k\in E}\al(k)} , & \text{ if } \eta  \text{ contains only }  0 \\ 
		\frac{\sum_{k\in A}\al(k) \mathds{1}_{\{k\neq-1\}}}{\sum_{k\in \{0,1\}}\al(k)} &  \text{ if } -1\notin \eta  \text{ and }  1\in \eta \\
		\frac{\sum_{k\in A}\al(k) \mathds{1}_{\{k\neq1\}}}{\sum_{k\in \{0,-1\}}\al(k)} &  \text{ if } 1\notin \eta  \text{ and }  -1\in \eta  \end{array} \right. .
	\end{align*}

		By pair-wise comparing of the 4 measures, except the first with the second one, the proof follows.
		\end{proof}
		For the soft-core model the case is different. Here one have to care how many pluses and minuses are in the boundary condition. Therefore we denote by $\eta_i^{\pm}:=\vert\{j\sim i\,:\, \omega_j = \pm 1\}\vert$ the number of pluses and minuses connected to the site $i$, respectively. The next lemma gives a representation for the $C_{ij}$.

	\begin{lemma}\label{lem: C_ij_sc}
		Let $i,j\in V$ with $i \sim j$. In the soft-core Widom-Rowlinson model  $C_{ij}$ is given by 
		\begin{align*}
		&C_{ij}(\gamma^{sc}_{\beta,\al})= \max\\
		&\Biggl\{\max_{\stackrel{0\leq\pl+\mi}{\leq B_i-1}}  \frac{\al(-1)(\al(0)(e^{-\beta\pl}-e^{-\beta(\pl+1)})+\al(1)(e^{-\beta(\pl+\mi)}-e^{-\beta(\pl+\mi+1)})}{( \al(0)+\al(1) e^{-\beta \mi}+\al(-1)e^{-\beta\pl})(\al(0)+\al(1) e^{-\beta \mi}+\al(-1)e^{-\beta(\pl+1)})},\\
		&\max_{\stackrel{0\leq\pl+\mi}{\leq B_i-1}}  \frac{\al(1)(\al(0)(e^{-\beta\mi}-e^{-\beta(\mi+1)})+\al(-1)(e^{-\beta(\pl+\mi)}-e^{-\beta(\pl+\mi+1)})}{( \al(0)+\al(1) e^{-\beta \mi}+\al(-1)e^{-\beta\pl})(\al(0)+\al(1) e^{-\beta (\mi+1)}+\al(-1)e^{-\beta\pl})},\\
		&\max_{\stackrel{0\leq\pl+\mi\leq B_i}{\mi>0}}\frac{\al(1)(\al(0)(e^{-\beta (\mi-1)}-e^{-\beta \mi})+\al(-1)(e^{-\beta(\pl+\mi-1)}-e^{-\beta(\pl+\mi+1)})\mathds{1}_{A_i}(\eta_i^+,\eta_i^-)}{(\al(0)+\al(1) e^{-\beta \mi}+\al(-1)e^{-\beta\pl})(\al(0)+\al(1) e^{-\beta (\mi-1)}+\al(-1)e^{-\beta(\pl+1)})},\\
		&\max_{\stackrel{0\leq\pl+\mi\leq B_i}{\mi>0}}\frac{\al(-1)(\al(0)(e^{-\beta \pl}-e^{-\beta( \pl+1)})+\al(1)(e^{-\beta(\pl+\mi-1)}-e^{-\beta(\pl+\mi+1)})\mathds{1}_{A^c_i}(\eta_i^+,\eta_i^-)}{(\al(0)+\al(1) e^{-\beta \mi}+\al(-1)e^{-\beta\pl})(\al(0)+\al(1) e^{-\beta (\mi-1)}+\al(-1)e^{-\beta(\pl+1)})}\Biggr\}
		\end{align*}
	 where $A_i := \{ (a,b) \in \{0,\ldots,B_i\}^2\,:\; \frac{\al(1)}{\al(-1)} e^{-\beta(b-a-1)} >1 \}$.\\
	 If $i\nsim j$ then $C_{ij}=0$.
	\end{lemma}

\begin{proof}
	Again the single-site probability-kernels reduce to 
	\begin{align*}
	\gamma_i^0 (\cdot \vert \eta)= \frac{e^{-\beta\sum_{i\sim j}\mathds{1}(\omega_i\eta_j=-1)}  \al(\cdot)}{\sum_{\tilde{\omega}\in\{-1,0,1\}} e^{-\beta\sum_{i\sim j}\mathds{1}(\tilde{\omega}_i\eta_j=-1)}  \al(\tilde{\omega})}. 
	\end{align*} 
	 With the definitions of $\eta_i^{\pm}$ one can write 
	\begin{align*}
	\gamma_i^0 (\{0\} \vert \eta) &= \frac{\al(0)}{\al(0)+\al(1) e^{-\beta \mi}+\al(-1)e^{-\beta\pl}},\\ 
	\gamma_i^0 (\{1\} \vert \eta) &= \frac{\al(1)e^{-\beta\mi}}{\al(0)+\al(1) e^{-\beta \mi}+\al(-1)e^{-\beta\pl}},\\ 
	\gamma_i^0 (\{-1\} \vert \eta) &= \frac{\al(-1)e^{-\beta\pl}}{\al(0)+\al(1) e^{-\beta \mi}+\al(-1)e^{-\beta\pl}}.
	\end{align*}
	
	To compute $C_{ij}$ we fix some boundary condition $\eta$. The second boundary condition $\zeta$ shall only differ by one site. So only 3 interesting cases exist: 1) $0 \leftrightarrow 1$ , 2) $0 \leftrightarrow -1$ and 3) $-1 \leftrightarrow 1$, so far it is possible. Since the total-variation distance is symmetric we only need to check one direction.\\
	We will only consider the first case since the computation are similar for the other cases. This means in the first case we change a $0$ in the boundary condition $\eta$ to a positive spin value to get the second boundary condition $\zeta$. Since only one site is different we have the relation $\zeta_i^+=\pl+1$ and $\zeta^-_i = \mi$. Hence
	\begin{align*}
	&2 d_{TV}(\gamma_i^0(\cdot\vert \eta),\gamma_i^0(\cdot\vert {{\zeta}})) \\
	&= \al(0)\left (-\frac{1}{\al(0)+\al(1) e^{-\beta \mi}+\al(-1)e^{-\beta\pl}} + \frac{1}{\al(0)+\al(1) e^{-\beta \mi}+\al(-1)e^{-\beta(\pl+1)}} \right) \\
	&+ \al(1) \left(- \frac{e^{-\beta\mi}}{\al(0)+\al(1) e^{-\beta \mi}+\al(-1)e^{-\beta\pl}} + \frac{e^{-\beta\mi}}{\al(0)+\al(1) e^{-\beta \mi}+\al(-1)e^{-\beta(\pl+1)}}  \right)\\
	&+ \al(-1) \left( \frac{e^{-\beta\pl}}{\al(0)+\al(1) e^{-\beta \mi}+\al(-1)e^{-\beta\pl}}- \frac{e^{-\beta(\pl+1)}}{\al(0)+\al(1) e^{-\beta \mi}+\al(-1)e^{-\beta(\pl+1)}} \right)\\ 
	&= 2 \al(-1) \frac{\al(0)(e^{-\beta\pl}-e^{-\beta(\pl+1)})+\al(1)(e^{-\beta(\pl+\mi)}-e^{-\beta(\pl+\mi+1)})}{( \al(0)+\al(1) e^{-\beta \mi}+\al(-1)e^{-\beta\pl})(\al(0)+\al(1) e^{-\beta \mi}+\al(-1)e^{-\beta(\pl+1)})}.
	\end{align*}
	Since we have to ensure in order to change a $0$ to a $1$ that not all sites are occupied by a particle for the boundary condition $\eta$. Therefore one have the restriction $\eta_i^++\eta_i^-\leq B_i -1$.
\end{proof}

The fractions in Lemma \ref{lem: C_ij_sc} do not depend on $B_i$. Hence we have for $i,k\in V$ with $B_k\leq B_i$ some monotonicity property $C_{kj}(\gamma^{sc}_{\beta,\al})\leq C_{ij}(\gamma^{sc}_{\beta,\al})$ since we take the four maximums over a larger set. 
For the case $B = \infty$ the Dobrushin constant $c(\gamma^{sc}_{\beta,\al})$ is only finite for $\al\in \{\delta_{-1},\delta_0,\delta_0\}$ with value $0$. This is the reason why we need graphs with finite $B$. 

\begin{proof}[Proof of Theorem \ref{thm: neigh}]
		Since $B<\infty$ the set $B_{dg}:= \{ k \in \N\,:\, \exists i\in V \text{ s.t } B_i=k \}$ is finite and note that $C_{ij}(\gamma^{sc}_{\beta,\al})$ does not depend on $j$ for all $j\sim i$. Hence the Dobrushin constant can be written as $\sup_{i\in V} \sum_{j\in V} C_{ij}(\gamma^{sc}_{\beta,\al})= \max_{k\in B_{dg}} k C_k(\gamma^{sc}_{\beta,\al})$ where $C_k= C_{\tilde{i}j}$ with $B_{\tilde{i}}=k$. 
		
		Take some sequence $(\al_n)_{n\in \N}$ in $\mathcal{M}_{1}(E)$  with limit $\delta_1, \delta_0$ or $\delta_{-1}$. Since all maximizing for $c(\gamma^{sc}_{\beta,\al})$ is taken over finite sets we can pull the limit through all of it. Hence we have only to care about the fractions inside of the max. One can see that $C_k(\gamma^{sc}_{\beta,\delta_l})=0$ for all $l\in E$. This implies $\lim_{n\rightarrow \infty} c(\gamma^{sc}_{\beta,\al_n}) =0$ and therefore the existence of the neighborhoods follows by continuity.
\end{proof}

{Later for the time-evolved model only a priori measures with $\al(0)=0$ are important and for those measures the fractions in Lemma \ref{lem: C_ij_sc} are easier to handle. To analyze this case we introduce the function  
	\begin{align}\label{eq: g function}
	g(\beta,B) := -e^{-\beta B} \left(e^{2 \beta }(1-B)+B+1+\sqrt{\left(e^{2 \beta }-e^{2 \beta } B+B+1\right)^2-4 e^{2 \beta }}\right)
	\end{align} 
	which is related to the zeros of the polynomials mentioned after Theorem \ref{thm: dob hc}. As long as $\beta< \log(\frac{B+1}{B-1})$ non of the polynomials have real roots and consequently they are strictly smaller than $0$.}

\begin{corollary}\label{Coro: Dob pm}
	{Let $\al\in \mathcal{M}_1(E)$ with $\al(0)=0$ and $B<\infty$. If $\beta\geq \log(\frac{B+1}{B-1})$ and $\max\{\al(1),\al(-1)\}> \frac{2}{2+g(\beta,B)} $ then $c(\gamma^{sc}_{\beta,\al}) < 1$. Furthermore, if $\beta< \log(\frac{B+1}{B-1})$  then $c(\gamma^{sc}_{\beta,\al}) < 1$ for all $\al$ with $\al(0)=0$.}
\end{corollary} 
{The last part implies that for small $\beta$ every soft-core model with $\al(0)=0$ satisfies the Dobrushin condition. This bound is slightly better than what we get by an application of Theorem \ref{Thm: beta B} since $\log(\frac{B+1}{B-1})> \frac{2}{B}$ for all $B>1$. }
\begin{proof}
{	For $\al(0)=0$ we can sum the third and fourth fraction in Lemma \ref{lem: C_ij_sc}  because they differ only on terms which are multiplied by $\al(0)$.  Because of the monotonicity we need only to check that $B_i C_{ij}$ is smaller than one for $i$ with $B_i=B$. The above mentioned polynomials are now quadratic in one variable and the leading coefficient is negative. One can show that only the third fraction is important, one time with $\eta_i^+=B-1$ and $\eta_i^-=1$, and second time with  $\eta_i^+=0$ and $\eta_i^-=B$. By this the result follows by an easy but long computation. }
\end{proof}

	\section{Proofs for the time-evolved models}\label{Sec: 4}
	We will use different methods to analyze the two models. The already mentioned cluster representation for the hard-core model and for the soft-core model a method involving the restricted constrained first-layer model explicitly. The first-layer corresponds to the model at time $0$ and the second layer corresponds to the time-evolved model. We need to find a quasilocal specification for the time-evolved measure and a good starting point is to combine the specifications for the starting measures with the transition kernel $p_t$. We concentrate only on the hard-core case for a moment but all ideas work also for the soft-core specification. Let $\omega \in \Omega$, $\al\in\mathcal{M}_1(E)$, $\Lambda\Subset \Z^d$ and $t>0$ then $\gamma^{\omega}_{\Lambda,\al,t}(d\eta):= \gamma^{hc}_{\Lambda,\al}(\prod_{i\in\Lambda}p_t(\cdot_i,d\eta_i)\vert \omega_{\Lambda^c})$ defines a probability measure on $(\Omega,\mathcal{F})$ at time $t$. Next we introduce a second finite volume $\Delta\Subset\Z^d$  which is contained in $\Lambda$ and a boundary condition $\tilde\eta\in \Omega$. Since $\gamma^{\omega}_{\Lambda,\al,t}$ is a probability measure on a finite space we can define 
	\begin{align*}
		\gamma^{\omega}_{\Lambda,\Delta,\al,t}(f\vert \tilde\eta):= \gamma^{\omega}_{\Lambda,\al,t}(f\vert \tilde\eta_{\Delta^c})= \frac{\sum_{\eta_\Delta\in\Omega_\Delta}f(\eta_\Delta\tilde{\eta}_{\Delta^c})\gamma^{hc}_{\Lambda,\al}(\prod_{i\in\Lambda\backslash\Delta}p_t(\cdot_i,\tilde\eta_i)\prod_{i\in\Delta}p_t(\cdot_i,\eta_i)\vert \omega_{\Lambda^c})}{\gamma^{hc}_{\Lambda,\al}(\prod_{i\in\Lambda\backslash\Delta}p_t(\cdot_i,\tilde\eta_i)\vert \omega_{\Lambda^c})}
	\end{align*}
	where $f:\Omega\rightarrow \R$ is a bounded measurable function.
	
	If the limit $\lim_{\Lambda\uparrow \Z^d}\gamma^{\omega}_{\Lambda,\Delta,\al,t}(f\vert \tilde\eta)$ exists and does not depend on $\omega$ for all $\Delta\Subset\Z^d$ and all boundary conditions $\tilde\eta$ the resulting probability kernel is a good candidate to provide a specification for the time-evolved measure. We start with the soft-core model.
	
	\subsection{Short-time Gibbs for the soft-core model}
	The idea of the proof relies on an uniform Dobrushin condition for the restricted constrained first-layer model which is a model at time $0$ with a constraint $\eta$ coming from time $t$. {We extend the approach of \cite{kuelske-opoku08} where only transformation kernels are investigated which are strictly positive.}
	\begin{definition}
		Let $\Lambda \Subset \Z^d$ and $i\in\Z^d$ then the $i$-restricted constrained first-layer model of the soft-core Widom-Rowlinson model is defined by 
		\begin{align*}
		\gamma^{i}_{\Lambda,t}[\eta](\omega_{\Lambda\backslash i}\vert \bar{\omega} ) = \frac{e^{-\mathcal{H}_\Lambda^i(\omega_\Lai\bar\omega_{\Lambda^c})}\prod_{j\in\Lambda\backslash i}p_t(\omega_j,\eta_j)\al(\omega_j)}{\sum_{\tilde{\omega}_\Lai\in E^{\Lai}}e^{-\mathcal{H}_\Lambda^i(\tilde\omega_\Lai\bar\omega_{\Lambda^c})}\prod_{j\in\Lambda\backslash i}p_t(\tilde{\omega}_j,\eta_j)\al(\tilde\omega_j)}
		\end{align*} 
		where $	\mathcal{H}_\Lambda^i(\omega) = \sum_{\{k,j\}\in \mathcal{E}_\Lambda^b} \Phi^{i}_{\{k,j\}}(\omega)$ with $\Phi^{i}_{\{k,j\}}(\omega)=\Phi_{\{k,j\}}(\omega)\mathds{1}_{i\cap\{k,j\}=\emptyset}$. 
		
	\end{definition}
	One can check that $\gamma^i_{\Lambda,t}$ defines a quasilocal specification on the graph $\Z^d\backslash \{i\}$ since the Hamiltonian has finite range and $\prod_{j\in\Lai}p_t$ depends only on the sites inside of $\Lai$.
	\begin{theorem}\label{Thm: evo dob}
		Let $i\in\Z^d$. Then there exists a time $t_0(\beta,\al)>0$ such that for all $t< t_0(\beta,\al)$ and $\eta\in \Omega$ the {$i$}-restricted constrained first-layer model satisfies the Dobrushin condition uniformly in $\eta$.
	\end{theorem}
	\begin{proof}
		  Since the specification is quasilocal we have only to check the condition 
		$		\bar{c}_{i,t} :=  \sup_{\eta\in \Omega}\bar{c}_{i,t}[\eta] <1$
		where \begin{align*}
		\bar{c}_{i,t}[\eta] := \sup_{i_0\in \Z^d\backslash \{i\}} \sum_{k\in \Z^d\backslash \{i\}}\bar{C}^{\eta,i}_{i_0k,t} 
		\end{align*}
		with 
		\begin{align*}
		\bar{C}^{\eta,i}_{i_0k,t} := \sup_{\omega,\bar{\omega}\in \Omega_{\Z^d\backslash\{i\}},\omega_{k^c}=\bar{\omega}_{k^c}}d_{TV}(\gamma^{i}_{i_0,t}[\eta](\cdot\vert\omega) ,\gamma^{i}_{i_0,t}[\eta](\cdot\vert\bar{\omega}) ).
		\end{align*}
		Note that $\bar{C}^{\eta,i}_{i_0k,t} $ is equal to zero if $i_0$ and $k$ are not nearest neighbor and consequently $\bar{C}^{\eta,i}_{i_0k,t} $ does not depend on $k$. This implies that $\bar{c}_{i,t}[\eta] = \sup_{i_0\in \Z^d\backslash \{i\}} \sum_{k\sim{i_0}} \bar{C}^{\eta,i}_{i_0k,t}$.   The $\eta$-dependence in 	$\bar{C}^{\eta,i}_{i_0k,t}$ occurs only at the site $i_0$. Hence it is useful to split the proof with respect to the possible values of $\eta_{i_0}$ and we can write $\bar{C}^{\eta,i}_{i_0k,t}=\bar{C}^{\eta_{i_0},i}_{i_0k,t}$. We start with $\eta_{i_0}=1$ and obtain in this case 
		\begin{align*}
		\gamma^{i}_{i_0,t}[\eta](\omega_{i_0}\vert\omega) = \left\{\begin{array}{ccl}
		\frac{e^{-\beta \miO}p_t(1,1)\alpha(1)}{\sum_{\tilde{\omega}_i\in\{-1,1\}}e^{-\sum_{j\sim i_0, j\neq i}\beta \mathds{1}(\tilde{\omega}_i\omega_j =-1)}p_t(\tilde{\omega}_i,1)\alpha(\tilde{\omega}_i)}&\text{if}& \omega_{i_0}=1\\
		0&\text{if} & \omega_{i_0}=0\\
		\frac{e^{-\beta \plO}p_t(-1,1)\alpha(-1)}{\sum_{\tilde{\omega}_i\in\{-1,1\}}e^{-\sum_{j\sim i_0,j\neq i}\beta \mathds{1}(\tilde{\omega}_i\omega_j =-1)}p_t(\tilde{\omega}_i,1)\alpha(\tilde{\omega}_i)}&\text{if}& \omega_{i_0}=-1
		\end{array}\right.
		\end{align*}
		where $\omega_{i_0}^{\pm,i}(\omega):= \vert \{j\in \Z^d\backslash \{i\} \;:\; j\sim i_0 \,,\, \omega_j= \pm 1  \}\vert$. Multiplying numerator and denominator by  $\frac{1}{p_t(1,1)\al(1)+p_t(-1,1)\al(-1)}$ yields
		\begin{align*}
		\gamma^{i}_{i_0,t}[\eta](\omega_{i_0}\vert\omega)= \left\{\begin{array}{ccl}
		\frac{e^{-\beta \miO}\tilde{\al}^1_t(1)}{\sum_{\tilde{\omega}_{i_0}\in\{-1,1\}}e^{-\sum_{j\sim i_0, j\neq i}\beta \mathds{1}(\tilde{\omega}_i\omega_j =-1)}\tilde{\al}^1_t(\tilde{\omega}_{i_0})}&\text{if}& \omega_{i_0}=1\\
		0&\text{if} & \omega_{i_0}=0\\
		\frac{e^{-\beta \plO}\tilde{\al}^1_t(-1)}{\sum_{\tilde{\omega}_{i_0}\in\{-1,1\}}e^{-\sum_{j\sim i_0, j\neq i}\beta \mathds{1}(\tilde{\omega}_i\omega_j =-1)}\tilde{\al}^1_t(\tilde{\omega}_{i_0})}
		&\text{if}& \omega_{i_0}=-1
		\end{array}\right.
		\end{align*}
		where  $$\tilde{\alpha}^{\eta_j}_t(\omega_{j}):= \frac{p_t(\omega_j,\eta_j)\al(\omega_j)}{p_t(1,\eta_j)\al(1)+p_t(0,\eta_j)\al(1)+p_t(-1,\eta_j)\al(-1)}.$$ Obviously $\tilde{\al}^{\eta_j}_t$ is a probability measure on $E$. This implies that we are in the same situation for the single-site kernels as in Section \ref{sec: dob uniq} 
		with the locally finite graph $\Z^d\backslash\{i\} $. Since $\lim_{t\rightarrow 0}\tilde{\al}^1_t(1) = 1$ Theorem \ref{thm: neigh} implies that there exists a $\tilde{t}_0>0$ such that for all $t<\tilde{t}_0$ the  $\bar{C}^{1,i}_{i_0k,t}$ are smaller then $\frac{1}{2d}$. Similarly it follows for $\eta_{i_0}=-1$ that there exists a $\bar{t}_0$ such that for all $t<\bar{t}_0$ the $\bar{C}^{-1,i}_{i_0k,t}$ are smaller then $\frac{1}{2d}$. For $\eta_{i_0}=0$ follows that $\gamma^{i}_{i_0,t}[\eta](\omega_{i_0}\vert\omega) = \delta_{0}(\omega_{i_0})$ and this implies $\bar{C}^{0,i}_{i_0k,t}=0$ since it does not depend on the boundary condition.\\
		A further look reveals that the only $i_0$-dependence of $\bar{C}^{\pm,i}_{\bar{i}_0k,t}$ comes from the two cases that $i_0$ and $i$ are nearest neighbors in $\Z^d$, or not. But by the comment after Lemma \ref{lem: C_ij_sc} we have $\bar{C}^{\pm1,i}_{\tilde{i}_0k,t} \leq \bar{C}^{\pm1,i}_{\bar{i}_0k,t}$  where $\tilde{i}_0$  is a neighbor of $i$ in $\Z^d$ and $\bar{i}_0$ is not. With $t_0=\min\{\tilde{t}_0,\bar{t}_0\}$ it follows that for all $t<t_0$ we have  $\max\{\bar{C}^{1,i}_{{i}_0k,t} ,\bar{C}^{-1,i}_{{i}_0k,t} \}<\frac{1}{2d}$. With this bound we can show that 
		\begin{align*}
		\bar{c}_{i,t} \leq \sup_{\eta\in \Omega}\sup_{i_0\in \Z^d\backslash \{i\}} 2d \bar{C}^{\eta,i}_{i_0k,t}<2d\max\{\bar{C}^{1,i}_{{i}_0k,t} ,\bar{C}^{-1,i}_{{i}_0k,t} \}<1
		\end{align*}
		which implies Dobrushin uniqueness uniformly in $\eta\in \Omega$. 
	\end{proof}
	
	\begin{corollary}\label{Corollary conv}
		For all $\alpha\in\mathcal{M}_1(E)$, $\eta \in \Omega$ and $\beta>0$ there exists an $t_0(\beta,\al)>0$ such that for all $t<t_0(\beta,\al)$ and $i\in \Z^d$ we have local convergence of $\gamma^{i}_{\Lambda,t}[\eta](\cdot\vert \bar{\omega} )$ with limit $\mu^i_{i^c,t}[\eta]$ where this measure is the unique Gibbs measure for the $i$-restricted constrained first-layer model. Moreover, $\eta \mapsto \mu^i_{i^c,t}[\eta]$ is measurable w.r.t. the evaluation $\sigma$-algebra.
	\end{corollary}
	\begin{proof}
		The convergence follows by \cite[Proposition 7.11]{georgii-book} since there exists a unique Gibbs measure by the Dobrushin uniqueness Theorem.
		
		For the last part, by standard arguments it suffices to show that  $\eta \mapsto \mu^i_{i^c,t}[\eta](A)$ is a measurable function for all local events $A$. Now, for arbitrary $\eta$-independent boundary condition $\bar{\omega}$ we have that $\mu^i_{i^c,t}[\eta](A) = \lim_{\Lambda\uparrow \Z^d} \gamma^{i}_{\Lambda,t}[\eta](A\vert \bar{\omega} )$ is measurable as limit of the measurable functions $\eta \mapsto \gamma^{i}_{\Lambda,t}[\eta](A\vert \bar{\omega} )$ which take only finely many values·
	\end{proof}
	Corollary \ref{Corollary conv} remains true if we replace $i$ with some $\Delta \Subset \Z^d$ and write
	\begin{align*}
	\gamma^{\Delta}_{\Lambda,t}[\eta](\omega_{\Lambda\backslash \Delta}\vert \bar{\omega} )  = \frac{e^{-H^\Delta_{\Lambda}(\omega_{\Lambda\backslash \Delta}\bar{\omega}_{\Lambda^c})}\prod_{j\in\Lambda\backslash \Delta}p_t(\omega_j,\eta_j)\al(\omega_j)}{\sum_{\tilde{\omega}_{\Lambda\backslash \Delta}\in \{-1,0,1\}^{\Lambda\backslash \Delta}}e^{-H^\Delta_{\Lambda}(\tilde{\omega}_{\Lambda\backslash \Delta}\bar{\omega}_{\Lambda^c})}\prod_{j\in\Lambda\backslash \Delta}p_t(\tilde{\omega}_j,\eta_j)\al(\omega_j)}
	\end{align*}
	where the $\Delta$-restricted Hamiltonian is defined by
$	\mathcal{H}_\Lambda^\Delta(\omega) = \sum_{\{k,j\}\in \mathcal{E}_\Lambda^b} \Phi^{\Delta}_{\{k,j\}}(\omega)$
	with the $\Delta$-restricted potential $\Phi^{\Delta}_{\{k,j\}}(\omega)=\Phi_{\{k,j\}}(\omega)\mathds{1}_{\Delta\cap\{k,j\}=\emptyset}$.
	Furthermore, $t_0$ is uniformly in $\Delta$ since thinning of the graph improves the Dobrushin constant. 
	
	The reason why we look at the restricted constrained model is that with its help we can easily rewrite  $\gamma^{\bar{\omega}}_{t,\Lambda,\Delta,\alpha,\beta}$ and show that it has a infinite-volume limit as $\Lambda\uparrow \Z^d$.
	\begin{lemma}
		Let $ \Lambda \Subset \Z^d$, with $\vert\Lambda\vert \geq 2$, $\alpha\in\mathcal{M}_1(E)$ and $\beta>0$. Then for every $\Delta \subset \Lambda$ and every boundary condition $\bar{\omega} \in \Omega$ at time $0$ and boundary condition $\eta\in \Omega$ the conditional probability
		$\gamma^{\bar{\omega}}_{t,\Lambda,\Delta,\alpha,\beta}$ can be rewritten as
		\begin{align*}
		&\gamma^{\bar{\omega}}_{\Lambda,\Delta,\alpha,\beta,t}(\eta_{\Delta}\vert \eta_{\Lambda \backslash \Delta} ) = \frac{\sum_{\omega_{\Lambda \backslash \Delta}\in \Omega_{\Lai}}\gamma^{\Delta}_{\Lambda,t}[\eta](\omega_{\Lambda\backslash \Delta}\vert \bar{\omega} ) \sum_{{\omega_\Delta}\in \Omega_\Delta}e^{-\mathcal{H}_\Delta({\omega_\Lambda\bar{\omega}_{\Lambda^c})}}\prod_{i\in \Delta}p_t(\omega_i,\eta_i)\alpha(\omega_i)}{\sum_{\omega_{\Lambda \backslash \Delta}\in \Omega_{\Lambda \backslash \Delta}}\gamma^{\Delta}_{\Lambda,t}[\eta](\omega_{\Lambda\backslash \Delta}\vert \bar{\omega} ) \sum_{{\omega_\Delta}\in \Omega_\Delta}e^{-\mathcal{H}_\Delta({\omega_\Lambda\bar{\omega}_{\Lambda^c})}}\prod_{i\in \Delta}\alpha(\omega_i)}.
		\end{align*}	 
		
	\end{lemma}
	
	\begin{proof}
		
		Splitting the Hamiltonian $\mathcal{H}_{\Lambda}(\omega_\Lambda\omega_{\Lambda^c}) =H^{\Delta}_{\Lambda}(\omega_{\Lambda\backslash \Delta}\omega_{\Lambda^c}) +\mathcal{H}_{\Delta}(\omega_\Lambda\omega_{\Lambda^c})$ and the sum in the definition of $\gamma^{\bar{\omega}}_{t,\Lambda,\Delta,\alpha,\beta}(\eta_{\Delta}\vert \eta_{\Lambda\backslash \Delta} )$ over $ \Omega_{\Lambda}$ into one  over $\Omega_{\Lambda\backslash \Delta}$ and one over $\Omega_{\Delta}$ gives the desired result.

	\end{proof}
	\begin{lemma}\label{lem: exis. speci}
		Let 	$\alpha\in\mathcal{M}_1(E)$, $\eta \in \Omega$ and $\beta>0$. Then there exists a $t_0(\beta,\al)>0$  such that for all $t<t_0(\beta,\al)$, all $\Delta\Subset \Z^d$ and all local bounded functions $f:\Omega\rightarrow \mathbb{R}$ it follows that 
		\begin{align*}
		\lim_{\Lambda\uparrow \Z^d } \gamma^{\bar{\omega}}_{\Lambda,\Delta,\alpha,\beta,t}(f\vert \eta_{\Lambda\backslash \Delta}) = \gamma_{\Delta,\alpha,\beta,t}(f\vert \eta_{ \Delta^c}).
		\end{align*} 
		with 
		\begin{align*}
		\gamma_{t,\Delta,\alpha,\beta}(\eta_\Delta\vert \eta_{ \Delta^c}) = \frac{\int_{ \Omega_{\Delta^c}}\mu_{ \Delta^c,t}[\eta_{\Delta^c}](d\omega_{\Delta^c} )\sum_{{\omega_\Delta}\in \Omega_\Delta}e^{-\mathcal{H}_\Delta({\omega_\Delta{\omega}_{\Delta^c})}}\prod_{i\in \Delta}p_t(\omega_i,\eta_i)\alpha(\omega_i)}{\int_{ \Omega_{\Delta^c}}\mu_{\Delta^c,t}[\eta_{\Delta^c}](d\omega_{\Delta^c} )\sum_{{\omega_\Delta}\in \Omega_\Delta}e^{-\mathcal{H}_\Delta({\omega_\Delta{\omega}_{\Delta^c})}}\prod_{i\in \Delta}\alpha(\omega_i)}. 
		\end{align*}
		where $\mu_{ \Delta^c}[\eta_{\Delta^c}]$ is the unique limit for the $\Delta$-restricted constrained first-layer model.
	\end{lemma}	   
	
	\begin{proof}
		First we choose $t_0$ small enough such that the $\Delta$-restricted constrained first-layer model satisfies the condition of Theorem \ref{Thm: evo dob} and consequently by Corollary \ref{Corollary conv} we have that 	$\lim_{\Lambda\uparrow \Z^d } \gamma^{\Delta}_{\Lambda,t}[\eta](\omega_{\Lambda\backslash \Delta}\vert \bar{\omega} ) (g)= \mu_{ \Delta^c}[\eta_{\Delta^c}](g )$ for all local bounded function $g:\Omega_{\Delta^c}\rightarrow \mathbb{R}$. For some local bounded function $f:\Omega\rightarrow \mathbb{R}$ define the function 
		\begin{align*}
		g^{\bar{\omega}}_{\Delta,\Lambda}(\omega_{\Delta^c},\eta) =  \sum_{\omega_{\Delta}\in \Omega_{\Delta}}e^{-\mathcal{H}_\Delta({\omega_\Lambda\bar{\omega}_{\Lambda^c})}}\prod_{i\in \Delta}p_t(\omega_i,\eta_i)\alpha(\omega_i) f(\eta).
		\end{align*}
		Since the Hamiltonian has only finite range we can choose $\Lambda$ big enough such that $g^{\bar{\omega}}_{\Delta,\Lambda}$ is independent of $\Lambda$ and write 
		\begin{align*}
		g^{\bar{\omega}}_{\Delta,\Lambda}(\omega_{\Delta^c},\eta)= g_{\Delta}(\omega_{\Delta^c},\eta)=\sum_{\omega_{\Delta}\in \Omega_{\Delta}}e^{-\mathcal{H}_\Delta({\omega_\Delta{\omega}_{\Delta^c})}}\prod_{i\in \Delta}p_t(\omega_i,\eta_i)\alpha(\omega_i) f(\eta).
		\end{align*}
		Additionally, the finite range property implies that $g_{\Delta}$ is a local function in $\omega$ and in $\eta$ such that we can rewrite
		\begin{align*}
		&\gamma^{\bar{\omega}}_{t,\Lambda,\Delta,\alpha,\beta}(f\vert \eta_{\Lambda \backslash \Delta} ) \\
		&=\sum_{\tilde{\eta}_{\Delta}\in \Omega_\Delta} \frac{\sum_{\omega_{\Lambda \backslash \Delta}\in \Omega_{\Lambda\backslash \Delta}}\gamma^{\Delta}_{\Lambda,t}[\eta](\omega_{\Lambda\backslash \Delta}\vert \bar{\omega} ) \sum_{{\omega_\Delta}\in \Omega_\Delta}e^{-\mathcal{H}_\Delta({\omega_\Lambda\bar{\omega}_{\Lambda^c})}}\prod_{i\in \Delta}p_t(\omega_i,\eta_i)\alpha(\omega_i)f(\tilde{\eta}_\Delta\eta_{\Delta^c})}{\sum_{\omega_{\Lambda \backslash \Delta}\in \Omega_{\Lambda \backslash \Delta}}\gamma^{\Delta}_{\Lambda,t}[\eta](\omega_{\Lambda\backslash \Delta}\vert \bar{\omega} ) \sum_{{\omega_\Delta}\in \Omega_\Delta}e^{-\mathcal{H}_\Delta({\omega_\Lambda\bar{\omega}_{\Lambda^c})}}\prod_{i\in \Delta}\alpha(\omega_i)}\\
		&=\sum_{\tilde{\eta}_{\Delta}\in \Omega_\Delta} \frac{\gamma^{\Delta}_{\Lambda,t}[\eta](g_{\Delta}(\cdot, \tilde{\eta}_{\Delta}\eta_{\Delta^c})\vert \bar{\omega})}{\gamma^{\Delta}_{\Lambda,t}[\eta](\sum_{\omega_{\Delta}\in \Omega_{\Delta}}e^{-\mathcal{H}_\Delta({\omega_\Delta{\cdot}_{\Delta^c})}}\prod_{i\in \Delta}\alpha(\omega_i)  \vert \bar{\omega})}.
		\end{align*}
		By taking the limit and with the help of Corollary \ref{Corollary conv} the proof is finished.
	\end{proof}   
	
	For the proof of short-time Gibbsianness we need the Dobrushin comparison Theorem which gives a bound on the difference of two Gibbs measure where one of them is admitted by some specification which satisfies the Dobrushin condition.
	
	\begin{theorem}\label{DLR short time}
		Let $\gamma$ and $\tilde{\gamma}$ be two specifications. Suppose $\gamma$  satisfies the Dobrushin condition. For each $i\in \Z^d$ we let $b_i$ be a measurable function on $\Omega$ such that 
		\begin{align*}
		d_{TV}(\gamma_i^0(\cdot\vert\omega),\tilde{\gamma}_i^0(\cdot\vert\omega)) \leq b_i(\omega)
		\end{align*}  
		for all $\omega\in \Omega$. If $\mu\in \mathcal{G}(\gamma)$ and  $\tilde{\mu}\in \mathcal{G}(\tilde{\gamma})$ then for all quasilocal bounded functions $f:\Omega \rightarrow \mathbb{R}$
		\begin{align*}
		\vert \mu(f)-\tilde{\mu}(f) \vert \leq \sum_{i,j\in \Z^d} \delta_i(f)D_{ij} \tilde{\mu}(b_j)
		\end{align*}
		where $\delta_i(f)= \sup_{\stackrel{\eta,\omega\in \Omega}{\eta_{\Z^d\backslash \{i\}}=\omega_{\Z^d\backslash \{i\}}}}\vert f(\eta)-f(\omega)\vert$ and $D:=(D_{ij})_{i,j\in \Z^d}:= \sum_{n=0}^\infty C^n$. Here $C^n$ is the n'th power of Dobrushin's interdependence matrix given by $\gamma$.
	\end{theorem}
	Actually this theorem is one of the ingredients to prove the Dobrushin uniqueness Theorem. It follows directly that there is at most one measure which is admitted by a specification which satisfies the Dobrushin condition. Assume that there exists two measures $\mu,\tilde{\mu} \in \mathcal{G}(\gamma)$ and $\gamma$ is specification which satisfies the Dobrushin condition then  $	\vert \mu(f)-\tilde{\mu}(f)\vert=0$ for every local bounded function $f$ since $b_i \equiv 0$. This implies $\mu = \tilde{\mu}$. We will use this theorem a bit differently now.
	\begin{lemma}\label{lem: admi. Dob}
		Let $\alpha \in \mathcal{M}_1(E),\beta>0$ and suppose $\mu$ is an arbitrary Gibbs measure for the soft-core Widom-Rowlinson model %with $\lim_{\Lambda\uparrow\Z^d } \gamma^{sc}_{\Lambda,\beta,\alpha}(\cdot\vert \bar{\omega}) =\mu^{\bar{\omega}}$ 
		then there exists a time $t_0(\beta,\al)>0$ such that for $t<t_0$ the time-evolved measure $\mu_t$ is admitted by the specification   $(\gamma_{\Delta,\alpha,\beta,t})_{\Delta\Subset \Z^d}$.
	\end{lemma}
	
	\begin{proof}
		It suffices to prove the lemma for extremal starting Gibbs measure $\nu$ since by the extremal decomposition  $\mu=\int_{\mathrm{ex}\,\mathcal{G}(\gamma^{sc}_{\beta,\al})} \nu\, \mathrm{w}_{\mu}(d\nu)$ we have $\mu_t = \int_{\mathrm{ex}\,\mathcal{G}(\gamma^{sc}_{\beta,\al})} \nu_t\, \mathrm{w}_{\mu}(d\nu)$. For more information about the extremal decomposition see \cite[Chapter 7.3]{georgii-book}.
		Let $f$ be a $\mathcal{F}_{\Delta}$-measurable bounded function. For $\Delta\Subset\Z^d$ it follows by the extremality of $\nu$ that there exists a boundary condition $\bar{\omega}\in \Omega$ with $\nu(f) = \lim_{\Lambda \uparrow \Z^d}  \gamma^{sc}_{\Lambda,\beta,a}(f\vert \bar\omega)$. Hence we have 
		\begin{align*}
		\nu_t(f) =\nu(p_t(f)) = \lim_{\Lambda \uparrow \Z^d} \gamma^{\bar{\omega}}_{\Lambda,\alpha,\beta,t}(f) = \lim_{\Lambda \uparrow \Z^d} \gamma^{\bar{\omega}}_{\Lambda,\alpha,\beta,t}(\gamma^{\bar{\omega}}_{\Lambda,\Delta,\alpha,\beta,t}(f\vert \cdot_{\Lambda\backslash \Delta})). 
		\end{align*}
		Let $\Gamma$ be a third finite subset of $\Z^d$ with $\Delta\subset\Gamma$ which allows us to estimate
		\begin{align*}
		&\vert \mu_t(f-\gamma_{\Delta,\alpha,\beta,t}(f\vert \cdot_{ \Delta^c}))\vert \\
		&\leq \vert \mu_t(f-\gamma^{\bar{\omega}}_{\Gamma,\Delta,\alpha,\beta,t}(f\vert \cdot_{\Gamma\backslash\Delta }))\vert + \vert \mu_t(\gamma^{\bar{\omega}}_{\Gamma,\Delta,\alpha,\beta,t}(f\vert \cdot_{\Gamma\backslash \Delta})-\gamma_{\Delta,\alpha,\beta,t}(f\vert \cdot_{ \Delta^c}))\vert \\
		&\leq \vert \mu_t(f-\gamma^{\bar{\omega}}_{\Gamma,\Delta,\alpha,\beta,t}(f\vert \cdot_{\Gamma\backslash \Delta}))\vert + \Vert \gamma^{\bar{\omega}}_{\Gamma,\Delta,\alpha,\beta,t}(f\vert \cdot_{\Gamma\backslash\Delta })-\gamma_{\Delta,\alpha,\beta,t}(f\vert \cdot_{ \Delta^c})\Vert\\
		&=\lim_{\Lambda \uparrow \Z^d } \vert \gamma^{\bar{\omega}}_{\Lambda,\alpha,\beta,t}(f-\gamma^{\bar{\omega}}_{\Gamma,\Delta,\alpha,\beta,t}(f\vert \cdot_{\Gamma\backslash \Delta}))\vert + \Vert \gamma^{\bar{\omega}}_{\Gamma,\Delta,\alpha,\beta,t}(f\vert \cdot_{\Gamma\backslash \Delta})-\gamma_{\Delta,\alpha,\beta,t}(f\vert \cdot_{ \Delta^c})\Vert\\
		&\leq \limsup_{\Lambda \uparrow \Z^d } \Vert \gamma^{\bar{\omega}}_{\Lambda,\Delta,\alpha,\beta,t}(f\vert \cdot_{\Lambda\backslash \Delta})-\gamma^{\bar{\omega}}_{\Gamma,\Delta,\alpha,\beta,t}(f\vert \cdot_{\Gamma\backslash \Delta})\Vert + \Vert \gamma^{\bar{\omega}}_{\Gamma,\Delta,\alpha,\beta,t}(f\vert \cdot_{\Gamma\backslash \Delta})-\gamma_{\Delta,\alpha,\beta,t}(f\vert \cdot_{ \Delta^c})\Vert\\
		&\leq \limsup_{\Lambda \uparrow \Z^d } \Vert \gamma^{\bar{\omega}}_{\Lambda,\Delta,\alpha,\beta,t}(f\vert \cdot_{\Lambda\backslash\Delta})-\gamma_{\Delta,\alpha,\beta,t}(f\vert \cdot_{ \Delta^c})\Vert + 2\Vert \gamma^{\bar{\omega}}_{\Gamma,\Delta,\alpha,\beta,t}(f\vert \cdot_{\Gamma\backslash \Delta})-\gamma_{\Delta,\alpha,\beta,t}(f\vert \cdot_{ \Delta^c})\Vert\\
		\end{align*}
		where $\Vert \gamma^{\bar{\omega}}_{\Lambda,\Delta,\alpha,\beta,t}(f\vert \cdot_{\Lambda\backslash\Delta})-\gamma_{\Delta,\alpha,\beta,t}(f\vert \cdot_{ \Delta^c})\Vert= \sup_{\eta\in \Omega}\vert \gamma^{\bar{\omega}}_{\Lambda,\Delta,\alpha,\beta,t}(f\vert \eta_{\Lambda\backslash\Delta})-\gamma_{\Delta,\alpha,\beta,t}(f\vert \eta_{ \Delta^c})\vert$. By this bound it is enough to show that $\Vert \gamma^{\bar{\omega}}_{\Lambda,\Delta,\alpha,\beta,t}(f\vert \cdot_{\Lambda\backslash\Delta})-\gamma_{\Delta,\alpha,\beta,t}(f\vert \cdot_{ \Delta^c})\Vert$ will be arbitrarily small if $\Lambda$ is growing. For this we introduce the functions \begin{align*}
		h_1(\omega_{\Lambda\backslash \Delta}\omega_{\Lambda^c},\eta_\Delta) = \sum_{{\omega_\Delta}\in \Omega_\Delta}e^{-\mathcal{H}_\Delta({\omega_\Delta{\omega}_{\Delta^c})}}\prod_{i\in \Delta}p_t(\omega_i,\eta_i)\alpha(\omega_i)
		\end{align*}
		and 
		\begin{align*}
		h_2(\omega_{\Lambda\backslash \Delta}\omega_{\Lambda^c}) = \sum_{{\omega_\Delta}\in \Omega_\Delta}e^{-\mathcal{H}_\Delta({\omega_\Delta{\omega}_{\Delta^c})}}\prod_{i\in \Delta}\alpha(\omega_i)
		\end{align*}
		such that we can write 
		\begin{align*}
		\vert \gamma^{\bar{\omega}}_{\Lambda,\Delta,\alpha,\beta,t}(f\vert \eta_{\Lambda\backslash\Delta})-\gamma_{\Delta,\alpha,\beta,t}(f\vert \eta_{ \Delta^c})\vert = \left\vert \sum_{\eta_\Delta\in \Omega_\Delta}\frac{\gamma^{\Delta}_{\Lambda,t}[\eta](f(\eta_\Delta)h_1(\cdot,\eta_\Delta)\vert \bar{\omega})}{\gamma^{\Delta}_{\Lambda,t}[\eta](h_2\vert \bar{\omega})}-\frac{\mu_{ \Delta^c,t}[\eta](f(\eta_\Delta)h_1(\cdot,\eta_\Delta))}{\mu_{ \Delta^c,t}[\eta](h_2)}\right\vert.
		\end{align*}
		Adding and subtracting a suitable middle term gives the bound
		%$
	%	\frac{\gamma^{\Delta}_{\Lambda,t}[\eta](f(\eta_\Delta)h_1(\cdot,\eta_\Delta)\vert \bar{\omega})\gamma^{\Delta}_{\Lambda,t}[\eta](h_2\vert \bar{\omega})}{\gamma^{\Delta}_{\Lambda,t}[\eta](h_2\vert \bar{\omega})\mu_{ \Delta^c,t}[\eta](h_2)}
	%	$
		\begin{align*}
		&\vert \gamma^{\bar{\omega}}_{\Lambda,\Delta,\alpha,\beta,t}(f\vert \eta_{\Lambda\backslash\Delta})-\gamma_{\Delta,\alpha,\beta,t}(f\vert \eta_{ \Delta^c})\vert \\
		&\leq \left\vert \sum_{\eta_\Delta\in \Omega_\Delta}\frac{\gamma^{\Delta}_{\Lambda,t}[\eta](f(\eta_\Delta)h_1(\cdot,\eta_\Delta)\vert \bar{\omega})}{\gamma^{\Delta}_{\Lambda,t}[\eta](h_2\vert \bar{\omega})}-\frac{\gamma^{\Delta}_{\Lambda,t}[\eta](f(\eta_\Delta)h_1(\cdot,\eta_\Delta)\vert \bar{\omega})\hspace{0.1cm}\gamma^{\Delta}_{\Lambda,t}[\eta](h_2\vert \bar{\omega})}{\gamma^{\Delta}_{\Lambda,t}[\eta](h_2\vert \bar{\omega})\hspace{0.1cm}\mu_{ \Delta^c,t}[\eta](h_2)}\right\vert\\
		&+\left\vert\sum_{\eta_\Delta\in \Omega_\Delta}\frac{\gamma^{\Delta}_{\Lambda,t}[\eta](f(\eta_\Delta)h_1(\cdot,\eta_\Delta)\vert \bar{\omega})\hspace{0.1cm}\gamma^{\Delta}_{\Lambda,t}[\eta](h_2\vert \bar{\omega})}{\gamma^{\Delta}_{\Lambda,t}[\eta](h_2\vert \bar{\omega})\hspace{0.1cm}\mu_{ \Delta^c,t}[\eta](h_2)}- \frac{\mu_{ \Delta^c,t}[\eta](f(\eta_\Delta)h_1(\cdot,\eta_\Delta))}{\mu_{ \Delta^c,t}[\eta](h_2)}\right\vert\\
		&= \left\vert\sum_{\eta_\Delta\in \Omega_\Delta}\gamma^{\Delta}_{\Lambda,t}[\eta](f(\eta_\Delta)h_1(\cdot,\eta_\Delta)\vert \bar{\omega})\frac{\mu_{ \Delta^c,t}[\eta](h_2)-\gamma^{\Delta}_{\Lambda,t}[\eta](h_2\vert \bar{\omega})}{\gamma^{\Delta}_{\Lambda,t}[\eta](h_2\vert \bar{\omega})\hspace{0.1cm}\mu_{ \Delta^c,t}[\eta](h_2)}\right\vert\\
		&+\left\vert\sum_{\eta_\Delta\in \Omega_\Delta}\frac{\gamma^{\Delta}_{\Lambda,t}[\eta](f(\eta_\Delta)h_1(\cdot,\eta_\Delta)\vert \bar{\omega})-\mu_{ \Delta^c,t}[\eta](f(\eta_\Delta)h_1(\cdot,\eta_\Delta))}{ \mu_{ \Delta^c,t}[\eta](h_2)}\right\vert\\
		&\leq \frac{\Vert f \Vert_{\infty}}{\mu_{ \Delta^c,t}[\eta](h_2)} \Bigg(\left\vert{\mu_{ \Delta^c,t}[\eta](h_2)-\gamma^{\Delta}_{\Lambda,t}[\eta](h_2\vert \bar{\omega})}{ }\right\vert + \sum_{\eta_\Delta\in \Omega_\Delta} \left\vert{\mu_{ \Delta^c,t}[\eta](h_1(\cdot,\eta_\Delta))-\gamma^{\Delta}_{\Lambda,t}[\eta](h_1(\cdot,\eta_\Delta)\vert \bar{\omega})}{ }\right\vert\Bigg).
		\end{align*}
		Note that the mapping $\eta\mapsto \mu_{ \Delta^c,t}[\eta](h_2)$ is $\mathcal{F}_{\Delta^c}$-measurable.
		We have shown that $\mu_{ \Delta^c,t}[\eta]$ is admitted by the specification $(\gamma^{\Delta}_{\Lambda,t}[\eta])_{\Lambda \Subset (\Z^d\backslash \Delta)}$ which satisfies for small $t$ the Dobrushin condition. We can interpret $\gamma^{\Delta}_{\Lambda,t}[\eta](\cdot\vert \bar{\omega})$ to be a measure admitted by the specification $(\gamma^{\Delta}_{\Lambda\cap\Lambda_1,t}[\eta])_{\Lambda_1\Subset (\Z^d\backslash \Delta)}$. Thus the single-site specifications $\gamma^{\Delta}_{\Lambda\cap\{i\},t}[\eta],\gamma^{\Delta}_{\{i\},t}[\eta]$ are equal whenever $i\in \Lambda$ and the total variation can be bounded by $1$ in the case where $i\notin \Lambda$. It follows from the Dobrushin comparison Theorem that 
		\begin{align*}
		\vert \gamma^{\bar{\omega}}_{\Lambda,\Delta,\alpha,\beta,t}(f\vert \eta_{\Lambda\backslash\Delta})-\gamma_{\Delta,\alpha,\beta,t}(f\vert \eta_{ \Delta^c})\vert 
		\leq  \frac{\Vert f \Vert_\infty}{\mu_{ \Delta^c,t}[\eta](h_2)} \sum_{i \in \Z^d\backslash \Delta }\Big[\delta_i(h_2)+\sum_{\eta_\Delta\in\Omega_\Delta}\delta_i(h_1(\cdot,\eta_\Delta))\Big]\sum_{j \in \Lambda^c} D_{ij} 
		\end{align*}  where $D_{ij}$ is given by the Dobrushin interdependence matrix of the restricted constrained first-layer model. Since the sum $ \sum_{j \in \Lambda^c} D_{ij} $ is finite for every $i$ and the $i$-sum is finite as $h_2$ is a local function it follows that 
		\begin{align*}
		\lim_{\Lambda \uparrow \Z^d}\vert \gamma^{\bar{\omega}}_{\Lambda,\Delta,\alpha,\beta,t}(f\vert \eta_{\Lambda\backslash\Delta})-\gamma_{\Delta,\alpha,\beta,t}(f\vert \eta_{ \Delta^c})\vert =0.
		\end{align*}
		Taking $\Gamma\uparrow\Z^d$ and using the same arguments as for $\Lambda$ the DLR-equation is proven. 
	\end{proof}
	The last part for proving short-time Gibbsianness is to show that $\gamma_{\Delta,\alpha,\beta,t}$ is quasilocal for small $t$. For this the Dobrushin comparison Theorem will be again a helpful tool.
	\begin{lemma}\label{lem: time quasi}
		Let $\al\in\mathcal{M}_1(E)$ and $\beta>0$. Then there exists a $t_0(\beta,\al)$ such that for all $t<t_0(\beta,\al)$ the specification $(\gamma_{\Delta,\alpha,\beta,t})_{\Delta\Subset \Z^d}$ is quasilocal.
	\end{lemma}  
	\begin{proof}
		An equivalent condition for quasilocality is to show that for all local bounded functions $f$ and all $\Delta\Subset\Z^d$
		\begin{align*}
		\lim_{\Lambda\uparrow \Z^d} \sup_{\genfrac{}{}{0pt}{}{{\eta,\bar\eta\in \Omega}}{\eta_\Lambda= \bar{\eta}_\Lambda}}\vert  \gamma_{\Delta,\alpha,\beta,t}(f\vert \eta ) - \gamma_{\Delta,\alpha,\beta,t}(f\vert \eta_{\Lambda}\bar{\eta}_{\Lambda^c} ) \vert =0.
		\end{align*}
		First we choose  $\Lambda$ big enough such that $f$ is $\mathcal{F}_\Lambda$-measurable and $\Delta \subset \Lambda$. Then we can use the same arguments as in the proof of Lemma \ref{lem: admi. Dob} to get
		\begin{align*}
		&\vert  \gamma_{\Delta,\alpha,\beta,t}(f\vert \eta ) - \gamma_{\Delta,\alpha,\beta,t}(f\vert \eta_{\Lambda}\bar{\eta}_{\Lambda^c} ) \vert \\
		&\leq  \frac{\Vert f \Vert_{\infty}}{\mu_{ \Delta^c,t}[\eta](h_2)} \Bigg(\left\vert{\mu_{ \Delta^c,t}[\eta](h_2)-\mu_{ \Delta^c,t}[\eta_\Lambda \bar{\eta}_{\Lambda^c}](h_2)}{ }\right\vert + \hspace{-2pt}\sum_{\tilde{\eta}_\Delta\in \Omega_\Delta} \left\vert{\mu_{ \Delta^c,t}[\eta](h_1(\cdot,\tilde{\eta}_\Delta))-\mu_{ \Delta^c,t}[\eta_\Lambda \bar{\eta}_{\Lambda^c}](h_1(\cdot,\tilde{\eta}_\Delta))}{ }\right\vert\Bigg).
		\end{align*}
		Now we can choose $t_0$ small enough such that the specification of the restricted constrained first-layer model satisfies the Dobrushin condition. Again we are in the situation where the Dobrushin comparison Theorem will be helpful. This time we have to compare the single-site kernels of the specifications $(\gamma^{\Delta}_{\Lambda,t}[\eta])_{\Lambda\Subset \Z^d\backslash \Delta}$ and $(\gamma^{\Delta}_{\Lambda,t}[\eta_{\Lambda}\bar{\eta}_{\Lambda^c}])_{\Lambda\Subset \Z^d\backslash \Delta}$ in total variational distance which coincide if $i\in \Lambda$. Therefore we can bound the distance by
		\begin{align*}
		d_{TV}(\gamma^{\Delta}_{i,t}[\eta](\cdot\vert\bar{\omega}) ,\gamma^{\Delta}_{i,t}[\eta_{\Lambda}\bar{\eta}_{\Lambda^c}](\cdot\vert{\bar{\omega}}) ) \leq \mathds{1}_{i\notin \Lambda }
		\end{align*}
		 By the Dobrushin comparison Theorem it follows again that 
		\begin{align*}
		\left\vert  \gamma_{\Delta,\alpha,\beta,t}(f\vert \eta ) - \gamma_{\Delta,\alpha,\beta,t}(f\vert \eta_{\Lambda}\bar{\eta}_{\Lambda^c} )\right\vert 
		\leq\frac{\Vert f \Vert_\infty}{\mu_{ \Delta^c,t}[\eta](h_2)} \sum_{i \in \Z^d\backslash \Delta }\Big[\delta_i(h_2)+\sum_{\tilde{\eta}_\Delta\in\Omega_\Delta}\delta_i(h_1(\cdot,\tilde{\eta}_\Delta))\Big]\sum_{j \in \Lambda^c} D_{ij} . 
		\end{align*}
		The function $h_2$ is bounded from below by $e^{-\beta \vert \mathcal{E}^b_\Delta\vert}$ and consequently $\mu_{ \Delta^c,t}[\eta](h_2)$ is bounded from below by the same bound. Hence the above is smaller than
		\begin{align*}
		{\Vert f \Vert_\infty}{e^{\beta \vert \mathcal{E}_\Delta^b\vert}} \sum_{i \in \Z^d\backslash \Delta }\Big[\delta_i(h_2)+\sum_{\tilde{\eta}_\Delta\in\Omega_\Delta}\delta_i(h_1(\cdot,\tilde{\eta}_\Delta))\Big]\sum_{j \in \Lambda^c} D_{ij}.
		\end{align*}
		The last expression does not depend on $\eta$ and $\bar\eta$. Furthermore, it goes to zero for $\Lambda\uparrow \Z^d$.
	\end{proof}
	Now we can prove the theorems for Gibbsianness of the time-evolved soft-core measure.
	\begin{proof}[Proof of Theorem \ref{thm: short time gibbs}]
		By Lemma \ref{lem: exis. speci} and Lemma \ref{lem: admi. Dob}  there exists a specification for the time-evolved measure. Furthermore, this specification is quasilocal by Lemma \ref{lem: time quasi}.

		\end{proof}
	For $\beta \geq \log(\frac{2d+1}{2d-1})$ with Corollary \ref{Coro: Dob pm} and the function $g$ defined by \eqref{eq: g function} we can give an explicit formula for $t_0$ since for the measures $\tilde\al_t^{\pm 1}$  we have $\tilde\al_t^{\pm 1}(0)=0$. Note that if $\al^1_t(1)> \frac{2}{2+g(\beta,2d)}$ and $\al^{-1}_t(-1)>\frac{2}{2+g(\beta,2d)} $ for some $t>0$ it follows that both inequalities holds for every $0<s<t$. These inequalities can be equivalently reformulated as $t< \atanh\Big({\frac{\al(\pm1)}{\al(\mp1)}}\frac{g(\beta,2d)}{2}\Big)$. Hence for all $$t<t_0:=\min\bigg\{ \atanh\bigg({\frac{\al(1)}{\al(-1)}}\frac{g(\beta,2d)}{2}\bigg), \atanh\bigg({\frac{\al(-1)}{\al(1)}}\frac{g(\beta,2d)}{2}\bigg)\bigg\}$$ the time-evolved measure $\mu_t$ is Gibbs. 
	\begin{proof}[Proof of Theorem \ref{thm: gibbs all time}]
		Since $\beta< \log(\frac{2d+1}{2d-1})$ every measure $\al$ with $\al(0)=0$ satisfies the  Dobrushin condition by the second part of Corollary \ref{Coro: Dob pm}. As a consequence  the restricted constrained first-layer model  satisfies the Dobrushin condition for all $t>0$ and all $\eta\in \Omega$. The rest of the proof is an application of the lemmas above with $t_0= \infty$.
	\end{proof}
	\begin{proof}[Proof of Theorem \ref{thm: gibbs high asm}]
		The only task we have to do is to show that there exists a $t_1$ such for all $t>t_1$ the restricted constrained first-layer model satisfies the Dobrushin condition uniformly in $\eta$.  From the discussion of Theorem \ref{Thm: evo dob} it is enough to show that $\al_t^\eta\in U^1_\epsilon\cup U^{-1}_\epsilon$ for all $t>t_1$ and all $\eta\in \{-1,1\}$. Note that $\eta=0$ is not important since $\al^0_t$ is again the Dirac measure in $0$. Starting with $\eta=1$ yields
		\begin{align*}
		\al_t^1(\omega)
		= \left\{\begin{array}{ccl} \frac{1}{1+ \frac{\al(-1)}{\al(1)}\tanh(t)} &\text{if}& \omega=1\\
		0 & \text{if} & \omega=0\\
		\frac{1}{1+ \frac{\al(1)}{\al(-1)}\cotanh(t)} &\text{if}& \omega=-1\end{array} \right..
		\end{align*}
		The function $t\mapsto\frac{1}{1+q\tanh(t)}$ is a monotonically decreasing function and $t\mapsto\frac{1}{1+q\cotanh(t)}$ monotonically increasing for $q>0$. Since $\lim_{t\rightarrow \infty} \al^{1}_t = \bar{\al}$ and by the continuity of $\tanh(t)$ it follows that there exists a $\bar{t}_1$ such that for all $t>\bar{t}_1$ the measure $\al^{1}_t \in U^1_\epsilon\cup U^{-1}_\epsilon$. With the same argument it follows that there exists an $\tilde{t}_1$ such that for all $t>\tilde{t}_1$ $\al_t^{-1}\in  U^1_\epsilon\cup U^{-1}_\epsilon$. By setting $t_1=\max\{\tilde{t}_1,\bar{t}_1\}$ we have $\al^{\pm}_t  \in U^1_\epsilon\cup U^{-1}_\epsilon$ for all $t>t_1$. The proofs follows again by  the above arguments, using Lemma \ref{lem: admi. Dob} and Lemma \ref{lem: time quasi} for $t>t_1$.
	\end{proof}
	
	\subsection{Loss of Gibbs for the soft-core model}
In this part we want to show that the time-evolved soft-core measure is not Gibbs if $\al(1)=\al(-1)$ and $t$ is large. Here it is more convenient to work with the parameters $h$ and $\lambda$, see Definition \ref{defi: models}. In \cite{enter-fernandez-hollander-redig02} the authors have proven that the time-evolved symmetric Ising model is not a Gibbs measure for large times. We want to use this result to prove something similar for the soft-core model. %For this we define a specification on the two layer system by
%\begin{align*}
%\gamma_{t,\Lambda}^{sc,2}((\omega_\Lambda,\eta_\Lambda)\vert (\omega_{\Lambda^c},\eta_{\Lambda^c}))= \frac{e^{-\mathcal{H}_{\Lambda}({\omega}_\Lambda\omega_{\Lambda^c})}\prod_{i\in\Lambda}p_t({\omega}_i,{\eta}_i)}{\sum_{\tilde{\omega}_{\Lambda}\in\Omega_\Lambda}\sum_{\tilde{\eta}_{\Lambda}\in\Omega_\Lambda} e^{-\mathcal{H}_{\Lambda}(\tilde{\omega}_\Lambda\omega_{\Lambda^c})}\prod_{i\in\Lambda}p_t(\tilde{\omega}_i,\tilde{\eta}_i)}
%\end{align*}
%where $\mathcal{H}_{\Lambda}(\omega):= \sum_{\{i,j\}\in\mathcal{E}_\Lambda^b } \beta\mathds{1}(\omega_i\omega_j=-1)-\sum_{i\in\Lambda }\log(\lambda)\omega_i^2$. 

For this we define the two-layer measure \begin{align*}
\mu_t^{sc,2}(d\omega,d\eta) := d\mu(\omega)p_t(\omega,d\eta)  
\end{align*}
on $\Omega \times \Omega$ where $\mu$ is a Gibbs-measure for the soft-core Widom-Rowlinson model with $h=0$ and $\lambda>0$. Note that we get the time-evolved measure by integrating $\mu_t^{sc,2}(d\omega,d\eta)$ over $\omega$. %If $\beta$ and $\lambda$ are big enough by Theorem \ref{thm : Sc Phase} there exists two different soft-core Gibbs measures $\mu^+$ and $\mu^-$ which come from the all $+$ and $-$-boundary condition respectively. 
The idea of the proof of non-Gibbsianness is that the model conditioned on a configuration $\eta\in\Omega$ with $\eta_i\neq 0$ for all $i\in \Z^d$  looks like an Ising model with a magnetic field given by the conditioning and $p_t$. In the proof of non-Gibbsianness for the time-evolved Ising measure the checkerboard configuration $\eta^{cb}\in\Omega$ is used, see \eqref{eq: check board} for its definition.
This configuration will also be a bad configuration for the time-evolved soft-core model, as we will see. The next lemma explains the connection between the soft-core model and the Ising model.
\begin{lemma}\label{lem: Sc-Is}
	Let $h=0$, $\lambda>0$ and $\beta>0$. Assume $\mu$ is a Gibbs measure for the soft-core Widom Rowlinson model. Then we have for any measurable function $f$ which depends only on the configuration at the origin for time $0$  that
	\begin{align}\label{eq: lim two layer}
	\lim_{\Lambda\uparrow \Z^d}	\mu_{t}^{sc,2}  (f\vert\eta_{\Lambda\backslash 0}) = \frac{\mu_{0^c}[\eta_{0^c}]\Big(\sum_{\omega_0\in E}f(\omega_0)\exp(-\beta\sum_{i\sim 0} \mathds{1}_{\omega_0\sigma_i=-1})\Big)}{\mu_{0^c}[\eta_{0^c}]\Big(\sum_{\omega_0\in E}	\exp(-\beta\sum_{i\sim 0} \mathds{1}_{\omega_0\sigma_i=-1})\Big)}
	\end{align}
	if $\eta\in \Omega^{s}:=\{-1,1\}^{\Z^d}\cap \{ \eta\in \Omega : \exists \Lambda \Subset \Z^d \text{ s.t. } \eta_i = s\; \forall i\in \Lambda^c\}$ for $s\in\{-1,1\}$. The $\sigma_i$'s are random variables distributed according to $\mu_{0^c}[\eta]$ which
	is the unique infinite-volume Gibbs measure of the Ising system on $\Z^d\backslash 0$ with $\eta$-dependent Hamiltonian $$\mathcal{H}^0_\Lambda[\eta](\omega):= -\beta\sum_{\{i,j\}\in\mathcal{E}^b_\Lambda\backslash\mathcal{E}^b_{\{0\}}}\mathds{1}_{\omega_i\omega_j=-1}-h_t\sum_{i\in\Lambda} \omega_i\eta_i $$
	for $\Lambda\Subset \Z^d\backslash 0$ and $h_t:= \frac{1}{2}\log\left(\frac{p_t(1,1)}{p_t(-1,1)}\right)$.
	
\end{lemma}
Note that we do not need the $\lambda$ dependence in the Hamiltonian because $\eta\in {\Omega}^{\pm1}$. The part with $\lambda$ does not depend on $\omega_\Lambda$ and will cancel out.

%   Inside of $\Lambda$ the sites of the first-layer have to take values in $\{-1,1\}$ because of the boundary condition at time $t$. Therefore the only difference between these two measures without the limit lays in $\Lambda^c$. Here it is possible that we have zeros in the first-layer for soft-core model but this influence will be pushed away by the limit.
\begin{proof}
	We only consider the case where $\eta \in \Omega^{1}$ since the other case follows by symmetry. The measure on the right hand side in \eqref{eq: lim two layer} is well defined since $\eta\in \Omega^1$ and therefore it differs only on a finite volume from the all-plus configuration $\eta^+$. Putting $\eta^+$  into the Hamiltonian it becomes an Ising-Hamiltonian with positive magnetic field. It is known by the Lee-Yang Theorem \cite[Chapter 3]{friedli_velenik_2017} that there exists a unique Gibbs-measure for this Hamiltonian. Therefore $\mu_{0^c}[\eta_{0^c}]$ is well-defined.
	
	Outside of $\Lambda\backslash 0$  the measure $\mu_{t}^{sc,2}  (\cdot\vert\eta_{\Lambda\backslash 0})$ gives also positive probability to the spin-value $0$ but this will pushed away by taking the $\Lambda$-limit.
	To see this we introduce a conditioning in the first-layer at $ \Lambda_-^c :=\partial_- \Lambda \cup \Lambda^c$ with $\partial_-\Lambda:= \{i\in \Lambda : \exists j\in \Lambda^c \text{ s.t. } i\sim j\}$ and define the interior $\interior{\Lambda}:= \Lambda\backslash \partial_- \Lambda$. On $\Lambda\backslash 0$  the conditioning acts only local. Hence  the measure can be written as 
	\begin{align*}
		\mu_{t}^{sc,2}(f\vert\eta_{\Lambda \backslash 0}) = \frac{\int\mu(d\omega_\Lambda)f(\sigma_0) \prod_{i\in \Lambda\backslash 0}p_t(\omega_i,\eta_i)  }{\int\mu(d\omega_\Lambda) \prod_{i\in \Lambda\backslash 0}p_t(\omega_i,\eta_i) }
	\end{align*}
	By the DLR-equation for the starting Widom-Rowlinson measure $\mu$ we can insert the specification kernel for the volume $\interior{\Lambda}$ which yields
	\begin{align}\label{eq: two layer form}
	\mu_{t}^{sc,2}(f\vert\eta_{\Lambda \backslash 0}) = \frac{\int\mu(d\omega_\Lambda)\prod_{i\in \partial_-\Lambda}p_t(\omega_i,\eta_i)\gamma^{sc}_{\interior{\Lambda}}(f(\cdot) \prod_{i\in \interior{\Lambda}\backslash 0}p_t(\cdot,\eta_i)\vert \omega_{\partial_-\Lambda} ) }{\int\mu(d\omega_\Lambda) \prod_{i\in \partial_-\Lambda}p_t(\omega_i,\eta_i)\gamma^{sc}_{\interior{\Lambda}}(\prod_{i\in \interior{\Lambda}\backslash 0}p_t(\cdot,\eta_i)\vert \omega_{\partial_-\Lambda} ) }
	\end{align}
	The next step is to rewrite the specification kernel to see an Ising part. It follows that
	\begin{align*}
		&\gamma^{sc}_{\interior{\Lambda}}\Big(f(\cdot) \prod_{i\in \interior{\Lambda}\backslash 0}p_t(\cdot,\eta_i)\vert \omega_{\partial_-\Lambda} \Big) \\
		&=  \frac{1}{Z^{\omega_{\partial_-\Lambda}}_{\interior{\Lambda}}} \sum_{\omega_{\interior{\Lambda}}\in \Omega_{\interior{\Lambda}} }f(\omega_0) \exp\left(-\beta\sum_{i\sim 0} \mathds{1}_{\omega_0\omega_i=-1}\right)\exp\left(-\beta\hspace{-3pt}\sum_{\{i,j\}\in\mathcal{E}^b_{\interior{\Lambda}}\backslash\mathcal{E}^b_{\{0\}}} \mathds{1}_{\omega_j\omega_i=-1}\right)\frac{\exp(-h_t\sum_{i\in\interior{\Lambda}\backslash 0}\omega_i\eta_i)}{(2\cosh(h_t))^{\vert\interior{\Lambda}\vert-1}}
	\end{align*}
	The cosh-term does not depend on any configuration and will later cancel out with the corresponding term in denominator of \eqref{eq: lim two layer}.
	Define for a finite volume $\Lambda$ the Ising specification $\gamma^\text{Is}_\Lambda[\eta]$ which corresponds to the Hamiltonian $\mathcal{H}^0_\Lambda[\eta]$ on the lattice $\Z^d\backslash 0$. Then we have 
	\begin{align*}
		\gamma_{\interior{\Lambda}}\Big(f(\cdot) \prod_{i\in \interior{\Lambda}\backslash 0}p_t(\cdot,\eta_i)\big\vert \omega_{\partial_-\Lambda}\Big) = \frac{Z^{\omega_{\partial_-\Lambda}}_{\interior{\Lambda}\backslash 0,\text{Is}}[\eta]}{Z^{\omega_{\partial_-\Lambda}}_{\interior{\Lambda}}}\gamma^\text{Is}_{\interior{\Lambda}\backslash 0}[\eta_{\interior{\Lambda}\backslash 0}]\Bigg(\sum_{\omega_0\in E}\frac{f(\omega_0)\exp(-\beta\sum_{i\sim 0} \mathds{1}_{\omega_0\tilde\sigma_i=-1})}{(2\cosh(h_t))^{\vert\interior{\Lambda}\vert-1}}\bigg\vert \omega_{\partial_-\Lambda}\Bigg)
	\end{align*}
	where the random variables $\tilde{\sigma}_i$ are distributed according to the conditional measure on the right hand side.
	By defining for every $\eta_{\Lambda\backslash0}$ the probability measure $\nu^{\eta_{\Lambda\backslash0}}_t$ via
	\begin{align*}
		\nu^{\eta_{\Lambda\backslash0}}_t(\varphi) =  \frac{\int\mu(d\omega_{\partial_- \Lambda})\prod_{i\in \partial_-\Lambda}p_t(\omega_i,\eta_i)  \frac{Z^{\omega_{\partial_-\Lambda}}_{\interior{\Lambda}\backslash 0,\text{Is}}[\eta]}{Z^{\omega_{\partial_-\Lambda}}_{\interior{\Lambda}}} \varphi(\omega_{\partial_\Lambda})}{\int\mu(d\omega_{\partial_- \Lambda})\prod_{i\in \partial_-\Lambda}p_t(\omega_i,\eta_i)  \frac{Z^{\omega_{\partial_-\Lambda}}_{\interior{\Lambda}\backslash 0,\text{Is}}[\eta]}{Z^{\omega_{\partial_-\Lambda}}_{\interior{\Lambda}}}  }
	\end{align*} 
	where $\varphi$ is a $\mathcal{F}_{\partial_-\Lambda}$-measurable function, we get
	\begin{align*}
	\mu_{t}^{sc,2}  (f\vert\eta_{\Lambda\backslash 0}) = \frac{\int\nu_t^{\eta_{\Lambda\backslash 0}}(d\omega_{\partial_-\Lambda})\gamma^\text{Is}_{\interior{\Lambda}\backslash 0}[\eta_{\interior{\Lambda}\backslash 0}]\Big(\sum_{\omega_0\in E}f(\omega_0)\exp(-\beta\sum_{i\sim 0} \mathds{1}_{\omega_0\tilde\sigma_i=-1})\big\vert \omega_{\partial_-\Lambda}\Big)}{\int\nu_t^{\eta_{\Lambda\backslash 0}}(d\omega_{\partial_-\Lambda})\gamma^\text{Is}_{\interior{\Lambda}\backslash 0}[\eta_{\interior{\Lambda}\backslash 0}]\Big(\sum_{\omega_0\in E}\exp(-\beta\sum_{i\sim 0} \mathds{1}_{\omega_0\tilde\sigma_i=-1})\big\vert \omega_{\partial_-\Lambda}\Big) }.
	\end{align*}
	Note that by the uniqueness of Gibbs measures for the specification $(\gamma_\Lambda[\eta])_{\Lambda\Subset\Z\backslash 0}$ we have 
	\begin{align*}
	\lim_{\Lambda\uparrow \Z^d} \gamma^\text{Is}_{ \interior{\Lambda}\backslash 0}[\eta_{\interior{\Lambda}\backslash 0}](h\vert \omega_{\partial_- \Lambda}) = \mu_{0^c}[\eta](h)
	\end{align*}
	for all local functions $h$ and $\eta \in \Omega^+$. Furthermore, by uniqueness this convergence is uniform in $\omega$ \cite[Proposition 7.11]{georgii-book}. Thus we have
	\begin{align*}
	&\lim_{\Lambda\uparrow \Z^d}\left\vert \int \nu^{\eta_{\Lambda\backslash0}}_t  (d\omega_{\partial_-\Lambda} )(\gamma^\text{Is}_{ \interior{\Lambda}\backslash 0}[\eta_{\interior{\Lambda}\backslash 0}](h\vert \omega_{\partial_- \Lambda})-\mu_{0^c}[\eta_{0^c}](h)) \right\vert\\
	&\leq \lim_{\Lambda\uparrow \Z^d} \sup_{\omega}\left\vert \gamma^\text{Is}_{ \interior{\Lambda}\backslash 0}[\eta_{\interior{\Lambda}\backslash 0}](h\vert \omega_{\partial_- \Lambda})-\mu_{0^c}[\eta_{0^c}](f) \right\vert=0
	\end{align*}
	for all local bounded $h$ and $\eta\in\Omega^+ $. Hence it follows that 
	\begin{align*}
		\lim_{\Lambda\uparrow \Z^d}	\mu_{t}^{sc,2}  (f\vert\eta_{\Lambda\backslash 0}) = \frac{\mu_{0^c}[\eta_{0^c}]\Big(\sum_{\omega_0\in E}f(\omega_0)\exp(-\beta\sum_{i\sim 0} \mathds{1}_{\omega_0\sigma_i=-1})\Big)}{\mu_{0^c}[\eta_{0^c}]\Big(\sum_{\omega_0\in E}	\exp(-\beta\sum_{i\sim 0} \mathds{1}_{\omega_0\sigma_i=-1})\Big)}
	\end{align*}
%	We can repeat exactly the same argument for the two-layer Ising model, which yields
%	\begin{align*}
%	&	\lim_{\Lambda\uparrow \Z^d}\vert \mu_{t}^{is,2}  (f\vert \eta_\Lambda 	) - \mu[\eta](f)\vert \\
%	&=\lim_{\Lambda\uparrow \Z^d}\left\vert \int \mu_{t}^{is,2}  (d\omega_{\Lambda_-^c} \vert \eta_\Lambda	)(\gamma_{ \interior{\Lambda}}[\eta](f\vert \omega_{\partial_- \Lambda})-\mu[\eta](f)) \right\vert\\
%	&\leq \lim_{\Lambda\uparrow \Z^d} \sup_{\omega}\left\vert \gamma_{ \interior{\Lambda}}[\eta](f\vert \omega_{\partial_- \Lambda})-\mu[\eta](f) \right\vert=0.
	%\end{align*}
%	Summarizing, while the distortion of the boundary spins provided by the measures appearing as integrals in the second lines of the two displays is different and not explicit, the uniform convergence of the Ising system inside the interior of \Lambda provides the necessary control. Hence the Lemma is proved.
\end{proof}
We can repeat this argument to get the convergence for the two-layer Ising Model $\mu_{t}^{\text{Is},2}$ where the starting measure $\mu^\text{Is}$ is a Gibbs measure for the symmetric Ising model, i.e. for $\eta\in\Omega^{+}$ and ${\mathcal{P}(\{-1,1\})}$-measurable function $f$ which depends only on the configuration at time $0$ we have
\begin{align*}
\lim_{\Lambda\uparrow \Z^d}	\mu_{t}^{\text{Is},2}  (f\vert\eta_{\Lambda\backslash 0}) = \frac{\mu_{0^c}[\eta_{0^c}]\Big(\sum_{\omega_0\in \{-1,1\}}f(\omega_0)\exp(-\beta\sum_{i\sim 0} \mathds{1}_{\omega_0\sigma_i=-1})\Big)}{\mu_{0^c}[\eta_{0^c}]\Big(\sum_{\omega_0\in \{-1,1\}}	\exp(-\beta\sum_{i\sim 0} \mathds{1}_{\omega_0\sigma_i=-1})\Big)}
\end{align*}
\begin{lemma}\label{lem: check bad}
	With the same assumption as in Lemma \ref{lem: Sc-Is} for large enough $\beta$ we have the existence of a time $t_1(\beta)$ such that for every $\epsilon>0$ and $t>t_1(\beta)$ there exists a set $\Gamma\Subset \Z^d$ with the property that for every $\Delta \Subset \Z^d$ with $\Gamma \subset \Delta$ the following is true
	\begin{align}\label{eq: sc disc.}
	\vert\lim_{\Lambda \uparrow \Z^d}\mu_t(\mathds{1}_{\eta_0=1}\vert \eta^{cb}_{\Delta\backslash 0}\eta^+_{\Lambda\backslash \Delta})-\lim_{\Lambda \uparrow \Z^d}\mu_t(\mathds{1}_{\eta_0=1}\vert \eta^{cb}_{\Delta\backslash 0}\eta^-_{\Lambda\backslash \Delta})\vert > \epsilon.
	\end{align}
\end{lemma}
\begin{proof}First we can write
	\begin{align*} 
	&\mu_t(\eta_0=\pm 1\vert \eta^{cb}_{\Delta\backslash 0}\eta^+_{\Lambda\backslash \Delta} ) = \int{\mu_{t}^{sc,2}}(d\omega_0 \vert\eta^{cb}_{\Delta\backslash 0}\eta^+_{\Lambda\backslash \Delta}  ) p_t(\omega_0,\pm1).
	\end{align*}
	and use Lemma \ref{lem: Sc-Is} for the functions $f^{\pm}: \{-1,0,1\}\rightarrow \R$ given by $\omega_0\mapsto p_t(\omega_0,\pm1)$. This implies that 
	\begin{align*}
	\lim_{\Lambda\uparrow \Z^d}	\frac{\mu_t(\eta_0=1\vert \eta^{cb}_{\Delta\backslash 0}\eta^+_{\Lambda\backslash \Delta} ) }{\mu_t(\eta_0=-1\vert \eta^{cb}_{\Delta\backslash 0}\eta^+_{\Lambda\backslash \Delta} ) } = \frac{\mu_{0^c}[\eta_{0^c}]\Big(\sum_{\omega_0\in E}p_t(\omega_0,1)\exp(-\beta\sum_{i\sim 0} \mathds{1}_{\omega_0\sigma_i=-1})\Big)}{\mu_{0^c}[\eta_{0^c}]\Big(\sum_{\omega_0\in E}p_t(\omega_0,-1)\exp(-\beta\sum_{i\sim 0} \mathds{1}_{\omega_0\sigma_i=-1})\Big)}
	\end{align*}	
	Since  $p_t(0,\pm1)=0$ we get 
	\begin{align}\label{eq: lim equal 2}
		\lim_{\Lambda\uparrow \Z^d}	\frac{\mu_t(\eta_0=1\vert \eta^{cb}_{\Delta\backslash 0}\eta^+_{\Lambda\backslash \Delta} ) }{\mu_t(\eta_0=-1\vert \eta^{cb}_{\Delta\backslash 0}\eta^+_{\Lambda\backslash \Delta} ) } = \lim_{\Lambda\uparrow \Z^d}	\frac{\mu^\text{Is}_t(\eta_0=1\vert \eta^{cb}_{\Delta\backslash 0}\eta^+_{\Lambda\backslash \Delta} ) }{\mu^\text{Is}_t(\eta_0=-1\vert \eta^{cb}_{\Delta\backslash 0}\eta^+_{\Lambda\backslash \Delta} ) }
	\end{align}
	where $\mu^\text{Is}_t$ is the time-evolved measure with any Gibbs measure of the symmetric Ising Gibbs model as a starting measure.  By \cite{enter-fernandez-hollander-redig02} it is known that there exists a $t_1(\beta)$, for $\beta$ which are much larger as the critical value of the inverse temperature $\beta_c^\text{Is}$ for the Ising model, such that for all $t>t_1(\beta)$ the configuration $\eta^{cb}$ is a bad for $\mu^{\text{Is}}_t$. Hence $\lim_{\Lambda\uparrow \Z^d}\mu^\text{Is}_t(\eta_0=1\vert \cdot_{\Delta\backslash 0}\eta^+_{\Lambda\backslash \Delta} )$ is discontinuous at $\eta^{cb}$ which implies that the right hand side of \eqref{eq: lim equal 2} is also discontinuous at $\eta^{cb}$. This implies that $\lim_{\Lambda\uparrow \Z^d}	\frac{\mu_t(\eta_0=1\vert \eta^{cb}_{\Delta\backslash 0}\eta^+_{\Lambda\backslash \Delta} ) }{\mu_t(\eta_0=-1\vert \eta^{cb}_{\Delta\backslash 0}\eta^+_{\Lambda\backslash \Delta} ) }$ is also discontinuous at $\eta^{cb}$. Hence, $\lim_{\Lambda\uparrow \Z^d}\mu_t(\eta_0=1\vert \cdot^{cb}_{\Delta\backslash 0}\eta^+_{\Lambda\backslash \Delta} )$ or  $\lim_{\Lambda\uparrow \Z^d}\mu_t(\eta_0=-1\vert \cdot_{\Delta\backslash 0}\eta^+_{\Lambda\backslash \Delta} )$ are discontinuous at $\eta^{cb}$ but by symmetry both of them are discontinuous. This implies \eqref{eq: sc disc.}.
	% for  is not exactly the same measure as in the proof of Lemma \ref{lem: Sc-Is} since there is no condition at time $t$ for the site $0$. But with the same arguments as in the above Lemma it will converge to a measure $\tilde{\mu}[\eta^{cb}_{\Delta\backslash 0}\eta^+_{\Delta^c}]$ with single-site specifications which are given by the Hamiltonian defined in Lemma \ref{lem: Sc-Is} for every site $i \in \Z^d\backslash \{0\}$, whereas for $i=0$ it is the soft-core Widom-Rowlinson interaction. By these arguments the limit $\lim_{\Lambda\uparrow \Z^d}\mu^+_t(\eta_0=+\vert \eta^{cb}_{\Delta\backslash 0}\eta^+_{\Lambda\backslash \Delta} )$ can be interpreted as a time evolved Ising measure with only a finite perturbation at the origin. By \cite{enter-fernandez-hollander-redig02} it is know that for large $\beta$ and $\lambda$ the checkerboard configuration is a bad configuration for the time evolved measure and finite perturbations do not change the discontinuous behavior of Gibbs measures. This implies
\end{proof}
\begin{proof}[Proof of Theorem \ref{thm: sc non gibbs}]
Choose $\beta\gg\beta_c^\text{Is}$.	By Lemma \ref{lem: check bad} there exists a time $t_2(\beta)$ such that for all $t>t_2(\beta)$ the checkerboard configuration is bad for the time-evolved measure $\mu_t$. Hence the time-evolved measure $\mu_t$ is not Gibbs for all $t>t_2(\beta)$.
\end{proof}
In \cite{kissel-kuelske18} the time-evolved mean-field version of the symmetric soft-core Widom-Rowlinson model $\mu_{t,N}^{\mathrm{mf}}$ was investigated. For mean-field models the correct notion for the Gibbs-property is called \emph{sequentially Gibbs}. A sequence exchangeable measures $(\mu_{t,N})_{N\in \N}$ satisfies the sequential Gibbs property if for every sequence of configurations $(\omega_{[2,N]})_{N\in \N}$ with $\omega_{[2,N]}\in E^{N-1}$ and $L_{N-1}(\omega_{[2,N-1]})\rightarrow \al\in \mathcal{M}_1(E)$, where $L_{N}$ is the empirical measure,  the limit $\lim_{N\rightarrow \infty}\mu^{\mathrm{mf}}_{t,N}(\omega_1\vert L_{N-1}(\omega_{[2,N]})) = \gamma_t^{\mathrm{mf}}(\omega_1\vert \al)$ exists and does not depend on the choice of sequence. A measure $\al$ is called bad empirical measure if the above property is not satisfied. It was proven in \cite{kissel-kuelske18} that the time-evolved mean-field model is not Gibbs for large $t$ and the first occurrence of this non-Gibbsian behavior happens for measures $\al\in\mathcal{M}_1(E)$ with $\al(0)=0$. This corresponds to configurations on the lattice, which contain only pluses and minuses. We conjecture that such fully occupied configurations are also the first bad configuration on the lattice. 
\begin{conjecture}
	Let $\beta$ be large enough and $\mu\in \mathcal{G}(\gamma^{sc}_{\beta,\lambda,0})$. Then there exists a time $t_{NG}(\beta)$ such that for all $t<t_{NG}(\beta)$ the time-evolved measure is Gibbs,  and non-Gibbs for all $t> t_{NG}(\beta)$ where $t_{NG}(\beta)$ is the exit-time from the Gibbsian region for the Ising-model 
	with Hamiltonian $-\beta\sum_{i\sim j} 1_{\omega_i \omega_j =-1}$
\end{conjecture}
\subsection{Time-evolved hard-core model}
{For the hard-core model we cannot use the method we established for the soft-core model. To see this, consider the  first-layer model single-site kernels with $\eta_{i_0}= 1$}
\begin{align*}
\gamma^i_{i_0,t}[\eta](\cdot \vert \bar{\omega}) = \frac{\prod_{j\sim i_0}\mathds{1}(\omega_{i_0}\bar{\omega}_{j}\neq -1)p_t(\omega_{i_0},1)\al(\omega_{i_0})}{\prod_{j\sim i_0}\mathds{1}(\bar{\omega}_{j}\neq -1)p_t(1,1)\al(1)+\prod_{j\sim i}\mathds{1}(\bar{\omega}_{j}\neq 1)p_t(-1,1)\al(\omega_{i_0})}.
\end{align*}
{Note that numerator and denominator  can both be simultaneously zero. This happens if there exist $k,m$ with ${i_0}\sim k,{i_0}\sim m$, $\bar{\omega}_k = 1$ and $\bar{\omega}_m= -1$. In this case we define the kernel to be zero. For two boundary conditions $\bar{\omega},\tilde{\omega}\in \Omega$ with $\bar{\omega}_j=1$ and $\tilde{\omega_j}= -1$ for some $j\sim i $  and $\bar\omega_{k}= \tilde\omega_k = 0$ for all $k\neq j$  it follows that $\gamma^i_{i_0,t}[\eta](\cdot \vert \bar{\omega})= \delta_{1}(\cdot)$, $\gamma^i_{i_0,t}[\eta](\cdot \vert \tilde{\omega})= \delta_{-1}(\cdot)$ and}
\begin{align*}
d_{TV}(\gamma^i_{i_0,t}[\eta](\cdot \vert \bar{\omega}),\gamma^i_{i_0,t}[\eta](\cdot \vert \tilde{\omega})) = 1.
\end{align*}
{This implies that the restricted constrained first-layer model for the hard-core case cannot satisfy the Dobrushin condition.}

For the proofs we follow the idea of \cite{jahnel-kuelske17b} where the continuous hard-core Widom-Rowlinson model was investigated. To use their method the discrete hard-core model has to be reformulated. 
For splitting the information of location and spin value of a particle we define a new configuration space $\tilde{\Omega} = \{(0,0),(1,-1),(1,1)\}^{\Z^d}$ where we identify $0\triangleq(0,0), 1 \triangleq(1,1)$ and $-1 \triangleq (1,-1)$. For an element $\boldsymbol{\omega}\in \tilde{\Omega}$ we write $\omega^{\sigma}:=(\omega,\sigma)=\boldsymbol{\omega}$. The first entry describes if there is a particle at some site $i$ and the second entry describes its spin value. With this identification we  rewrite first the specification of the hard-core Widom-Rowlinson model and then give a new formula for $\gamma^{\bar{\omega}}_{\Lambda,\Delta,\alpha,t}$. 
\begin{lemma}\label{lem: ident.}
	Let $\al\in\mathcal{M}_1(E)$ and $\Lambda \Subset \Z^d$. Then for $\omega_{\Lambda}\omega_{\Lambda^c}\in\Omega$ it follows with the above identification $\omega\triangleq\bar{\boldsymbol{\omega}}=(\bar{\omega},\bar{\sigma})$  that
	\begin{align*}
	\gamma^{hc}_{\Lambda,\beta,\al}(\omega_{\Lambda}\vert \omega_{\Lambda^c} ) &= \frac{\al(0)^{\vert\Lambda^0_{\bar{\omega}}\vert} I^{hc}_{\Lambda}(\bar{\sigma}_{\Lambda_{\bar{\omega}}^1}0_{\Lambda_{\bar{\omega}}^0}\bar{\sigma}_{\Lambda^c}) \al(1)^{\bar{\sigma}_{\Lambda}^+}\al(-1)^{\bar{\sigma}_{\Lambda}^-}}{\sum_{\tilde{\omega}_\Lambda\in\{0,1\}^\Lambda}\al(0)^{\vert\Lambda^0_{\tilde{\omega}}\vert}\sum_{\sigma_{\Lambda^1_{\tilde{\omega}}}\in\{-1,1\}^{\Lambda^1_{\tilde{\omega}}}} I^{hc}_{\Lambda}(\sigma_{\Lambda^1_{\tilde{\omega}}}0_{\Lambda_{\bar{\omega}}^0}\bar{\sigma}_{\Lambda^c})\al(1)^{\sigma_{\Lambda}^+}\al(-1)^{\sigma_{\Lambda}^-}}\\
	&=: \gamma^{hc}_{\Lambda,\beta,\al}(\bar{\boldsymbol{\omega}}_{\Lambda}\vert \bar{\boldsymbol{\omega}}_{\Lambda^c} )
	\end{align*}
	where $\Lambda^1_{\bar{\omega}}= \{i\in \Lambda \,:\; \bar{\omega}_i = 1\}$, $\Lambda^0_{\bar{\omega}}= \Lambda \backslash \Lambda^1_{\bar{\omega}}$ and $\sigma^{\pm}_{\Lambda}= \vert \{i \in\Lambda \,:\, \sigma_i=\pm1\}\vert$. In addition we define the sum $\sum_{\sigma_{\Lambda^1_{\tilde{\omega}}}\in\{-1,1\}^{\Lambda^1_{\tilde{\omega}}}} I^{hc}_{\Lambda}(\sigma_{\Lambda^1_{\tilde{\omega}}}0_{\Lambda_{\bar{\omega}}^0}\bar{\sigma}_{\Lambda^c}) \al(1)^{\sigma_{\Lambda}^+}\al(-1)^{\sigma_{\Lambda}^-}$ to be equal 1 if $\Lambda^1_{\bar{\omega}}= \{i\in \Lambda \,:\; \bar{\omega}_i = 1\}=\emptyset$.  
\end{lemma}
\begin{proof}
	%	We show this equality only where $\Lambda$ is a singleton $\{i\}$ and since 
	%	\begin{align*}
	%	\rho_{i}(\bar{\boldsymbol{\omega}}) = \frac{ e^{-\mathcal{H}_{i}(\bar{\sigma}_{i_{\bar{\omega}}^1}0_{i\backslash i_{\bar{\omega}}^1}\bar{\sigma}_{i^c})} }{\sum_{\tilde{\omega}_i\in\{0,1\}^i}\al(0)^{\vert i^0_{\tilde{\omega}}\vert}\sum_{\sigma_{i^1_{\tilde{\omega}}}\in\{-1,1\}^{i^1_{\tilde{\omega}}}} e^{-\mathcal{H}_{i}(\sigma_{i^1_{\tilde{\omega}}}0_{i\backslash i_{\tilde{\omega}}^1}\bar{\sigma}_{i^c})} \al(1)^{\sigma_{i}^+}\al(-1)^{\sigma_{i}^-}}
	%	\end{align*}
	%	 is positive for all $\bar{\boldsymbol{\omega}}$ the rest of the proof follows by (Georgii 1.33). 
	We start with the left hand side
	\begin{align*}
	\gamma^{hc}_{\Lambda,\beta,\al}(\omega_{\Lambda}\vert \omega_{\Lambda^c} ) = \frac{I^{hc}_\Lambda(\omega_\Lambda\omega_\Lambda^c)\prod_{i\in \Lambda}\al(\omega_i)}{\sum_{\tilde{\omega}_\Lambda\in\{-1,0,1\}^\Lambda}I^{hc}_\Lambda(\tilde{\omega}_\Lambda{\omega}_{\Lambda^c})\prod_{i\in \Lambda}\al(\tilde{\omega}_i)}.
	\end{align*}
	Using the identification above we have $\sigma_i = \omega_i$ and thus 
	\begin{align*}
	\frac{I^{hc}_\Lambda(\omega_\Lambda\omega_{\Lambda^c})\prod_{i\in \Lambda}\al(\omega_i)}{\sum_{\tilde{\omega}_\Lambda\in\{-1,0,1\}^\Lambda}I^{hc}_\Lambda(\tilde{\omega}_\Lambda{\omega}_{\Lambda^c})\prod_{i\in \Lambda}\al(\tilde{\omega}_i)} = \frac{I^{hc}_\Lambda(\bar{\sigma}_\Lambda\bar{\sigma}_{\Lambda^c})\prod_{i\in \Lambda}\al(\bar{\sigma}_i)}{\sum_{\tilde{\sigma}_\Lambda\in\{-1,0,1\}^\Lambda}I^{hc}_\Lambda(\tilde{\sigma}_\Lambda{\sigma}_{\Lambda^c})\prod_{i\in \Lambda}\al(\tilde{\sigma}_i)}.
	\end{align*}
	Now we have to bring a second sum into the play. We do this by adding sums over indicator function which are equal to one, which yields
	\begin{align*}
	&\frac{I^{hc}_\Lambda(\bar{\sigma}_\Lambda\bar{\sigma}_{\Lambda^c})\prod_{i\in \Lambda}\al(\bar{\sigma}_i)}{\sum_{\tilde{\sigma}_\Lambda\in\{-1,0,1\}^\Lambda}I^{hc}_\Lambda(\tilde{\sigma}_\Lambda{\sigma}_{\Lambda^c})\prod_{i\in \Lambda}\al(\tilde{\sigma}_i)}\\
	&=\frac{\al(0)^{\vert\Lambda^0_{\bar{\omega}}\vert} I^{hc}_{\Lambda}(\bar{\sigma}_{\Lambda_{\bar{\omega}}^1}0_{ \Lambda_{\bar{\omega}}^0}\bar{\sigma}_{\Lambda^c}) \al(1)^{\bar{\sigma}_{\Lambda}^+}\al(-1)^{\bar{\sigma}_{\Lambda}^-}}{\sum_{\tilde{\omega}_{\Lambda}\in\{0,1\}^\Lambda}\sum_{\tilde{\sigma}_\Lambda\in\{-1,0,1\}^\Lambda}I^{hc}_\Lambda(\tilde{\sigma}_\Lambda{\sigma}_{\Lambda^c})\prod_{i\in \Lambda}\al(\tilde{\sigma}_i)\prod_{i\in\Lambda}(\mathds{1}(\tilde{\omega}_i= \tilde{\sigma}_i =0)+\mathds{1}(\tilde{\omega}_i\neq 0, \tilde{\sigma}_i\neq 0))}.
	\end{align*}
	Decompose the products over $\Lambda$ into one over $\Lambda^1_{\tilde{\omega}}$ and one over  $\Lambda^0_{\tilde{\omega}}$ yields
	\begin{align*}
	&\gamma^{hc}_{\Lambda,\al}(\omega_{\Lambda}\vert \omega_{\Lambda^c} )\\
	&=\frac{\al(0)^{\vert\Lambda^0_{\bar{\omega}}\vert} I^{hc}_{\Lambda}(\bar{\sigma}_{\Lambda_{\bar{\omega}}^1}0_{ \Lambda_{\bar{\omega}}^0}\bar{\sigma}_{\Lambda^c}) \al(1)^{\bar{\sigma}_{\Lambda}^+}\al(-1)^{\bar{\sigma}_{\Lambda}^-}}{\sum_{\tilde{\omega}_{\Lambda}\in\{0,1\}^\Lambda}\al(0)^{\vert \Lambda_{\tilde{\omega}}^0\vert}\sum_{\tilde{\sigma}_\Lambda\in\{-1,0,1\}^\Lambda}I^{hc}_\Lambda(\tilde{\sigma}_\Lambda{\sigma}_{\Lambda^c})\al(1)^{\tilde{\sigma}_{\Lambda}^+}\al(-1)^{\tilde{\sigma}_{\Lambda}^-}\prod_{i\in \Lambda_{\tilde{\omega}}^1}\mathds{1}_{\mathcal{A}}(\tilde{\omega}_i,\tilde{\sigma}_i)\prod_{i\in \Lambda_{\tilde{\omega}}^0}\mathds{1}_{\mathcal{B}}(\tilde{\omega}_i,\tilde{\sigma}_i)}
	\end{align*}
	where  $\mathcal{A}= \{(\tilde{\omega},\tilde{\sigma})\in \{0,1\}\times \{-1,0,1\}\;:\,  \tilde{\omega}\neq 0, \tilde{\sigma}\neq 0  \}$ and $\mathcal{B}= \{(\tilde{\omega},\tilde{\sigma})\in \{0,1\}\times \{-1,0,1\}\;:\,  \tilde{\omega}=\tilde{\sigma}= 0  \}$. If we restrict the second sum with respect to the products over the indicators we get the desired formula. 
\end{proof}
Consequently for ${\eta}\in {\Omega}$, ${\bar{\omega}}\in {\Omega}$  and $\Delta\subset \Lambda \Subset \Z^d$ we have 
\begin{align*}
&\gamma^{\bar{\omega}}_{\Lambda,\Delta,\alpha,t}(\eta_\Delta\vert \eta_{\Lambda\backslash \Delta})\\
&=\frac{\alpha(0)^{\vert \Delta_{\eta}^0\vert}\sum_{\sigma_{\Lambda_{\eta}^1}\in\{-1,1\}^{\Lambda^1_\eta}}I^{hc}_{\Lambda}(\sigma_{\Lambda^1_\eta}0_{\Lambda^0_\eta}\bar{\sigma}_{\Lambda^c})  \alpha(1)^{\sigma_\Lambda^1}\alpha(-1)^{\sigma_\Lambda^{-1}}\prod_{i\in \Delta^1_{\eta}} p_t(\sigma_i,\hat{\sigma}_i)\prod_{i\in\Lambda\backslash \Delta^1_\eta} p_t(\sigma_i,\hat{\sigma}_i)}{\sum_{\omega_\Delta\in \{0,1\}^{\Delta}}\alpha(0)^{\vert\Delta_\omega^0\vert}\sum_{\sigma_{\Lambda_{\omega_\Delta\eta_{\Lambda\backslash \Delta}}^1}\in\{-1,1\}^{\Lambda^1}}I^{hc}_{\Lambda}(\sigma_{\Lambda^1}0_{{\Lambda^0}}\bar{\sigma}_{\Lambda^c})  \alpha(1)^{\sigma_\Lambda^1}\alpha(-1)^{\sigma_\Lambda^{-1}}\prod_{i\in\Lambda\backslash \Delta^1_\eta} p_t(\sigma_i,\hat{\sigma}_i)}\\
&= \gamma^{\bar{\boldsymbol{\omega}}}_{t,\Lambda,\Delta,\alpha,\beta}(\boldsymbol{\eta}_\Delta\vert \boldsymbol{\eta}_{\Lambda\backslash \Delta})
\end{align*}
where $\boldsymbol{\eta}=(\eta,\hat{\sigma})$ and $\boldsymbol{\bar{\omega}}=(\bar{\omega},\bar{\sigma})$ are the identified configuration of $\eta$ and $\bar{\omega}$. Note that for a starting configuration $\boldsymbol{\omega}= (\omega,\sigma)$ and a corresponding evolved configuration $\boldsymbol{\eta}= (\eta,\hat{\sigma})$ the first entry of them are equal, i.e. $\omega = \eta$,  due to the preservation of the particle number under the time evolution.\\
With this reformulation we can again rewrite $\gamma^{\bar{\omega}}_{t,\Lambda,\Delta,\alpha,\beta}$ by using clusters. For this we define $q_t:= \frac{p_t(1,1)}{p_t(1,-1)}$ and $\a_r := \frac{\al(1)}{\al(-1)}$.

\begin{lemma}\label{lem: cluster rep}
	Let  $\al\in\mathcal{M}_1(E)$ and $\Delta\subset\Lambda \Subset \Z^d$. Then for all $\mathcal{F}_{\Delta}$-measurable bounded functions $f$, all $\boldsymbol{\eta}\in \tilde{\Omega}$ and all $\bar{\boldsymbol{\omega}}\in \tilde{\Omega}$ we have that
	\begin{align*}
&\gamma^{\bar{\boldsymbol{\omega}}}_{t,\Lambda,\Delta,\alpha}(f\vert {\boldsymbol{\eta}}_{\Lambda\backslash \Delta})=\\
&	\frac{\sum_{\omega_\Delta\in\{0,1\}^\Delta} \frac{\al(0)^{\Delta_\omega^0}}{\al(-1)^{-\Delta_\omega^1}}f_{\eta_{\Lambda\backslash \Delta}}^{\bar{\omega}_{\Lambda^c}}(\omega_\Delta) \prod_{C\in \mathcal{C}_{\Delta}(\omega_{\Delta}\eta_{\Lambda\backslash \Delta })} (
	\al_r^{\vert C\cap \Delta\vert}\mathds{1}_{\sigma_{C\cap\Lambda^c}=1} +\hspace{-0.6pt}e^{- \sum_{i\in C\cap \Lambda\backslash \Delta}(\log(\al_r)+\log(q_t)\hat{\sigma}_i)}\mathds{1}_{\sigma_{C\cap\Lambda^c}=-1})}{\sum_{\omega_\Delta\in\{0,1\}^\Delta} \frac{\al(0)^{\Delta_\omega^0}}{\al(-1)^{-\Delta_\omega^1}}\prod_{C\in \mathcal{C}_{\Delta}(\omega_{\Delta}\eta_{\Lambda\backslash \Delta })} (
	\al_r^{\vert C\cap \Delta\vert}\mathds{1}_{\sigma_{C\cap\Lambda^c}=1} +e^{- \sum_{i\in C\cap \Lambda\backslash \Delta}(\log(\al_r)+\log(q_t)\hat{\sigma}_i)}\mathds{1}_{\sigma_{C\cap\Lambda^c}=-1})}
\end{align*}
with 
\begin{align*}
f_{\eta_{\Lambda\backslash \Delta}}^{\bar{\omega}_{\Lambda^c}}(\omega_\Delta)=
 \frac{\sum_{\hat{\sigma}_{\Delta_{\omega}^1}\in\{-1,1\}^{\Delta^1_{\omega}}}f(\omega_\Delta,\hat{\sigma}_{\Delta^1_{\omega}}0_{\Delta^0_{\omega}}) {A}(\omega_\Delta\eta_{\Lambda\backslash \Delta},\hat{\sigma}_{\Delta^1_\omega}\hat{\sigma}_{{\Lambda\backslash\Delta}_\eta^1})}{\prod_{C\in \mathcal{C}_{\Delta}(\omega_{\Delta}\eta_{\Lambda\backslash \Delta })} (
	\al_r^{\vert C\cap \Delta\vert}\mathds{1}_{\sigma_{C\cap\Lambda^c}=1} +e^{- \sum_{i\in C\cap \Lambda\backslash \Delta}(\log(\al_r)+\log(q_t)\hat{\sigma}_i)}\mathds{1}_{\sigma_{C\cap\Lambda^c}=-1})}
\end{align*}
and
\begin{align*}
&{A}(\omega_\Lambda,\hat{\sigma}_{\Lambda_{\omega}^1})\\
&=\prod_{C\in \mathcal{C}_{\Delta}(\omega_{\Delta}\eta_{\Lambda\backslash \Delta })} (
\al_r^{\vert C\cap \Delta\vert}\frac{p_t(1,1)^{\hat{\sigma}^+_{C\cap \Delta}}}{p_t(-1,1)^{-\hat{\sigma}^-_{C\cap \Delta}}}\mathds{1}_{\sigma_{C\cap\Lambda^c}=1} +\frac{p_t(-1,1)^{\hat{\sigma}^+_{C\cap \Delta}}}{p_t(1,1)^{-\hat{\sigma}^-_{C\cap \Delta}}}e^{- \sum_{i\in C\cap \Lambda\backslash \Delta}(\log(\al_r)+\log(q_t)\hat{\sigma}_i)}\mathds{1}_{\sigma_{C\cap\Lambda^c}=-1}).
\end{align*}
\end{lemma}
\begin{proof}
	First we define 
	\begin{align*}
	&U^{\hat{\sigma}_{\Lambda\backslash \Delta},\bar{\boldsymbol{\omega}}}(f,\omega_\Delta\eta_{\Lambda\backslash \Delta})\\
	&=\sum_{\hat{\sigma}_{\Delta_{\omega}^1}\in\{-1,1\}^{\Delta^1_{\omega}}}f(\omega_\Delta,\hat{\sigma}_{\Delta^1_{\omega}}0_{\Delta^0_{\omega}})\sum_{\sigma_{\Lambda_{\omega_\Delta\eta_{\Lambda\backslash \Delta}}^1}\in\{-1,1\}^{\Lambda^1}}I^{hc}_{\Lambda}(\sigma_{\Lambda^1}0_{{\Lambda^0}}\bar{\sigma}_{\Lambda^c})  \al(1)^{\sigma_\Lambda^1}\al(-1)^{\sigma_\Lambda^{-1}}\prod_{i\in\Lambda^1_{\omega_\Delta\eta_{\Lambda\backslash \Delta}}} p_t(\sigma_i,\hat{\sigma}_i)
	\end{align*}
	and with that $\gamma^{\bar{\omega}}_{t,\Lambda,\Delta,\alpha,\beta}$ becomes
	\begin{align*}
	\gamma^{\bar{\omega}}_{t,\Lambda,\Delta,\alpha,\beta}(f\vert {\boldsymbol{\eta}}_{\Lambda\backslash \Delta})=\frac{\sum_{\omega_\Delta\in\{0,1\}^\Delta}\al(0)^{\vert \Delta_{\omega}^0\vert} U^{\hat{\sigma}_{\Lambda\backslash \Delta},\bar{\boldsymbol{\omega}}}(f,\omega_\Delta\eta_{\Lambda\backslash \Delta})}{\sum_{\omega_\Delta\in\{0,1\}^\Delta} \al(0)^{\vert \Delta_{\omega}^0\vert} U^{\hat{\sigma}_{\Lambda\backslash \Delta},\bar{\boldsymbol{\omega}}}(1,\omega_\Delta\eta_{\Lambda\backslash \Delta})}.
	\end{align*}
	The $U$'s can now be rewritten with the help of a cluster representation because the hard-core constraint acts independently on disjoint clusters. In other words we have 
	$I^{hc}_{\Lambda}= \prod_{C\in \mathcal{C}(\omega_{\Delta}\eta_{\Lambda\backslash \Delta }\bar{\omega}_{\Lambda^c}) } I^{hc}_{C\cap \Lambda}$. Hence it follows that
	\begin{align*}
	&U(f,\omega_\Delta\eta_{\Lambda\backslash \Delta}) \\
	&=\sum_{\hat{\sigma}_{\Delta_{\omega}^1}\in\{-1,1\}^{\Delta^1_{\omega}}}f(\omega_\Delta,\hat{\sigma}_{\Delta^1_{\omega}}0_{\Delta^0_{\omega}}) \prod_{C\in \mathcal{C}(\omega_{\Delta}\eta_{\Lambda\backslash \Delta }\bar{\omega}_{\Lambda^c}) } \sum_{\sigma_{C\cap \Lambda}}  I^{hc}_{C\cap \Lambda}(\sigma_{C\cap\Lambda}\bar{\sigma}_{\Lambda^c}) \al(1)^{\sigma_{C\cap\Lambda}^1}\al(-1)^{\sigma_{C\cap\Lambda}^{-1}}\prod_{i\in\Lambda\cap C}p_t(\sigma_i,\hat{\sigma}_i)\\
	&=: U^{\hat{\sigma}_{\Lambda\backslash \Delta},\bar{\boldsymbol{\omega}}}(f,\mathcal{C}(\omega_\Delta\eta_{\Lambda \backslash \Delta}\bar{\omega}_{\Lambda^c}))
	\end{align*}
	Since $f$ is a $\mathcal{F}_\Delta$-measurable function the last expression $U^{\hat{\sigma}_{\Lambda\backslash \Delta},\bar{\boldsymbol{\omega}}}(f,\mathcal{C}(\omega_\Delta\eta_{\Lambda \backslash \Delta}\bar{\omega}_{\Lambda^c}))$ is equal to   $$U^{\hat{\sigma}_{\Lambda\backslash \Delta},\bar{\boldsymbol{\omega}}}(f,\mathcal{C}_\Delta(\omega_\Delta\eta_{\Lambda \backslash \Delta}\bar{\omega}_{\Lambda^c}))U^{\hat{\sigma}_{\Lambda\backslash \Delta},\bar{\boldsymbol{\omega}}}(1,\mathcal{C}_{\Delta^c}(\omega_\Delta\eta_{\Lambda \backslash \Delta}\bar{\omega}_{\Lambda^c})).$$ 
	By the discussion below Definition \ref{Defi: Cluster} the term $U^{\hat{\sigma}_{\Lambda\backslash \Delta},\bar{\boldsymbol{\omega}}}(1,\mathcal{C}_{\Delta^c}(\omega_\Delta\eta_{\Lambda \backslash \Delta}\bar{\omega}_{\Lambda^c}))$ does not depend on $\omega_\Delta$ and consequently will cancel out in 	$\gamma^{\bar{\omega}}_{t,\Lambda,\Delta,\alpha,\beta}(f\vert {\boldsymbol{\eta}}_{\Lambda\backslash \Delta})$.\\
	By defining $f_{\eta_{\Lambda\backslash \Delta}}^{\bar{\omega}_{\Lambda^c}}(\omega_\Delta):=  \frac{ U_{\Delta}^{\hat{\sigma}_{\Lambda\backslash \Delta},\bar{\boldsymbol{\omega}}}(f,\mathcal{C}(\omega_\Delta\eta_{\Lambda \backslash \Delta}\bar{\omega}_{\Lambda^c}))}{ U_{\Delta}^{\hat{\sigma}_{\Lambda\backslash \Delta},\bar{\boldsymbol{\omega}}}(1,\mathcal{C}(\omega_\Delta\eta_{\Lambda \backslash \Delta}\bar{\omega}_{\Lambda^c}))}$ we can rewrite $\gamma^{\bar{\omega}}_{\Lambda,\Delta,\alpha,t}(f\vert {\boldsymbol{\eta}}_{\Lambda\backslash \Delta})$ as
	\begin{align*}
	\gamma^{\bar{\omega}}_{\Lambda,\Delta,\alpha,t}(f\vert {\boldsymbol{\eta}}_{\Lambda\backslash \Delta})=\frac{\sum_{\omega_\Delta\in\{0,1\}^\Delta}\al(0)^{\vert \Delta_{\omega}^0\vert}f_{\eta_{\Lambda\backslash \Delta}\bar{\omega}_{\Lambda^c}}(\omega_\Delta) U_\Delta^{\hat{\sigma}_{\Lambda\backslash \Delta},\bar{\boldsymbol{\omega}}}(1,\mathcal{C}(\omega_\Delta\eta_{\Lambda \backslash \Delta}\bar{\omega}_{\Lambda^c}))}{\sum_{\omega_\Delta\in\{0,1\}^\Delta} \al(0)^{\vert \Delta_{\omega}^0\vert} U_\Delta^{\hat{\sigma}_{\Lambda\backslash \Delta},\bar{\boldsymbol{\omega}}}(1,\mathcal{C}(\omega_\Delta\eta_{\Lambda \backslash \Delta}\bar{\omega}_{\Lambda^c}))}.
	\end{align*}
	We will now focus on $U_\Delta^{\hat{\sigma}_{\Lambda\backslash \Delta},\bar{\boldsymbol{\omega}}}(1,\mathcal{C}(\omega_\Delta\eta_{\Lambda \backslash \Delta}\bar{\omega}_{\Lambda^c}))$ and since the spin values of particles inside a single cluster at time $0$ have all to be equal, one have
	\begin{align*}
	&U_\Delta^{\hat{\sigma}_{\Lambda\backslash \Delta},\bar{\boldsymbol{\omega}}}(1,\mathcal{C}(\omega_\Delta\eta_{\Lambda \backslash \Delta}\bar{\omega}_{\Lambda^c}))\\
	&=\prod_{C\in \mathcal{C}_{\Delta}(\omega_{\Delta}\eta_{\Lambda\backslash \Delta }\bar{\omega}_{\Lambda^c}) } \Big(\al(1)^{\vert C\cap\Lambda\vert}p_t(1,1)^{\hat\sigma^+_{C\cap \Lambda \backslash \Delta}} p_t(-1,1)^{\hat\sigma^-_{C\cap \Lambda \backslash \Delta}} \mathds{1}_{\sigma_{C\cap \Lambda^c = 1}}\\
	&+  \al(-1)^{\vert C\cap\Lambda\vert}p_t(1,1)^{ \hat\sigma^-_{C\cap \Lambda \backslash \Delta}} p_t(-1,1)^{ \hat\sigma^+_{C\cap \Lambda \backslash \Delta}} \mathds{1}_{\sigma_{C\cap \Lambda^c = 1}}\Big).
	\end{align*}
	
	Define the magnetization at time $t$ via $m_{C\cap\Lambda\backslash \Delta}:= \frac{1}{\vert C \cap \Lambda\backslash \Delta\vert }\sum_{i\in C \cap \Lambda \backslash \Delta} \hat{\sigma}_i$ and rewrite the exponents as 
	$\hat{\sigma}_{C\cap \Lambda\backslash \Delta}^{\pm 1} = \vert C \cap \Lambda \backslash \Delta\vert \frac{(1\pm m_{C\cap\Lambda\backslash \Delta})}{2}.$
	With the  magnetization we can obtain for $U_\Delta^{\hat{\sigma}_{\Lambda\backslash \Delta},\bar{\boldsymbol{\omega}}}(1,\mathcal{C}(\omega_\Delta\eta_{\Lambda \backslash \Delta}\bar{\omega}_{\Lambda^c}))$ the expression 
	\begin{align*}
	&U_\Delta^{\hat{\sigma}_{\Lambda\backslash \Delta},\bar{\boldsymbol{\omega}}}(1,\mathcal{C}(\omega_\Delta\eta_{\Lambda \backslash \Delta}\bar{\omega}_{\Lambda^c}))\\
	= &\prod_{C\in \mathcal{C}_\Delta(\omega_{\Delta}\eta_{\Lambda\backslash \Delta }\bar{\omega}_{\Lambda^c}) } \Big(\al(1)^{\vert C\cap \Lambda\vert}p_t(1,1)^{\vert C\cap \Lambda\backslash \Delta\vert}p_t(1,-1)^{\vert C\cap \Lambda\backslash \Delta\vert}\left(\frac{p_t(1,1)}{p_t(1,-1)}\right)^{\frac{m_{C\cap\Lambda\backslash \Delta}\vert C\cap \Lambda \backslash \Delta\vert}{2}}\mathds{1}_{\sigma_{C\cap \Lambda^c}=1}\\
	&+\al(-1)^{\vert C\cap \Lambda\vert}p_t(1,-1)^{\vert C\cap \Lambda\backslash \Delta\vert}p_t(1,1)^{\vert C\cap \Lambda\backslash \Delta\vert}\left(\frac{p_t(1,1)}{p_t(1,-1)}\right)^{\frac{-m_{C\cap\Lambda\backslash \Delta}\vert C\cap \Lambda \backslash \Delta\vert}{2}}\mathds{1}_{\sigma_{C\cap\Lambda^c}=-1}\Big). 
	\end{align*}
	where we used the symmetry of $p_t$.
	Now we can pull out $ (p_t(1,1)p_t(1,-1))^{\frac{1}{2}\vert C\cap \Lambda\backslash \Delta\vert}$ in each term and note that 
	\begin{align*}
	\prod_{C\in \mathcal{C}_\Delta(\omega_\Delta)}	(p_t(1,1)p_t(1,-1))^{\frac{1}{2}\vert C\cap \Lambda\backslash \Delta\vert} = (p_t(+,+)p_t(+,-))^{\frac{1}{2}\vert C_\Delta(\omega_\Delta)\cap \Lambda\backslash \Delta\vert}
	\end{align*} 
	which does not depend on $\omega_\Delta$ so we put this term in some constant $c$ which will cancel out later with the corresponding term in the denominator. For the next step we define quantities which only depend on the coloring at some positive time 
	$\rho_{\pm}^{C\backslash \Delta} = (\al(\pm 1)q_t^{\frac{\pm m_{C\cap \Lambda\backslash \Delta}}{2} })^{\vert \mathcal{C}\cap \Lambda \backslash \Delta\vert} $
	and rewrite 
	\begin{align*}
	&U_\Delta^{\hat{\sigma}_{\Lambda\backslash \Delta},\bar{\boldsymbol{\omega}}}(1,\mathcal{C}(\omega_\Delta\eta_{\Lambda \backslash \Delta}\bar{\omega}_{\Lambda^c}))\\
	&= c\prod_{C\in \mathcal{C}_\Delta(\omega_{\Delta}\eta_{\Lambda\backslash \Delta }\bar{\omega}_{\Lambda^c}) } 
	(\al(1)^{\vert C\cap \Delta\vert}\rho_+^{C\backslash \Delta}\mathds{1}_{\sigma_{C\cap\Lambda^c}=1}
	+\al(-1)^{\vert C\cap \Delta\vert}\rho_-^{C\backslash \Delta}\mathds{1}_{\sigma_{C\cap\Lambda^c}=-1})\\
	&= c\al(-1)^{\vert\Delta^1_{\omega}\vert}\prod_{C\in \mathcal{C}_\Delta(\omega_{\Delta}\eta_{\Lambda\backslash \Delta }\bar{\omega}_{\Lambda^c}) }\rho_+^{C\backslash \Delta} (\frac{\al(1)}{\al(-1)}^{\vert C\cap \Delta\vert}\mathds{1}_{\sigma_{C\cap\Lambda^c}=1}
	+\frac{\rho_-^{C\backslash \Delta}}{\rho_+^{C\backslash \Delta}}\mathds{1}_{\sigma_{C\cap\Lambda^c}=-1}).
	\end{align*}
	The product $\prod_{C\in \mathcal{C}_\Delta(\omega_\Delta)}\rho_+^{C\backslash \Delta }$ does not depend on $\omega_\Delta$ which sounds a bit strange, but if $\Delta$ completely contains some cluster $C$ then $\rho_+^{C\backslash \Delta }$ is equal $1$. Now if $\vert C\cap \Lambda \backslash \Delta\vert>0  $ and  $C \in \mathcal{C}_\Delta(\omega_\Delta)$ then the product depends only on the points in $\Delta^c$. This implies that the product is independent of $\omega_\Delta$ and we can put it into the constant. By definition ${\al}_r $ and $q_t$ 
	we get
	\begin{align*}
	&U_\Delta^{\hat{\sigma}_{\Lambda\backslash \Delta},\bar{\boldsymbol{\omega}}}(1,\mathcal{C}(\omega_\Delta\eta_{\Lambda \backslash \Delta}\bar{\omega}_{\Lambda^c}))\\
	&=c\al(-1)^{\vert\Delta^1_{\omega}\vert}\prod_{C\in \mathcal{C}_\Delta(\omega_{\Delta}\eta_{\Lambda\backslash \Delta }\bar{\omega}_{\Lambda^c}) } 
	\Big(\al_r^{\vert C\cap \Delta\vert}\mathds{1}_{\sigma_{C\cap\Lambda^c}=1}
	+\al_r^{-\vert C\cap \Lambda\backslash \Delta\vert}q_t^{-\sum_{i\in C\cap \Lambda\backslash \Delta}\hat{\sigma}_i}\mathds{1}_{\sigma_{C\cap\Lambda^c}=-1}\Big).
	\end{align*}
	The expression for $U_\Delta^{\hat{\sigma}_{\Lambda\backslash \Delta},\bar{\boldsymbol{\omega}}}(f,\mathcal{C}(\omega_\Delta\eta_{\Lambda \backslash \Delta}\bar{\omega}_{\Lambda^c}))$ can be obtained by following the same steps as above with some additional term $\prod_{i\in C\cap \Delta}p_t$. 
	This concludes the proof.

\end{proof}

Different to the soft-core case we do not take the limit of $\gamma^{\bar{\boldsymbol{\omega}}}_{\Lambda,\Delta,\alpha,t}$. It is not clear if the limit would exists. The only parts in $\gamma^{\bar{\boldsymbol{\omega}}}_{\Lambda,\Delta,\alpha,t}$ which depend on sites in $\Delta^c$ are the  exponentials $e^{- \sum_{i\in C\cap \Lambda\backslash \Delta}(\log(\al_r)+\log(q_t)\hat{\sigma}_i)}$. If we ignore these parts for infinite clusters we can define a probability kernel

\begin{definition}
	Let $\al\in \mathcal{M}_1(E)$, $t>0$ and $\Delta \Subset \Z^d$. Then for all $\mathcal{F}_\Delta$-measurable bounded functions $f$ and all $\boldsymbol{\eta}\in \tilde{\Omega}$  we define
	\begin{align*}
	&\gamma^\infty_{\Delta,\alpha,t}(f\vert {\boldsymbol{\eta}}_{ \Delta^c})=\\
	&	\frac{\sum_{\omega_\Delta\in\{0,1\}^\Delta} \frac{\al(0)^{\Delta_\omega^0}}{\al(-1)^{-\Delta_\omega^1}}f_{\eta_{ \Delta^c}}(\omega_\Delta) \prod_{C\in \mathcal{C}_{\Delta}(\omega_{\Delta}\eta_{ \Delta^c })} (
		\al_r^{\vert C\cap \Delta\vert} +e^{- \sum_{i\in C\cap \Delta^c}(\log(\al_r)+\log(q_t)\hat{\sigma}_i)}\mathds{1}_{\vert C \vert <\infty})}{\sum_{\omega_\Delta\in\{0,1\}^\Delta} \frac{\al(0)^{\Delta_\omega^0}}{\al(-1)^{-\Delta_\omega^1}}\prod_{C\in \mathcal{C}_{\Delta}(\omega_{\Delta}\eta_{ \Delta^c })} (
		\al_r^{\vert C\cap \Delta\vert} +e^{- \sum_{i\in C\cap \Delta^c}(\log(\al_r)+\log(q_t)\hat{\sigma}_i)}\mathds{1}_{\vert C \vert <\infty})}
	\end{align*}
	with 
	\begin{align*}
	&f_{\eta_{\Delta^c}}(\omega_\Delta)=
	\frac{\sum_{\hat{\sigma}_{\Delta_{\omega}^1}\in\{-1,1\}^{\Delta^1_{\omega}}}f(\omega_\Delta,\hat{\sigma}_{\Delta^1_{\omega}}0_{\Delta^0_{\omega}})\tilde{A}(\omega_\Delta,\hat{\sigma}_{\Delta_{\omega}^1}\hat{\sigma}_{{\Delta^c}_\eta^1})}{\prod_{C\in \mathcal{C}_{\Delta}(\omega_{\Delta}\eta_{ \Delta^c })} (
		\al_r^{\vert C\cap \Delta\vert} +e^{- \sum_{i\in C\cap \Delta^c}(\log(\al_r)+\log(q_t)\hat{\sigma}_i)}\mathds{1}_{\vert C \vert <\infty})}
	\end{align*}
	and
	\begin{align*}
	&\tilde{A}(\omega_\Delta,\hat{\sigma}_{\Delta_{\omega}^1}\hat{\sigma}_{{\Delta^c}_\eta^1})\\
	&=\prod_{C\in \mathcal{C}_{\Delta}(\omega_{\Delta}\eta_{\Delta^c })} (
	\al_r^{\vert C\cap \Delta\vert}\frac{p_t(1,1)^{\hat{\sigma}^+_{C\cap \Delta}}}{p_t(-1,1)^{-\hat{\sigma}^-_{C\cap \Delta}}}+\frac{p_t(-1,1)^{\hat{\sigma}^+_{C\cap \Delta}}}{p_t(1,1)^{-\hat{\sigma}^-_{C\cap \Delta}}}e^{- \sum_{i\in C\cap  \Delta^c}(\log(\al_r)+\log(q_t)\hat{\sigma}_i)}\mathds{1}_{\vert C \vert <\infty}).
	\end{align*}
\end{definition}
In the next theorem we prove that $(\gamma^\infty_{\Delta,\alpha,t})_{\Delta \Subset \Z^d}$ defines a specification.
\begin{theorem}\label{thm: time har spec}
	For all $t>0$ and all $\al \in \mathcal{M}_1(E)$ the family of probability kernels $(\gamma^\infty_{\Delta,\alpha,t})_{\Delta \Subset \Z^d}$ is a specification.
\end{theorem}
\begin{proof}
	First by defining the subset of finite cluster $\mathcal{C}^\mathrm{f}_{\Delta}(\omega_{\Delta}\eta_{\Delta^c })$ of $\mathcal{C}_{\Delta}(\omega_{\Delta}\eta_{\Delta^c })$ we can rewrite the probability kernels as
	\begin{align*}
	\gamma^\infty_{\Delta,\alpha,t}(f\vert {\boldsymbol{\eta}}_{ \Delta^c})=
	\frac{\sum_{\omega_\Delta\in\{0,1\}^\Delta} \frac{\al(0)^{\Delta_\omega^0}}{\al(1)^{-\Delta_\omega^1}}f_{\boldsymbol{\eta}_{ \Delta^c}}(\omega_\Delta) \prod_{C\in \mathcal{C}^\mathrm{f}_{\Delta}(\omega_{\Delta}\eta_{ \Delta^c })} (1
		+\al_r^{-\vert C\cap \Delta\vert}e^{- \sum_{i\in C\cap \Delta^c}(\log(\al_r)+\log(q_t)\hat{\sigma}_i)})}{\sum_{\omega_\Delta\in\{0,1\}^\Delta} \frac{\al(0)^{\Delta_\omega^0}}{\al(1)^{-\Delta_\omega^1}}\prod_{C\in \mathcal{C}^f_{\Delta}(\omega_{\Delta}\eta_{ \Delta^c })} (1
		+\al_r^{-\vert C\cap \Delta\vert}e^{- \sum_{i\in C\cap \Delta^c}(\log(\al_r)+\log(q_t)\hat{\sigma}_i)})}
	\end{align*}
	with 
	\begin{align*}
	f_{\boldsymbol{\eta}_{\Delta^c}}(\omega_\Delta)= \frac{\sum_{\hat{\sigma}_{\Delta_{\omega}^1}\in\{-1,1\}^{\Delta^1_{\omega}}}f(\omega_\Delta\eta_{\Delta^c},\hat{\sigma}_{\Delta^1_{\omega}}0_{\Delta^0_{\omega}}\hat{\sigma}_{\Delta^c})\tilde{A}_{\boldsymbol{\eta}_{ \Delta^c}}(\omega_\Delta,\hat{\sigma}_{\Delta_{\omega}^1})}{\prod_{C\in \mathcal{C}^\mathrm{f}_{\Delta}(\omega_{\Delta}\eta_{ \Delta^c })} (1
		+\al_r^{-\vert C\cap \Delta\vert}e^{- \sum_{i\in C\cap \Delta^c}(\log(\al_r)+\log(q_t)\hat{\sigma}_i)})}
	\end{align*}
	and
	\begin{align*}
	\tilde{A}_{\boldsymbol{\eta}_{ \Delta^c}}(\omega_\Delta,\hat{\sigma}_{\Delta_{\omega}^1})
	=\frac{p_t(1,1)^{\vert \hat{\sigma}_\Delta^+\vert}}{p_t(-1,1)^{-\vert \hat{\sigma}_\Delta^-\vert}}\prod_{C\in \mathcal{C}^\mathrm{f}_{\Delta}(\omega_{\Delta}\eta_{\Delta^c })} (
	1+e^{-(\sum_{i\in C\cap \Delta}\log(\al_r)+\log(q_t)\hat{\sigma}_i+ \sum_{i\in C\cap  \Delta^c}\log(\al_r)+\log(q_t)\hat{\sigma}_i)}).
	\end{align*}
	{Note that after pulling out $\prod_{C\in\mathcal{C}_{\Delta}(\omega_\Delta)}\a_r^{\vert C \cap \Delta \vert}$ the term after the first sum $\frac{\al(0)^{\Delta_\omega^0}}{\al(-1)^{-\Delta_\omega^1}}$ has changed to $\frac{\al(0)^{\Delta_\omega^0}}{\al(1)^{-\Delta_\omega^1}}$}. We have to check consistency and properness of these kernels. For properness let $B\in \mathcal{F}_{\Delta^c}$ then 	$f_{\eta_{\Delta^c}}(\omega_\Delta)=\mathds{1}_B$ since $\mathds{1}_B$ does not depend on the sum. Consequently $\gamma^\infty_{\Delta,\alpha,t}(\mathds{1}_B\vert {\boldsymbol{\eta}}_{ \Delta^c})= \mathds{1}_{B}$ which implies properness. 
	In this representation  consistency follows by the usual computation.
\end{proof}

For regimes where $\log(\al_r)+\log(q_t)\hat{\sigma}_i$ is strictly positive we could define $\gamma^\infty_{\Delta,\alpha,t}$ without the exclusion of infinite clusters. Hence in such regimes $\gamma^\infty_{\Delta,\alpha,t}$ might define a specification for the time-evolved measure and is quasilocal. To find the right regime we have only to check  the case where $\hat{\sigma}_i=-1$ since $q_t>1$ and we will later assume that $\a_r>1$. By this it follows that $t> \arccotanh(\a_r)=t_G$. We are now in the same situation as in \cite{jahnel-kuelske17b} and the following proofs of the lemmas are adaption of proofs in \cite{jahnel-kuelske17b}.

\begin{lemma}\label{lem: gamma infty adm}
	Let $\al\in\mathcal{M}_1(E)$ with $\al(-1)<\al(1)$  and $\mu^+\in\mathcal{G}(\gamma^{hc}_\al)$. Then $\mu_t^+$ is admitted by $(\gamma^\infty_{\Delta,\alpha,t})_{\Delta \Subset \Z^d}$ for all $t\geq t_G$.
\end{lemma}
\begin{proof}
{	Similarly as in the proof of Lemma  \ref{lem: admi. Dob} it is enough to prove that $$\lim_{\Lambda\uparrow \Z^d}\sup_{\eta\in \Omega}\vert \gamma^{\bar{\boldsymbol{\omega}}}_{\Lambda,\Delta,\alpha,t}(f\vert \eta_{\Lambda\backslash\Delta})-\gamma^\infty_{\Delta,\alpha,\beta,t}(f\vert \eta_{ \Delta^c})\vert=0$$
	for every local bounded function $f$. Since $t>t_G$ the exponential in $\gamma^{\bar{\boldsymbol{\omega}}}_{\Lambda,\Delta,\alpha,t}(f\vert \eta_{\Lambda\backslash\Delta})$  behaves fine and goes to zero for increasing $\Lambda$. Hence one can bound effects which are related to infinite clusters uniformly. 
For further details on the proof see Proposition $4.4$ in \cite{jahnel-kuelske17b}.}
\end{proof}

\begin{lemma}\label{lem: gamma infty quasi}
	Let $\al\in\mathcal{M}_1(E)$ with $\al(-1)<\al(1)$ and $t>t_G$. Then the specification $(\gamma^\infty_{\Delta,\alpha,t})_{\Delta \Subset \Z^d}$ is quasilocal.
\end{lemma}
\begin{proof}
	{Again the behavior at infinity does not effect the kernels $\gamma^\infty_{t,\Delta,\alpha,\beta}$ in a bad way since $t>t_G$. Hence the specification is quasilocal. 
	For a detailed proof see  $4.5$ in \cite{jahnel-kuelske17b}.}
\end{proof}
\begin{proof}[Proof of Theorem \ref{thm: hc is gibbs}]
	By Lemma \ref{lem: gamma infty adm} and Lemma \ref{lem: gamma infty quasi} the specification $(\gamma^\infty_{\Delta,\alpha,t})_{\Delta \Subset \Z^d}$ is a quasilocal specification for the time-evolved measure. Hence $\mu_t^+$ is a Gibbs measure.
	
\end{proof}
For the non-Gibbs part we define a similar kernel as $\gamma^\infty_{\Delta,\alpha,t}$ but with the difference that this kernel does not see infinite cluster.

\begin{definition}
	For $\al\in \mathcal{M}_1(E)$ and $\Delta\Subset\Z^d$ we define
	\begin{align*}
	\gamma^\mathrm{f}_{\Delta,\alpha,t}(f\vert {\boldsymbol{\eta}}_{ \Delta^c})=
	\frac{\sum_{\omega_\Delta\in\{0,1\}^\Delta} \frac{\al(0)^{\Delta_\omega^0}}{\al(-1)^{-\Delta_\omega^1}}f^\mathrm{f}_{\boldsymbol{\eta}_{ \Delta^c}}(\omega_\Delta) \prod_{C\in \mathcal{C}^\mathrm{f}_{\Delta}(\omega_{\Delta}\eta_{ \Delta^c })} (
		\al_r^{\vert C\cap \Delta\vert}+e^{- \sum_{i\in C\cap \Delta^c}(\log(\al_r)+\log(q_t)\hat{\sigma}_i)})}{\sum_{\omega_\Delta\in\{0,1\}^\Delta} \frac{\al(0)^{\Delta_\omega^0}}{\al(-1)^{-\Delta_\omega^1}}\prod_{C\in \mathcal{C}^\mathrm{f}_{\Delta}(\omega_{\Delta}\eta_{ \Delta^c })} (\al_r^{\vert C\cap \Delta\vert}
		+e^{- \sum_{i\in C\cap \Delta^c}(\log(\al_r)+\log(q_t)\hat{\sigma}_i)})}
	\end{align*}
	with 
	\begin{align*}
	f^\mathrm{f}_{\boldsymbol{\eta}_{\Delta^c}}(\omega_\Delta)= \frac{\sum_{\hat{\sigma}_{\Delta_{\omega}^1}\in\{-1,1\}^{\Delta^1_{\omega}}}f(\omega_\Delta,\hat{\sigma}_{\Delta^1_{\omega}}0_{\Delta^0_{\omega}})\tilde{A}_{\boldsymbol{\eta}_{ \Delta^c}}(\omega_\Delta,\hat{\sigma}_{\Delta_{\omega}^1})}{\prod_{C\in \mathcal{C}^\mathrm{f}_{\Delta}(\omega_{\Delta}\eta_{ \Delta^c })} (
		\al_r^{\vert C\cap \Delta\vert}	+e^{- \sum_{i\in C\cap \Delta^c}(\log(\al_r)+\log(q_t)\hat{\sigma}_i)})}
	\end{align*}
	and
	\begin{align*}
	\tilde{A}_{\boldsymbol{\eta}_{ \Delta^c}}(\omega_\Delta,\hat{\sigma}_{\Delta_{\omega}^1})
	=\prod_{C\in \mathcal{C}^\mathrm{f}_{\Delta}(\omega_{\Delta}\eta_{\Delta^c })} (
	\al_r^{\vert C\cap \Delta\vert}\frac{p_t(1,1)^{\hat{\sigma}^+_{C\cap \Delta}}}{p_t(-1,1)^{-\hat{\sigma}^-_{C\cap \Delta}}}+\frac{p_t(-1,1)^{\hat{\sigma}^+_{C\cap \Delta}}}{p_t(1,1)^{-\hat{\sigma}^-_{C\cap \Delta}}}e^{- \sum_{i\in C\cap \Lambda\backslash \Delta}(\log(\al_r)+\log(q_t)\hat{\sigma}_i)}).
	\end{align*}
\end{definition}
 The nice property of these kernels is that they are conditional probabilities of the time-evolved measure $\mu_t$ if we exclude configurations which have infinite clusters connected to some $\Delta$. 
\begin{lemma}\label{lem : con. prob mu_t gamma^f}
	Let $\mu^+\in \mathcal{G}(\gamma^{hc}_\al)$ for the asymmetric or $\mu\in \mathcal{G}(\gamma^{hc}_\al)$ for the symmetric model. Then for all $\Delta\Subset\Z^d$ and $0<t<\infty$ for $\al(1)=\al(-1)$ or $t<t_G$ for $\al(1)>\al(-1)$  we have
	\begin{align*}
	\mu_t(\cdot\vert \boldsymbol{\eta}_{\Delta^c})\mathds{1}_{\Delta\nleftrightarrow\infty}(\cdot\eta_{\Lambda^c}) =  \gamma^\mathrm{f}_{\Delta,\alpha,t}(\cdot\vert \boldsymbol{\eta}_{\Delta^c})\mathds{1}_{\Delta\nleftrightarrow\infty}(\cdot\eta_{\Lambda^c})
	\end{align*}
	for all $\eta\in\tilde{\Omega}$ $\mu_t$-$a.s.$. Here the event $\{\Delta\nleftrightarrow\infty\}$ describes that $\Delta$ is not connected to any infinite cluster.
\end{lemma}
\begin{proof}
	This lemma is only useful if the event $\{\Delta\nleftrightarrow\infty\}$ has positive probability under $\mu_t$. Since the particles remain on their place under the time evolution we have $\mu_t(\Delta\nleftrightarrow\infty)=\mu(\Delta\nleftrightarrow\infty)$. Let $\Lambda$ be a finite set which contains $\Delta$ and satisfies $\min_{i\in \Lambda^c,j\in\Delta}\Vert i-j\Vert_1>5$. With the DLR equation for the starting Widom-Rowlinson Gibbs measure it follows that 
	\begin{align*}
	\mu(\Delta\nleftrightarrow\infty) = \mu(\gamma^{hc}_{\Lambda,\al}(\Delta\nleftrightarrow\infty\vert\cdot))\geq   \mu(\gamma^{hc}_{\Lambda,\al}(\{\forall i\in \Lambda\backslash\Delta\,:\; \omega_i=0\}\vert\cdot)) > 0.
	\end{align*}
	For the proof of the above equation we define the cofinal sequence defined by the sets $\Lambda_n:=[-n,n]^d\cap \Z^d$. Then  for all bounded $\mathcal{F}_\Delta$-measurable functions $f$ and $n$ sufficiently large such that $\Delta\subset \Lambda_n$ it follows that 
	\begin{align*}
	\mu_t((f-\gamma^\mathrm{f}_{\Delta,\alpha,t}(f\vert \cdot_{\Delta^c}))\mathds{1}_{\Delta\nleftrightarrow\infty}) \leq \mu_t((f-\gamma^\mathrm{f}_{\Delta,\alpha,t}(f\vert \cdot_{\Delta^c}))\mathds{1}_{\Delta\nleftrightarrow\Lambda^c_n})+ 2\Vert f\Vert \mu(\mathds{1}_{\Delta\nleftrightarrow\infty}-\mathds{1}_{\Delta\nleftrightarrow\Lambda_n^c}).
	\end{align*}
	where $\Delta\nleftrightarrow\Lambda_n^c$ is the event that $\Delta$ and $\Lambda_n^c$ are not connected with some cluster. Note, that we replaced $\mu_t$ by $\mu$ since the events only depends on the locations of particles. Since $\mathds{1}_{\Delta\nleftrightarrow\Lambda_n^c}$ converges to $\mathds{1}_{\Delta\nleftrightarrow\infty}$ point-wise as $n$ tends to infinity the second summand converges to $0$ by dominated convergence. 
	
	For the first summand it suffices to prove the statement for extremal initial Gibbs measures $\nu\in \ex \mathcal{G}(\gamma^{hc}_\al)$ and then use the extremal decomposition $\mu = \int_{\ex \mathcal{G}(\gamma^{hc}_\al)}\nu\, w_{\mu}(d\nu)$ (see \cite[Theorem 7.26]{georgii-book}). By extremality we can choose some suitable boundary condition $\bar{\boldsymbol{\omega}}$ such that we can write $\nu_t  = \lim_{\Lambda\uparrow \Z^d} 	\gamma^{\bar{\boldsymbol{\omega}}}_{t,\Lambda,\alpha}=\lim_{\Lambda\uparrow \Z^d} 	\gamma^{hc}_{\Lambda,\al}(p_t\vert\bar{\boldsymbol{\omega}})$. Then it follows for all $\Lambda_n\subset \Lambda$ that
	\begin{align*}
	\mu_t((f-\gamma^\mathrm{f}_{\Delta,\alpha,t}(f\vert \cdot_{\Delta^c}))\mathds{1}_{\Delta\nleftrightarrow\Lambda^c_n})=\lim_{\Lambda\uparrow \Z^d} 	\gamma^{\bar{\boldsymbol{\omega}}}_{t,\Lambda,\alpha}((f-\gamma^\mathrm{f}_{\Delta,\alpha,t}(f\vert \cdot_{\Delta^c}))\mathds{1}_{\Delta\nleftrightarrow\Lambda_n^c}) .
	\end{align*}
	Note that $(f-\gamma^\mathrm{f}_{\Delta,\alpha,t}(f\vert \cdot_{\Delta^c}))\mathds{1}_{\Delta\nleftrightarrow\Lambda_n^c}$ is a local function because of the indicator function.
	On the event $\mathds{1}_{\Delta\nleftrightarrow\Lambda_n^c}$ the probability kernels $\gamma^\mathrm{f}_{\Delta,\alpha,t}$ and $\gamma^{\bar{\omega}}_{\Lambda,\Delta,\alpha,t}$  coincides which implies 
	\begin{align*}
	\mu_t((f-\gamma^\mathrm{f}_{\Delta,\alpha,t}(f\vert \cdot_{\Delta^c}))\mathds{1}_{\Delta\nleftrightarrow\Lambda^c_n}) &= \lim_{\Lambda\uparrow \Z^d} 	\gamma^{\bar{\boldsymbol{\omega}}}_{\Lambda,\alpha,t}((f-\gamma^{\bar{\boldsymbol{\omega}}}_{\Lambda,\Delta,\alpha,t}(f\vert \cdot_{\Delta^c}))\mathds{1}_{\Delta\nleftrightarrow\Lambda_n^c})\\
	&=\lim_{\Lambda\uparrow \Z^d} 	\gamma^{\bar{\boldsymbol{\omega}}}_{\Lambda,\alpha,t}((f-f)\mathds{1}_{\Delta\nleftrightarrow\Lambda_n^c}) = 0.
	\end{align*}
\end{proof}
The next Lemma will later imply a contradiction to the statement that there exists a quasilocal specification for the time evolved measure. 
\begin{lemma}\label{lem : gamma^f bad}
	Let $\al\in \mathcal{M}_1(E)$. Then for all $\Lambda\Subset \Z^d$ and $t>0$ for $\al(1) =\al(-1)$ or $t<t_G$ for $\al(1)>\al(-1)$ there exists $n_0\in\N$, a  $\mathcal{F}_{\Delta}$-measurable bounded function $f$ and $\delta>0$ such that for all $n\geq n_0$ 
	\begin{align}\label{eq: gamma f pos}
	\inf_{\boldsymbol{\eta}\in \{\Delta\leftrightarrow \Lambda^c_n\}} \vert \gamma^\mathrm{f}_{\Delta,\alpha,t}(f\vert \boldsymbol{\eta}_{\Lambda\backslash\Delta}\boldsymbol{\eta}^+_{\Lambda_n\backslash\Lambda})-\gamma^\mathrm{f}_{\Delta,\alpha,t}(f\vert \boldsymbol{\eta}_{\Lambda\backslash\Delta}\boldsymbol{\eta}^-_{\Lambda_n\backslash\Lambda})\vert>\delta
	\end{align}
	where $\boldsymbol{\eta}^{\pm}$ are configurations with $\sigma_i=\pm1$ for all $i$ with $\eta_i=1$. 
\end{lemma}
\begin{proof}
 The proof follows by the same arguments as in the proof of Proposition 4.12 of \cite{jahnel-kuelske17b}. By the assumption for the parameters one is able to manipulate the exponentials in $\gamma^\mathrm{f}_{\Delta,\alpha,t}(f\vert \boldsymbol{\eta}_{\Lambda\backslash\Delta}\boldsymbol{\eta}^+_{\Lambda_n\backslash\Lambda}) $ and $\gamma^\mathrm{f}_{\Delta,\alpha,t}(f\vert \boldsymbol{\eta}_{\Lambda\backslash\Delta}\boldsymbol{\eta}^-_{\Lambda_n\backslash\Lambda})$ to be smaller than 1 and bigger than one, respectively. This leads to \eqref{eq: gamma f pos}.  Note that we do not have to bound the number of particles in $\Lambda\backslash\Delta$ as in \cite{jahnel-kuelske17b} since $\Lambda\backslash \Delta_\omega^1 \leq \vert \Lambda\backslash \Delta \vert <\infty$.
\end{proof}
On configurations which have only finite clusters connected to some finite set  $\Delta$ we have shown that the conditional probability of $\mu_t$ is equal to the probability kernel $\gamma_\Delta^\mathrm{f}$ by Lemma \ref{lem : con. prob mu_t gamma^f} and by Lemma $\ref{lem : gamma^f bad}$ we have shown that the quasilocality of $\gamma^\mathrm{f}$ will fail if there exist big enough clusters. This will help to prove Theorem \ref{thm: non-gibbs hc} and Theorem \ref{thm: non-gibbs hc low}. Let $\emptyset_\Lambda$ the event that there are no particles in $\Lambda$.

\begin{proof}[Proof of Theorem \ref{thm: non-gibbs hc}]
Let assume that $(\tilde{\gamma}_\Lambda)_{\Lambda\Subset\Z^d}$ is a quasilocal specification for $\mu_t$. We will derive a contradiction. Define the integral
	\begin{align*}
	I_{\Lambda,n}^\delta 
	= \int\mu(d\boldsymbol{\omega})\int \mu_t(d\omega^{\hat{\sigma}}\vert \boldsymbol{{\omega}})\frac{\mathds{1}_{\emptyset_{\bar{\Lambda}_n\backslash \Lambda_n}}(\boldsymbol{\omega})}{\gamma^{hc}_{\bar{\Lambda}_n\backslash\Lambda_n,\al}(\mathds{1}_{\emptyset_{\bar{\Lambda}_n\backslash \Lambda_n}}(\cdot)\vert \boldsymbol{\omega})}\mathds{1}_{\vert \tilde{\gamma}_\Delta(f\vert \omega^{\hat{\sigma}}_{\Lambda\backslash \Delta}\omega^+_{\Lambda_n\backslash\Lambda}\omega^{\hat{\sigma}}_{\Lambda_n^c})-\tilde{\gamma}_\Delta(f\vert \omega^{\hat{\sigma}}_{\Lambda\backslash \Delta}\omega^-_{\Lambda_n\backslash\Lambda}\omega^{\hat{\sigma}}_{\Lambda_n^c})\vert>\delta}
	\end{align*}
	where $f$ is a bounded $\mathcal{F}_{\Delta}$-measurable function, $\mu_t(d\omega^{\hat{\sigma}}\vert \boldsymbol{{\omega}})$ is the independent spin flip and $\bar{\Lambda}_n, \Lambda_n\Subset \Z^d$ with $\Lambda\subset\Lambda_n\subset\bar{\Lambda}_n$, $\min_{i\in \Lambda_n,j\in\bar{\Lambda}_n^c}\Vert i-j\Vert_1>5$ and $\Lambda_n\uparrow\Z^d$.  We can bound the last indicator function from above if we use the supremum over all configurations outside of $\Lambda$ and then in a second step we use the DLR-equation for $\mu$
	\begin{alignat*}{2}
	& I_{\Lambda,n}^\delta 
	\\& \leq\int\hspace{-0.7pt}\mu(d\boldsymbol{\omega})\int \mu_t(d\omega^{\hat{\sigma}}\vert \boldsymbol{{\omega}})\frac{\mathds{1}_{\emptyset_{\bar{\Lambda}_n\backslash \Lambda_n}}(\boldsymbol{\omega})}{\gamma^{hc}_{\bar{\Lambda}_n\backslash\Lambda_n,\al}(\mathds{1}_{\emptyset_{\bar{\Lambda}_n\backslash \Lambda_n}}(\cdot)\vert \boldsymbol{\omega})}\mathds{1}_{\sup_{\boldsymbol{\omega}^1,\boldsymbol{\omega}^2}\vert \tilde{\gamma}_\Delta(f\vert \omega^{\hat{\sigma}}_{\Lambda\backslash \Delta}\boldsymbol{\omega}^1_{\Lambda^c})-\tilde{\gamma}_\Delta(f\vert \omega^{\hat{\sigma}}_{\Lambda\backslash \Delta}\boldsymbol{\omega}^2_{\Lambda^c})\vert>\delta}\\
	& =\int\hspace{-0.7pt}\mu(d\boldsymbol{\omega})\gamma^{hc}_{\bar{\Lambda}_n\backslash\Lambda_n,\al}\left(\int \mu_t(d\omega^{\hat{\sigma}}\vert \boldsymbol{{\omega}})\frac{\mathds{1}_{\emptyset_{\bar{\Lambda}_n\backslash \Lambda_n}}(\boldsymbol{\omega})}{\gamma^{hc}_{\bar{\Lambda}_n\backslash\Lambda_n,\al}(\mathds{1}_{\emptyset_{\bar{\Lambda}_n\backslash \Lambda_n}}(\cdot)\vert \boldsymbol{\omega})}
\mathds{1}_{\sup_{\boldsymbol{\omega}^1,\boldsymbol{\omega}^2}\vert \tilde{\gamma}_\Delta(f\vert \omega^{\hat{\sigma}}_{\Lambda\backslash \Delta}\boldsymbol{\omega}^1_{\Lambda^c})-\tilde{\gamma}_\Delta(f\vert \omega^{\hat{\sigma}}_{\Lambda\backslash \Delta}\boldsymbol{\omega}^2_{\Lambda^c})\vert>\delta}\bigg\vert \boldsymbol{\omega}\right)\\
	&=\int\hspace{-0.7pt}\mu(d\boldsymbol{\omega})\gamma^{hc}_{\bar{\Lambda}_n\backslash\Lambda_n,\al}\left(\frac{\mathds{1}_{\emptyset_{\bar{\Lambda}_n\backslash \Lambda_n}}(\boldsymbol{\omega})}{\gamma^{hc}_{\bar{\Lambda}_n\backslash\Lambda_n,\al}(\mathds{1}_{\emptyset_{\bar{\Lambda}_n\backslash \Lambda_n}}(\cdot)\vert \boldsymbol{\omega})}\bigg\vert \boldsymbol{\omega}\right)
	\int \mu_t(d\omega^{\hat{\sigma}}_{\Lambda\backslash\Delta}\vert \boldsymbol{{\omega}})\mathds{1}_{\sup_{\boldsymbol{\omega}^1,\boldsymbol{\omega}^2}\vert \tilde{\gamma}_\Delta(f\vert \omega^{\hat{\sigma}}_{\Lambda\backslash \Delta}\boldsymbol{\omega}^1_{\Lambda^c})-\tilde{\gamma}_\Delta(f\vert \omega^{\hat{\sigma}}_{\Lambda\backslash \Delta}\boldsymbol{\omega}^2_{\Lambda^c})\vert>\delta}\\
	& =\int \hspace{-0.7pt}\mu_t(d\boldsymbol{\eta}) \mathds{1}_{\sup_{\boldsymbol{\omega}^1,\boldsymbol{\omega}^2}\vert \tilde{\gamma}_\Delta(f\vert \boldsymbol{\eta}_{\Lambda\backslash \Delta}\boldsymbol{\omega}^1_{\Lambda^c})-\tilde{\gamma}_\Delta(f\vert \boldsymbol{\eta}_{\Lambda\backslash \Delta}\boldsymbol{\omega}^2_{\Lambda^c})\vert>\delta}.
	\end{alignat*}	
	
	Since the last integral is bounded by $1$ it follows by dominated convergence and by the assumption of quasilocality that the integral tends to zero as $\Lambda\uparrow\Z^d$. \\
	The reason for the indicator function $\mathds{1}_{\emptyset_{\bar{\Lambda}_n\backslash \Lambda_n}}(\boldsymbol{\cdot})$ is that we know that $\Delta$ and $\Lambda_n^c$ are disconnected. By Lemma \ref{lem : con. prob mu_t gamma^f} we have 
	\begin{align*}
	\tilde{\gamma}_\Delta(f\vert \boldsymbol{\eta}_{\Delta^c}) = \mu_t(f\vert \boldsymbol{\eta}_{\Delta^c}) = \gamma^\mathrm{f}_{\Delta,\al,t} (f\vert \boldsymbol{\eta}_{\Delta^c})
	\end{align*}
	on the above event. This implies that we can replace in $I^\delta_{\Lambda,n}$ the specification $\tilde{\gamma}$ with the kernel $\gamma^\mathrm{f}$. This gives the lower bound
	\begin{align*}
	I^{\delta}_{\Lambda,n}
	&=\int\mu(d\boldsymbol{\omega})\int \mu_t(d\omega^{\hat{\sigma}}\vert \boldsymbol{{\omega}})\frac{\mathds{1}_{\emptyset_{\bar{\Lambda}_n\backslash \Lambda_n}}(\boldsymbol{\omega})}{\gamma^{hc}_{\bar{\Lambda}_n\backslash\Lambda_n,\al}(\mathds{1}_{\emptyset_{\bar{\Lambda}_n\backslash \Lambda_n}}(\cdot)\vert \boldsymbol{\omega})}\mathds{1}_{\vert {\gamma}^\mathrm{f}_\Delta(f\vert \omega^{\hat{\sigma}}_{\Lambda\backslash \Delta}\omega^+_{\Lambda_n\backslash\Lambda})-{\gamma}^{\mathrm{f}}_\Delta(f\vert \omega^{\hat{\sigma}}_{\Lambda\backslash \Delta}\omega^-_{\Lambda_n\backslash\Lambda})\vert>\delta}\\
	&\geq \int\mu(d\boldsymbol{\omega})\int \mu_t(d\omega^{\hat{\sigma}}\vert \boldsymbol{{\omega}})\frac{\mathds{1}_{\emptyset_{\bar{\Lambda}_n\backslash \Lambda_n}}(\boldsymbol{\omega})}{\gamma^{hc}_{\bar{\Lambda}_n\backslash\Lambda_n,\al}(\mathds{1}_{\emptyset_{\bar{\Lambda}_n\backslash \Lambda_n}}(\cdot)\vert \boldsymbol{\omega})}\mathds{1}_{\vert {\gamma}^\mathrm{f}_\Delta(f\vert \omega^{\hat{\sigma}}_{\Lambda\backslash \Delta}\omega^+_{\Lambda_n\backslash\Lambda})-{\gamma}^{\mathrm{f}}_\Delta(f\vert \omega^{\hat{\sigma}}_{\Lambda\backslash \Delta}\omega^-_{\Lambda_n\backslash\Lambda})\vert>\delta}
	\mathds{1}_{\{\Delta \leftrightarrow \Lambda_n^c\}}.
	\end{align*}
	By Lemma \ref{lem : gamma^f bad} we find a $\delta>0$ and a measurable bounded function such that for all n which are bigger than some $n_0$ the first indicator function is equal to $1$. This implies that the integral over $\mu_t(\cdot\vert \boldsymbol{\omega})$ is equal to $1$ and we can bound the fraction from below by $1$. Hence
	\begin{align*}
	I^{\delta}_{\Lambda,n}\geq \int\mu(d\boldsymbol{\omega}) \mathds{1}_{\{\Delta \leftrightarrow \Lambda_n^c\}} \geq \mu(\{\Delta \leftrightarrow \infty\})>0.
	\end{align*}
	as we are in the percolation regime. This gives the desired contradiction. For the full measure of the bad configuration note that $\lim_{\Lambda\uparrow \Z^d} \mu(\{\Lambda \leftrightarrow \infty\}) =1$.
\end{proof}
\begin{proof}[Proof of Theorem \ref{thm: non-gibbs hc low}]
	Again assume that $(\tilde{\gamma}_\Lambda)_{\Lambda\Subset\Z^d}$ is a quasilocal specification for $\mu_t$. Define the set $X= \{i\in \Z^d : i_1 \geq 0 \text{ and } i_k = 0 \}$ and the configuration $\boldsymbol{\zeta}^b$  with $\boldsymbol{\zeta}_i^b = (1,1)$ for all $i \in X$ and $\boldsymbol{\zeta}_i^b = (0,0)$ for all $i\in X^c$. Furthermore, we define the function  $$g^n(\boldsymbol{\omega})= 1_{\emptyset_{\bar \Lambda_n\backslash \Lambda_n} }(\omega) 1_{\boldsymbol{\omega}_{\Lambda_n}= \boldsymbol{\zeta}^b_{\Lambda_n}} \gamma^{hc}_{\bar{\Lambda}_n\backslash \Lambda_n,\al}(1_{\emptyset_{\bar \Lambda_n\backslash \Lambda_n} }(\cdot) 1_{\boldsymbol{\cdot}_{\Lambda_n}= \boldsymbol{\zeta}_{\Lambda_n}}\vert \omega)^{-1}$$ for $\Lambda \subset \Lambda_n \subset \bar{\Lambda}_n \Subset \Z^d$ with $\min_{i\in \Lambda_n,j\in\bar{\Lambda}_n^c}\Vert i-j\Vert_1>5$.
	Then
	
	\begin{align*}
	\tilde{I}_{\Lambda,n} &:= \int \int g^n(\boldsymbol{\omega}) \vert \tilde\gamma_\Delta(f\vert \omega^+_{\Lambda\backslash \Delta} \omega^+_{\Lambda_n\backslash \Lambda}\hat{\boldsymbol{\eta}}_{\Lambda_n^c})-\tilde\gamma_\Delta(f\vert \omega^+_{\Lambda\backslash \Delta} \omega^-_{\Lambda_n\backslash \Lambda}\hat{\boldsymbol{\eta}}_{\Lambda_n^c})\vert \mu_t(d\hat{\boldsymbol{\eta}}\vert \boldsymbol{\omega})\mu(d\boldsymbol{\omega})\\
	&= \int  g^n(\boldsymbol{\omega}) \big\vert \gamma^{\mathrm{f}}_\Delta(f\vert \omega^+_{\Lambda\backslash \Delta} \omega^+_{\Lambda_n\backslash \Lambda})-\gamma^{\mathrm{f}}_\Delta(f\vert \omega^+_{\Lambda\backslash \Delta} \omega^-_{\Lambda_n\backslash \Lambda})\big\vert \mu(d\boldsymbol{\omega}) \\
	&\geq \delta \int  g^n(\boldsymbol{\omega}) \mu(d\boldsymbol{\omega}) = \delta
	\end{align*}
	for $n$ big enough. On the other hand 
	\begin{align*}
	\tilde{I}_{\Lambda,n} &\leq \int g^n(\boldsymbol{\omega}) \sup_{\boldsymbol{\eta}^1,\boldsymbol{\eta}^2\in \Omega}\vert \tilde\gamma_\Delta(f\vert \omega^+_{\Lambda\backslash \Delta} {\boldsymbol{\eta}}_{\Lambda^c})-\tilde\gamma_\Delta(f\vert \omega^+_{\Lambda\backslash \Delta} {\boldsymbol{\eta}}_{\Lambda^c})\vert \mu(d\boldsymbol{\omega})\\
	&= \int \gamma^{hc}_{\bar{\Lambda}_n\backslash \Lambda_n,\al}\Big(g^n(\cdot) \sup_{\boldsymbol{\eta}^1,\boldsymbol{\eta}^2\in \Omega}\vert \tilde\gamma_\Delta(f\vert \cdot^+_{\Lambda\backslash \Delta} {\boldsymbol{\eta}}^1_{\Lambda^c})-\tilde\gamma_\Delta(f\vert \cdot^+_{\Lambda\backslash \Delta} {\boldsymbol{\eta}}^2_{\Lambda^c})\vert \Big\vert \boldsymbol{\omega}\Big) \mu(d\boldsymbol{\omega}).\\
	\end{align*}
	Because of the decoupling event $g^n(\cdot) \sup_{\boldsymbol{\eta}^1,\boldsymbol{\eta}^2\in \Omega}\vert \tilde\gamma_\Delta(f\vert \cdot^+_{\Lambda\backslash \Delta} {\boldsymbol{\eta}}^1_{\Lambda^c})-\tilde\gamma_\Delta(f\vert \cdot^+_{\Lambda\backslash \Delta} {\boldsymbol{\eta}}^2_{\Lambda^c})\vert$ does not depend on the configurations in ${\bar{\Lambda}^c_n}$. This leads to 
	\begin{align*}
	\tilde{I}_{\Lambda,n} &\leq \gamma^{hc}_{\bar{\Lambda}_n\backslash \Lambda_n,\al}\Big(g^n(\cdot) \sup_{\boldsymbol{\eta}^1,\boldsymbol{\eta}^2\in \Omega}\vert \tilde\gamma_\Delta(f\vert \cdot^+_{\Lambda\backslash \Delta} {\boldsymbol{\eta}}_{\Lambda^c})-\tilde\gamma_\Delta(f\vert \cdot^+_{\Lambda\backslash \Delta} {\boldsymbol{\eta}}_{\Lambda^c})\Big\vert \emptyset\Big) \\
	&=  \sup_{\boldsymbol{\eta}^1,\boldsymbol{\eta}^2\in \Omega}\vert \tilde\gamma_\Delta(f\vert \boldsymbol{\zeta}^b_{\Lambda\backslash \Delta} {\boldsymbol{\eta}}_{\Lambda^c})-\tilde\gamma_\Delta(f\vert \boldsymbol{\zeta}^b_{\Lambda\backslash \Delta} {\boldsymbol{\eta}}_{\Lambda^c})\vert.
	\end{align*}
	Now $\lim_{\Lambda\uparrow \Z^d} \sup_{\boldsymbol{\eta}^1,\boldsymbol{\eta}^2\in \Omega}\vert \tilde\gamma_\Delta(f\vert \boldsymbol{\zeta}^b_{\Lambda\backslash \Delta} {\boldsymbol{\eta}}_{\Lambda^c})-\tilde\gamma_\Delta(f\vert \boldsymbol{\zeta}^b_{\Lambda\backslash \Delta} {\boldsymbol{\eta}}_{\Lambda^c})\vert $ has to be bigger than $0$ otherwise it would lead to the contradiction  $0\geq \delta >0 $. Thus $\boldsymbol{\zeta}^b$ is a bad configuration which is contradiction to the assumption that $\tilde\gamma$ is quasilocal.
	
		For the last part of the theorem let $\Omega^{\mathrm{f}} \subset \Omega$ the space of configurations which contain no infinite cluster. Then it follows by the very definition of the low intensity regime that $\mu_t(\Omega^{\mathrm{f}}) =1$. For $\boldsymbol{\eta}\in \Omega^{\mathrm{f}}$  there exists a finite $\Lambda\supset \Delta$ such that there is no connection between $\Delta$ and $\Lambda^c$. Hence 
		\begin{align*}
		\sup_{\boldsymbol{\zeta},\boldsymbol{\chi}\in\tilde{\Omega}} \vert \gamma^\infty_\Delta (f\vert \boldsymbol{\eta}_{\Lambda\backslash \Delta}\boldsymbol{\zeta}_{\Lambda^c} )-  \gamma^\infty_\Delta (f\vert \boldsymbol{\eta}_{\Lambda\backslash \Delta}\boldsymbol{\chi}_{\Lambda^c} )\vert =0
		\end{align*}
		and by Lemma \ref{lem : con. prob mu_t gamma^f}  we have $\int_{\Omega^{\mathrm{f}}} g(\boldsymbol{\eta}) \mu_t(d\boldsymbol{\eta}) = \int_{\Omega^{\mathrm{f}}}  \gamma_\Delta^\infty (g \vert \boldsymbol{\eta})  \mu_t(d\boldsymbol{\eta})$ for every measurable function $g$ and $\Delta \Subset \Z^d $.
\end{proof}
\section*{Acknowledgments}
Sascha Kissel has been supported by
the German Research Foundation (DFG) via Research Training Group RTG 2131 \textit{High dimensional
	Phenomena in Probability - Fluctuations and Discontinuity}.

	 \bibliographystyle{plain}
	 \bibliography{bibtex_uni}

\begin{thebibliography}{10}

\bibitem{chayes-chayes-kotecky95}
J.~T. Chayes, L.~Chayes, and R.~Koteck\'{y}.
\newblock The analysis of the {W}idom-{R}owlinson model by stochastic geometric
  methods.
\newblock {\em Comm. Math. Phys.}, 172:551--569, 1995.

\bibitem{enter-ermolaev-Iacobelli-kuelske12}
A.~van Enter, V.~Ermolaev, G.~Iacobelli, and C.~K\"ulske.
\newblock Gibbs–non-{G}ibbs properties for evolving {I}sing models on trees.
\newblock {\em Ann. Inst. H. Poincaré Probab. Statist.}, 48:774--791, 2012.

\bibitem{enter-fernandez-hollander-redig02}
A.~van Enter, R.~Fern\'{a}ndez, F.~den Hollander, and F.~Redig.
\newblock Possible {L}oss and {R}ecovery of {G}ibbsianness during the
  stochastic evolution of {G}ibbs measures.
\newblock {\em Comm. Math. Phys.}, 226:101--130, 2002.

\bibitem{enter-fernandez-sokal93}
A.~van Enter, R.~Fern\'{a}ndez, and A.~Sokal.
\newblock Regularity {P}roperties and {P}athologies of {P}osition-{S}pace
  {R}enormalization-{G}roup {T}ransformations: {S}cope and {L}imitations of
  {G}ibbsian {T}heory.
\newblock {\em J. Stat. Phys.}, 72:879--1167, 1993.

\bibitem{fernandez-pfister97}
R.~Fern{\'a}ndez and C.~Pfister.
\newblock Global specifications and nonquasilocality of projections of {G}ibbs
  measures.
\newblock {\em Ann. Probab}, 25:1284--1315, 1997.

\bibitem{friedli_velenik_2017}
S.~Friedli and Y.~Velenik.
\newblock {\em {S}tatistical {M}echanics of {L}attice {S}ystems: {A} {C}oncrete
  {M}athematical {I}ntroduction}.
\newblock Cambridge University Press, 2017.

\bibitem{gallavotti-lebowitz71}
G.~Gallavotti and J.~Lebowitz.
\newblock Phase {T}ransitions in {B}inary {L}attice {G}ases.
\newblock {\em J. Math. Phys.}, 12:1129--1133, 1971.

\bibitem{georgii-book}
H.-O. Georgii.
\newblock {\em Gibbs measures and phase transitions}, volume~9 of {\em de
  Gruyter Studies in Mathematics}.
\newblock Walter de Gruyter \& Co., Berlin, second edition, 2011.

\bibitem{georgii-haggstrom-maes01}
H.-O. Georgii, O.~H\"aggstr\"om, and C.~Maes.
\newblock The random geometry of equilibrium phases.
\newblock In {\em Phase {T}ransitions and {C}ritical {P}henomena}, pages
  1--147. Academic, London, 2001.

\bibitem{higuchi83}
Y.~Higuchi.
\newblock Applications of a stochastic inequality to two-dimensional {I}sing
  and {W}idom- {R}owlinson models.
\newblock {\em Lect. Notes in Math.}, 1021:230--237, 1983.

\bibitem{higuchi-takei04}
Y.~Higuchi and M.~Takei.
\newblock Some results on the phase structure of the two-dimensional
  {W}idom-{R}owlinson model.
\newblock {\em Osaka J. Math.}, 41:237--255, 2004.

\bibitem{hollander-redig-zuijlen15}
F.~den Hollander, F.~Redig, and W.~van Zuijlen.
\newblock Gibbs-non-{G}ibbs dynamical transitions for mean-field interacting
  {B}rownian motions.
\newblock {\em Stoch. Process. Appl.}, 125:371--400, 2015.

\bibitem{jahnel-kuelske17c}
B.~Jahnel and C.~K\"ulske.
\newblock Gibbsian representation for point processes via hyperedge potentials.
\newblock {\em arXiv:1707.05991}, 2017.

\bibitem{jahnel-kuelske17a}
B.~Jahnel and C.~K\"ulske.
\newblock Sharp thresholds for {G}ibbs-non-{G}ibbs transition in the fuzzy
  {P}otts models with a {K}ac-type interaction.
\newblock {\em Bernoulli}, 23:2808--2827, 2017.

\bibitem{jahnel-kuelske17b}
B.~Jahnel and C.~K\"ulske.
\newblock The {W}idom–{R}owlinson model under spin flip: {I}mmediate loss and
  sharp recovery of quasilocality.
\newblock {\em Ann. Appl. Probab.}, 27:3845--3892, 2017.

\bibitem{kissel-kuelske18}
S.~Kissel and C.~K\"ulske.
\newblock Dynamical {G}ibbs-non-{G}ibbs transitions in {C}urie-{W}eiss
  {W}idom-{R}owlinson models.
\newblock {\em To appear in Markov Proc. Rel. Fields}, 2019+.

\bibitem{kissel-kuelske-rozikov19}
S.~Kissel, C.~K\"ulske, and U.A. Rozikov.
\newblock Hard-{C}ore and {S}oft-{C}ore {W}idom-{R}owlinson models on {C}ayley
  trees.
\newblock {\em To appear in J. Stat. Mech.}, 2019+.

\bibitem{kraaij-redig-zuijlen17}
R.~Kraaij, F.~Redig, and W.~van Zuijlen.
\newblock A {H}amilton-{J}acobi point of view on mean-field {G}ibbs-non-{G}ibbs
  transitions.
\newblock {\em arXiv:1711.03489}, 2017.

\bibitem{kuelske19}
C.~K\"ulske.
\newblock Gibbs-non {G}ibbs transitions in different geometries: The
  {W}idom-{R}owlinson model under stochastic spin-flip dynamics.
\newblock {\em arXiv:1901.10347}, 2019.

\bibitem{kuelske-le07}
C.~K\"ulske and A.~Le~Ny.
\newblock Spin-flip dynamics of the {C}urie-{W}eiss model: {L}oss of
  {G}ibbsianness with possibly broken symmetry.
\newblock {\em Comm. Math. Phys.}, 271:431--454, 2007.

\bibitem{kuelske-le-redig04}
C.~K\"ulske, A.~Le~Ny, and F.~Redig.
\newblock Relative entropy and variational properties of generalized {G}ibbsian
  measures.
\newblock {\em Ann. Probab.}, 32:1691--1726, 2004.

\bibitem{kuelske-opoku08}
C.~K\"ulske and A.A. Opoku.
\newblock The {P}osterior metric and the {G}oodness of {G}ibbsianness for
  transforms of {G}ibbs measures.
\newblock {\em Electron. J. Probab.}, 13:1307--1344, 2008.

\bibitem{widom-rowlinson70}
J.S. Rowlinson and B.~Widom.
\newblock New {M}odel for the {S}tudy of {L}iquid-{V}apor {P}hase
  {T}ransitions.
\newblock {\em J. Chem. Phys.}, 52:1670--1684, 1970.

\bibitem{rozikov-book}
U.~Rozikov.
\newblock {\em Gibbs measures on {Cayley} trees}.
\newblock World {S}cientific {P}ublishing {C}o, 2013.

\end{thebibliography}
\end{document}